\documentclass[11pt]{article}
\usepackage{epsfig,mst-stylefile,amssymb,amsmath,harvard}
\usepackage[svgnames]{xcolor}
\usepackage{multirow}
\usepackage{tabularx,caption}
\usepackage{array}
\usepackage{booktabs}
\usepackage{xr}

\usepackage{caption}
\usepackage[normalem]{ulem}

\newcommand{\bbR}{{\Bbb R}}
\renewcommand{\R}{{\Bbb R}}
\newcommand{\bbN}{{\Bbb N}}
\newcommand{\bbC}{{\Bbb C}}

\newcommand{\bbP}{{\Bbb P}}
\newcommand{\bbQ}{{\Bbb Q}}

\newcommand{\bbE}{{\Bbb E}}

\newcommand{\E}{{\Bbb E}}
\newcommand{\C}{{\Bbb C}}

\newcommand{\N}{{\mathbb N}}

\usepackage[normalem]{ulem}

\newcommand{\imag}{{\mathbf i}}

\newcommand{\Cov}{\textnormal{Cov}}

\newcommand{\eig}{\textnormal{eig}}

\DeclareMathOperator{\tr}{tr}

\usepackage[svgnames]{xcolor}

\graphicspath{{figures/}}
\renewcommand{\cite}{\citeyear}

\begin{document}

\title{On operator fractional L\'{e}vy motion: integral representations and time reversibility
}

\author{
B.\ Cooper Boniece\\ Department of Mathematics and Statistics \\ Washington University in St.\ Louis  \and Gustavo Didier\thanks{Corresponding author.  Email: \texttt{gdidier@tulane.edu}} \thanks{The second author's long term visits to ENS de Lyon were supported by the school, the CNRS, and the Carol Lavin Bernick faculty grant. The authors are grateful to two anonymous reviewers for the constructive comments that led to a significantly improved manuscript.} \\ Mathematics Department\\ Tulane University}
\date{\vspace{-2ex}}

\maketitle

\begin{abstract}
In this paper, we construct operator fractional L\'{e}vy motion (ofLm), a broad class of non-Gaussian stochastic processes that are covariance operator self-similar, have wide-sense stationary increments and display infinitely divisible marginal distributions. The ofLm class generalizes the univariate fractional L\'evy motion as well as the multivariate operator fractional Brownian motion (ofBm). The ofLm class can be divided into two types, namely, moving average (maofLm) and real harmonizable (rhofLm), both of which share the covariance structure of ofBm under assumptions. We show that maofLm and rhofLm admit stochastic integral representations in the time and Fourier domains, and establish their distinct small- and large-scale limiting behavior. We characterize time reversibility for ofLm through parametric conditions related to its L\'evy measure, starting from a framework for the uniqueness of finite second moment, multivariate stochastic integral representations. In particular, we show that, under non-Gaussianity, the parametric conditions for time reversibility are generally more restrictive than those for the Gaussian case (ofBm).\end{abstract}

\section{Introduction}

Let $X = \{X(t)\}_{t \in \bbR}$ be a $\bbR^p$-valued stochastic process with finite second moments. A process $X$ is called \emph{proper} if its distribution at time $t$ is not concentrated on any proper subspace of $\bbR^p$ for each $t\in \bbR\setminus\{0\}$. We say $X$ is \textit{covariance operator self-similar} (cov.o.s.s.) if its distribution is proper and its covariance function satisfies
\begin{equation}\label{e:cov_oss}
\Cov(X(cs),X(ct)) = c^{H} \Cov(X(s),X(t))c^{H^*}, \quad s,t \in \bbR, \quad c > 0,
\end{equation}
for some (Hurst) matrix $H$ whose eigenvalues have real parts lying in the interval $(0,1]$. In \eqref{e:cov_oss}, $c^H := \exp\{H \log c\} =\sum_{k\in \N}\frac{(H\log c)^k}{k!}$ and $^*$ denotes the (conjugate) transpose. In this paper, we construct \textit{operator fractional L\'{e}vy motion} (ofLm), a broad class of generally non-Gaussian stochastic processes that are cov.o.s.s., have wide-sense stationary increments (namely, the mean and covariance of the increments do not change with time) and display infinitely divisible (ID) marginal distributions. The ofLm class subsumes, among others, the univariate fractional Brownian and L\'evy motions (fBm and fLm, respectively), as well as the multivariate operator fractional Brownian motion (ofBm). The ofLm class can be divided into two types, namely, moving average (maofLm) and real harmonizable (rhofLm), both of which share the covariance structure of ofBm, under assumptions. We show that both maofLm and rhofLm admit stochastic integral representations in the time and Fourier domains, and establish their distinct small- and large-scale limiting behaviors. We characterize time reversibility for ofLm, starting from a framework for the uniqueness of finite second moment, multivariate stochastic integral representations with respect to ID random measures. In particular, we show that, under non-Gaussianity, the parametric conditions for time reversibility are more restrictive than those arising in the Gaussian case (ofBm) when the models are comparable.

The concept of self-similarity provides a mathematical underpinning for the modeling of \textit{scale invariance} in a wide range of natural and social systems such as in critical phenomena (Sornette \cite{sornette:2006}), dendrochronology (Bai and Taqqu~\cite{bai:taqqu:2018}), stock market prices (Willinger et al.\ \cite{willinger:taqqu:teverovsky:1999}) and turbulence (Kolmogorov~\cite{Kolmogorovturbulence}). A univariate stochastic process $X$ is called \textit{self-similar} (s.s.) if it exhibits the scaling property
\begin{equation}\label{e:def_ss}
\{X(ct)\}_{t\in\R} \stackrel {\textnormal{f.d.d.}}{=} \{c^H X(t)\}_{t\in\R}, \quad c>0,
\end{equation}
for some scalar parameter $H\in (0,1]$, where $\stackrel{\textnormal{f.d.d.}}{=}$ denotes the equality of finite-dimensional distributions. An example of a s.s.\ process is the celebrated fBm (Mandelbrot and Van Ness \cite{mandelbrot:vanness:1968}, Embrechts and Maejima \cite{embrechts:maejima:2002}, Pipiras and Taqqu \cite{pipiras:taqqu:2017}).

On the other hand, new technological developments have ushered in the modern era of ``Big Data" (Brody \cite{brody:2011}). Many systems nowadays are monitored by several low-cost sensors and recording devices, leading to the storage of hundreds to several tens of thousands of time series. In \textit{multivariate} or \textit{high-dimensional} data, scaling behavior does not always appear along standard coordinate axes, and often involves multiple scaling relations. This situation is encountered in many applications such as in climate studies (Isotta et al.\ \cite{isotta:etal:2014}), hydrology (Benson et al.\ \cite{benson:baeumer:scheffler:2006}), finance (Meerschaert and Scalas \cite{meerschaert:scalas:2006}), neuroscience (Ciuciu et al.\ \cite{ciuciu:varoquaux:abry:sadaghiani:kleinschmidt:2012}) and network traffic (Abry and Didier \cite{abry:didier:2018:n-variate}).

A multivariate stochastic process $X$ is called \textit{operator self-similar} (o.s.s.) if it satisfies relation \eqref{e:def_ss} for some Hurst matrix $H$ whose eigenvalues have real parts lying in the interval $(0,1]$ (Laha and Rohatgi \cite{laha:rohatgi:1981}, Hudson and Mason \cite{hudson:mason:1982}). A canonical model for multivariate fractional systems is ofBm, namely, a Gaussian, o.s.s., stationary-increment stochastic process (Maejima and Mason \cite{maejima:mason:1994}, Mason and Xiao \cite{mason:xiao:2002}, Didier and Pipiras \cite{didier:pipiras:2012}). However, \textit{non-Gaussian} behavior is pervasive in a myriad of natural phenomena and artificial systems. This includes features such as burstiness or heavy tails (Leland et al.\ \cite{leland:taqqu:willinger:wilson:1993}, Paxson and Floyd \cite{paxson:floyd:1995}, Willinger et al.\ \cite{willinger:govindan:jamin:paxson:shenker:2002}, Boniece et al.\ \cite{boniece:didier:sabzikar:2020}). Among non-Gaussian scale invariant constructs, the mathematical generality and richness of fractional L\'{e}vy-type processes such as fLm have inspired a large body of work (Brockwell and Marquardt \cite{brockwell:marquart:2005}, Marquardt \cite{marquardt:2006}, Lacaux and Loubes \cite{lacaux:loubes:2007}, Bender and Marquardt \cite{bender:marquardt:2008}, Basse and Pedersen \cite{basse:pedersen:2009}, Tikanm{\"a}ki and Mishura~\cite{tikanmaki:mishura:2011}). Fractional L\'{e}vy-type processes have also become popular in physical applications since they provide a broad family of second order models displaying fractional covariance structure (Barndorff-Nielsen and Schmiegel \cite{barndorff-nielsen:schmiegel:2008}, Suciu \cite{suciu:2010}, Magdziarz and Weron \cite{magdziarz:weron:2011}, Zhang et al.\ \cite{zhang:li:zhang:2015}, Xu et al.\ \cite{xu:li:zhang:li:kurths:2016}). While of great importance in applications, the theory of their \textit{multivariate} counterparts is a topic that has been relatively little explored in the literature (e.g., Marquardt \cite{marquardt:2007}, Barndorff-Nielsen and Stelzer \cite{barndorff-nielsen:stelzer:2011}, Moser and Stelzer \cite{moser:stelzer:2013}).

In this paper, we mathematically construct a broad class of (multivariate) cov.o.s.s., wide-sense stationary-increment, stochastic processes with ID marginal distributions called \textit{operator fractional L\'{e}vy motion} (ofLm). It comprises two subclasses of stochastic processes, framed in the time and frequency (Fourier) domains. In the latter, real harmonizable ofLm (rhofLm) is defined by means of a stochastic integral of the form
\begin{equation}\label{e:ofLm_harm}
\{\widetilde{X}_H(t)\}_{t \in \bbR} \stackrel{\textnormal{f.d.d.}}= \Big\{\int_{\bbR} \Big(\frac{e^{\imag t x}-1}{\imag x} \Big) \big\{x^{H-(1/2)I}_+ A + x^{H-(1/2)I}_- \overline{A}\big\} \widetilde{{\mathcal M}}(dx) \Big\}_{t \in \bbR}
\end{equation}
for some complex matrix $A$, where $\widetilde{{\mathcal M}}(dx)$ is a $\bbC^p$-valued ID random measure. In the time domain, under mild constraints, moving average ofLm (maofLm) admits the stochastic integral representation
$$
\{X_H(t)\}_{t \in \bbR} \stackrel{\textnormal{f.d.d.}}= \Big\{\int_{\bbR} \big[\big\{(t-s)^{H-(1/2)I}_+ - (-s)^{H-(1/2)I}_+\big\} M_+
$$
\begin{equation}\label{e:ofLm_MA}
+ \big\{(t-s)^{H-(1/2)I}_- - (-s)^{H-(1/2)I}_-\big\} M_- \big] {\mathcal M}(ds) \Big\}_{t \in \bbR}
\end{equation}
for real matrices $M_+$, $M_-$, where ${\mathcal M}(ds)$ is a $\bbR^p$-valued ID random measure. In particular, when the random measures are Gaussian, \eqref{e:ofLm_harm} and \eqref{e:ofLm_MA} provide representations of the same stochastic process, namely, ofBm (Didier and Pipiras \cite{didier:pipiras:2011}). The random measures can be induced by multivariate L\'{e}vy processes (independent and stationary increment processes), in which case they generalize Cram\'{e}r-Wold representations based on Brownian noise (e.g., Doob \cite{doob:1953}, Rozanov \cite{rozanov:1967}).

OfLm was first considered as a model, without proofs, in Boniece, Didier et al.\ \cite{boniece:didier:wendt:abry:2019:eusipco} and Boniece, Wendt et al.\ \cite{boniece:wendt:didier:abry:2019:camsap}. In this paper, we broadly define ofLm and mathematically establish its fundamental properties such as finite-dimensional distributions and sample path behavior (Theorem \ref{t:rhofLm_maofLm_integ_repres_chf}). In particular, ofLm provides a flexible theoretical framework for the study of the effects of departures from non-Gaussianity in multivariate fractional constructs while keeping finite second moments. This can be seen, for instance, in natural alternative stochastic integral representations in the time and Fourier domains (Proposition \ref{p:XH_harm_YH_ma}; cf.\ Marquardt and Stelzer \cite{marquardt:stelzer:2007} on CARMA processes). Moreover, the study of scaling behavior lays bare some of the striking differences from the Gaussian case (cf.\ Benassi et al.\ \cite{benassi:cohen:istas:2002,benassi:cohen:istas:2004} on scalar random fields). On the one hand, non-Gaussian ofLm is shown to never be o.s.s. On the other hand, rhofLm and maofLm approach ofBm at short and long time scales, respectively (see Proposition \ref{p:maofLm->BH(t)_rhofLm->BH(t)}). In addition, for certain choices of ID random measure (L\'{e}vy noise), rhofLm and maofLm approach o.s.s., operator-stable processes at long and short time scales, respectively (see Proposition \ref{p:lass2}).

Recall that a stochastic process $X = \{X(t)\}_{t \in \bbR}$ is said to be \textit{time-reversible} if
\begin{equation}\label{e:X_is_time_revers}
\{X(t)\}_{t \in \bbR}\stackrel{\textnormal{f.d.d.}}=\{X(-t)\}_{t \in \bbR}.
\end{equation}
Equivalently, $\{\pm 1\}$ are domain symmetries of $X$ (Didier et al.\ \cite{didier:meerschaert:pipiras:2018}). All univariate, Gaussian stationary or stationary-increment stochastic processes are time-reversible. More generally, in the univariate context, confirmation of time irreversibility is relevant in both theory and modeling because it can be viewed, for example, as evidence of either non-Gaussianity or nonlinearity (see Weiss \cite{weiss:1975}, Cox \cite{cox:1981}, Section 3, Cheng \cite{cheng:1999}, and De Gooijer \cite{de_gooijer:2017}, p.\ 315; see also Jacod and Protter \cite{jacod:protter:1988}, Cox \cite{cox:1991} and Rosenblatt \cite{rosenblatt:2000}, chapter 1). In particular, time reversibility is well known to be a topic of central importance in Physics (e.g., Ku{\'s}mierz et al.\ \cite{kusmierz:chechkin:gudowska-nowak:bier:2016}). For an ofBm $B_H$ -- a multivariate, stationary-increment Gaussian process --, time reversibility is equivalent to the availability of the classical and convenient fBm-like covariance formula
\begin{equation}\label{e:time_revers_ofBm}
\bbE B_{H}(s)B_{H}(t)^* = \frac{1}{2} \big\{|s|^{H}\Sigma|s|^{H^*}+|t|^{H}\Sigma|t|^{H^*} - |t-s|^{H}\Sigma|t-s|^{H^*}\big\}, \quad s,t \in \bbR,
\end{equation}
where $\Sigma = \bbE B_{H}(1)B_{H}(1)^*$ (Didier and Pipiras \cite{didier:pipiras:2011}, Proposition 5.2). In this paper, we provide parametric characterizations of time reversibility for maofLm and rhofLm (Theorems \ref{t:maofLm_time-reversibility} and \ref{t:rhofLm_time-revers_general}). In particular, the results show that, under regularity assumptions, time reversibility for ofLm requires parametric conditions that are strictly stronger than those for ofBm (see Examples \ref{ex:time_revers_ofBm_vs_maofLm} and \ref{ex:time_revers_ofBm_vs_rhofLm}). Characterizing time reversibility involves starting from expressions of the form \eqref{e:ofLm_harm} and \eqref{e:ofLm_MA} and arriving at statements about integrands. In turn, this calls for results on the uniqueness of ID stochastic integrals that replace classical covariance Fourier inversion-type results for the Gaussian case (as in Didier and Pipiras \cite{didier:pipiras:2011}). For this purpose, we draw upon the seminal work of Kabluchko and Stoev \cite{kabluchko:stoev:2016} (see also Maruyama \cite{maruyama:1970}, Rajput and Rosi\'{n}ski \cite{rajput:rosinski:1989}, Rosi\'{n}ski \cite{rosinski:1989}) to analyze the uniqueness of finite second moment, multivariate stochastic integral representations with respect to compensated Poisson random measures.

The paper is organized as follows. In Section \ref{s:preliminaries}, we lay out a mathematical setting for multivariate stochastic integrals with respect to finite second moment, compensated Poisson random measures in both time and Fourier domains. In Section \ref{s:fLm}, we use the framework of Section \ref{s:preliminaries} to construct rhofLm and maofLm, and establish their essential distributional, sample path and scaling properties. In Section \ref{s:time_revers}, we characterize time reversibility for maofLm and rhofLm. All proofs, as well as auxiliary concepts and results, can be found in the Appendix.

\section{Preliminaries}\label{s:preliminaries}

Let $M(p,q,\bbR)$ and $M(p,q,\bbC)$ be, respectively, the spaces of $\bbR$-- and $\bbC$--valued $p \times q$ matrices, $p,q \in \bbN$, and let $M(p,\bbR) = M(p,p,\bbR)$ and $M(p,\bbC) = M(p,p,\bbC)$. Also, let $GL(p,\bbR)$ and $GL(p,\bbC)$ denote the corresponding groups of nonsingular matrices on the fields $\bbR$ and $\bbC$, respectively. For $M \in M(p,\bbC)$, $\textnormal{eig}(M)$ denotes the set of possibly repeated eigenvalues (characteristic roots) of $M$, and $\Re \hspace{1mm}\textnormal{eig}(M)$ denotes the set of their (possibly repeated) real parts.  Whenever convenient, given $M\in M(p,\C)$, we write $\lambda_i(M)$, $i=1,\ldots, p$ for the (possibly repeated) eigenvalues of $M$, indexed by the ordering $\Re \lambda_1(M)\leq\ldots\leq \Re \lambda_p(M)$. The symbol $I$ denotes the identity matrix, and $\textnormal{diag}(d_1,\hdots,d_p)$ represents a diagonal matrix with main diagonal entries $d_1,\hdots,d_p \in \bbC$. The symbol $\|\cdot\|$ denotes the Euclidean norm of a vector or the corresponding operator norm for a matrix. In the latter case, for a square matrix $M$, $\|M\|^2$ is given by the largest eigenvalue of $M^*M$ or $MM^*$. %

\subsection{Stochastic integrals}\label{s:stoch_integrals}

In this section, we use compensated Poisson random measures associated with finite second moment L\'{e}vy measures to describe a framework for stochastic integration. This framework provides a multivariate generalization of the ones in Benassi et al.\ \cite{benassi:cohen:istas:2002,benassi:cohen:istas:2004} and Marquardt \cite{marquardt:2006} (see also Marquardt \cite{marquardt:2007}). The ultimate goal is to construct moving average and harmonizable classes of fractional stochastic processes (Section \ref{s:fLm}), so we consider stochastic integration in both frequency (Fourier) and time domains. In this section, we provide the definitions and expressions that are essential in the construction of ofLm. The Poisson random measures considered herein can be viewed as stemming from the jump measure of a L\'evy process (see, e.g., Sato \cite{sato:1999}, Chapter 4). More properties of stochastic integrals can be found in Section \ref{s:properties_stoch_integrals}.

In the proposed framework, the differences between integration in the Fourier and time domains lie in the Poisson random measure domain ($\bbC^p$ or $\bbR^p$, respectively) and in the classes of integrands considered. In the former case, we mainly consider Hermitian integrands, as to ensure $\bbR^p$--valued stochastic integrals.

We first consider the Fourier domain. So, let $\mu_{\bbC^p}(d{\mathbf z}) \equiv \mu(d{\mathbf z})$ be a L\'{e}vy measure on ${\mathcal B}(\bbC^p)$ satisfying
\begin{equation}\label{e:int_|z|^2_mu(dz)<infty}
\int_{\bbC^p} {\mathbf z}^* {\mathbf z} \hspace{1mm}\mu(d{\mathbf z}) < \infty, \quad \mu(\{{\mathbf 0}\}) = 0.
\end{equation}

\begin{example}
Let $\mu(d{\mathbf z})=c\nu(d{\mathbf z})$, where $c>0$ and $\nu(d{\mathbf z})$ is any probability measure on $\C^p$ with finite second moments satisfying $\nu(\{\mathbf 0\}) = 0$. Then, \eqref{e:int_|z|^2_mu(dz)<infty} is satisfied. %
\end{example}

\begin{example}
Let $\alpha\in(0,2)$. For some $c>0$, define $\mu(d\mathbf z)= e^{-c\|\mathbf z\|}\big/\|\mathbf z\|^{1+\alpha} d\mathbf z$.
Then, \eqref{e:int_|z|^2_mu(dz)<infty} is satisfied (the measure $\mu$ is an instance of a \textit{tempered stable distribution}; see Rosi\'{n}ski \cite{rosinski:2007} or Grabchak \cite{grabchak:2016}).
\end{example}
Now, for $\mu(d{\mathbf z})$ as in \eqref{e:int_|z|^2_mu(dz)<infty}, let
\begin{equation}\label{e:Ntilde(dxi,dz)}
\widetilde{N}(dx,d{\mathbf z}) = N(dx,d{\mathbf z})- \bbE N(dx,d{\mathbf z})= N(dx,d{\mathbf z}) -  \hspace{0.5mm}dx \mu(d{\mathbf z}) \in \bbR
\end{equation}
be a compensated Poisson random measure on ${\mathcal B}(\bbR \times \bbC^p)$ (see, for example, Sato \cite{sato:1999}, Section 19, or Applebaum \cite{applebaum:2009}, Section 2.3). %
We define the space of integration kernels
$$
{\mathcal L}^2_{dx\otimes \mu(d{\mathbf z})}\equiv {\mathcal L}^2_{dx\otimes \mu_{\bbC^p}(d{\mathbf z})}\equiv {\mathcal L}^2\Big(\bbR \times \bbC^p, {\mathcal B}(\bbR \times \bbC^p),dx\otimes \mu_{\bbC^p}(d{\mathbf z})\Big)
$$
$$
= \Big\{\text{measurable }\varphi:\bbR \times \bbC^p \rightarrow \bbC^p :  \|\varphi\|_{{\mathcal L}^2_{dx\otimes \mu_{\bbC^p}(d{\mathbf z})}}   < \infty\Big\},
$$
where
\begin{equation}\label{e:|varphi|_{L^2}}
\|\varphi\|^2_{{\mathcal L}^2_{dx\otimes \mu_{\bbC^p}(d{\mathbf z})}} \equiv \|\varphi\|^2_{{\mathcal L}^2_{dx\otimes \mu(d{\mathbf z})}} := \int_{\bbR}\int_{\bbC^p} \varphi(x,{\mathbf z})^* \varphi(x,{\mathbf z}) \hspace{1mm}\mu(d{\mathbf z})dx.
\end{equation}
Fix the sets $B_{1,i} \times B_{2,i} \in {\mathcal B}(\bbR) \otimes {\mathcal B}(\bbC^p)$, as well as the vectors $\varphi_i \in \bbC^p$, $i = 1,\hdots, I$. Consider the elementary function $\varphi(x,{\mathbf z}) = \sum^{I}_{i=1} \varphi_{i}1_{B_{1,i} \times B_{2,i}}(x,{\mathbf z})$. We define the stochastic integral of the elementary function $\varphi$ with respect to the random measure $\widetilde{N}$ by means of the expression
$$
\int_{\bbR \times \bbC^p} \varphi(x,{\mathbf z}) \widetilde{N}(dx,d{\mathbf z}) := \sum^{I}_{i=1} \varphi_{i} \widetilde{N}(B_{1,i},B_{2,i}).
$$
Next, fix $\varphi \in {\mathcal L}^2_{dx\otimes \mu(d{\mathbf z})}$, and let $\{\varphi_{n}\}_{n \in \bbN} \subseteq {\mathcal L}^2_{dx\otimes \mu(d{\mathbf z})}$ be elementary functions converging to $\varphi$ in the norm $\|\cdot\|_{{\mathcal L}^2_{dx\otimes \mu(d{\mathbf z})}}$. Then,
\begin{equation}\label{e:int_varphi_N}
\bbC^p \ni \int_{\bbR \times \bbC^p} \varphi(x,{\mathbf z}) \widetilde{N}(dx,d{\mathbf z}) = L^2(\bbP)\textnormal{--}\lim_{n \rightarrow \infty} \int_{\bbR \times \bbC^p}\varphi_{n}(x,{\mathbf z})\widetilde{N}(dx,d{\mathbf z})
\end{equation}
is well defined as the \textit{stochastic integral of the function $\varphi$ with respect to the compensated Poisson random measure $\widetilde{N}$}  (see Section \ref{s:properties_stoch_integrals} for details). In particular, the limit random vector does not depend on the chosen sequence of elementary functions. Now let
\begin{equation}\label{e:L^2(R,M(p,C))}
\widetilde{g} \in L^2\big(\bbR,M(p,\bbC)\big)= \Big\{ \text{measurable } {\mathfrak g}:\bbR \rightarrow M(p,\bbC):  \int_{\bbR} \tr\big( {\mathfrak g}(x) {\mathfrak g}(x)^*  \big) dx < \infty \Big\}.
\end{equation}
We also define the random measure $\widetilde{{\mathcal M}}(dx)$ on ${\mathcal B}(\bbR)$ by means of the relation
\begin{equation}\label{e:int_f(xi)M(dxi)_def}
\int_{\bbR} \widetilde{g}(x) \widetilde{{\mathcal M}}(dx) := \int_{\bbR \times \bbC^p} \{\widetilde{g}(x){\mathbf z} + \widetilde{g}(-x)\overline{{\mathbf z}}\}\widetilde{N}(dx,d{\mathbf z}) \in \bbC^p.
\end{equation}
In particular, for
\begin{equation}\label{e:L^2_Herm}
\widetilde{g} \in L^2_{\textnormal{Herm}}(\bbR) = \Big\{{\mathfrak g} \in L^2\big(\bbR,M(p,\bbC)\big): \hspace{1mm} {\mathfrak g}(-x) = \overline{{\mathfrak g}(x)}\Big\},
\end{equation}
expression \eqref{e:int_f(xi)M(dxi)_def} reduces to
\begin{equation}\label{e:int_f(xi)M(dxi)}
\int_{\bbR} \widetilde{g}(x) \widetilde{{\mathcal M}}(dx) = \int_{\bbR \times \bbC^p} 2 \Re(\widetilde{g}(x){\mathbf z}) \widetilde{N}(dx,d{\mathbf z}) = 2 \Re\Big( \int_{\bbR \times \bbC^p} \widetilde{g}(x) {\mathbf z}  \widetilde{N}(dx,d{\mathbf z}) \Big)  \in \bbR^p.
\end{equation}

So, let $\{\widetilde{g}_t\}_{t \in \bbR} \subseteq L^2_{\textnormal{Herm}}(\bbR)$. We can define the stochastic process $\widetilde{X} =\{\widetilde{X}(t)\}_{t \in \bbR}$ by means of the stochastic integral
\begin{equation}\label{e:Xtilde(t)=int_ft(omega)M(domega)}
\widetilde{X}(t) = \int_{\bbR}\widetilde{g}_{t}(x)\widetilde{{\mathcal M}}(dx), \quad t \in \bbR.
\end{equation}
Equivalently, based on relation \eqref{e:int_f(xi)M(dxi)}, we can reexpress $\widetilde{X}$ as
\begin{equation}\label{e:X-tilde(t)}
\widetilde{X}(t) = \int_{\bbR \times \bbC^p} 2 \Re(\widetilde{g}_{t}(x){\mathbf z})\widetilde{N}(dx,d{\mathbf z})
= 2 \Re\Big(\int_{\bbR \times \bbC^p} \widetilde{g}_{t}(x){\mathbf z}\widetilde{N}(dx,d{\mathbf z})  \Big).
\end{equation}

To construct the analogous time domain framework, we start with the following definition. As in \eqref{e:Ntilde(dxi,dz)}, we consider the compensated Poisson random measure
\begin{equation}\label{e:N(ds,dz)_real_z}
\widetilde{N}(ds,d{\mathbf z}) = N(ds,d{\mathbf z})- \bbE N(ds,d{\mathbf z})= N(ds,d{\mathbf z}) -   ds\mu(d{\mathbf z}) \in \bbR
\end{equation}
on ${\mathcal B}(\bbR^{p+1})$, where $\mu(d{\mathbf z})$ is a L\'{e}vy measure {on ${\mathcal B}(\bbR^p)$ and satisfying
\begin{equation}\label{e:int_|z|^2_mu(dz)<infty_in_R}
\int_{\bbR^p} {\mathbf z}^* {\mathbf z} \hspace{1mm}\mu(d{\mathbf z}) < \infty, \quad \mu(\{{\mathbf 0}\}) = 0.
\end{equation}
We naturally define the space of integration kernels ${\mathcal L}^2_{ds\otimes \mu(d{\mathbf z})}$ as in \eqref{e:|varphi|_{L^2}}, where $\bbR^{p+1}$ replaces $\bbR \times \bbC^p$, $\mu(d{\mathbf z})$ is a L\'{e}vy measure on ${\mathcal B}(\bbR^{p})$, and $\|\varphi\|^2_{{\mathcal L}^2_{ds\otimes \mu(d{\mathbf z})}}$ is defined as in \eqref{e:|varphi|_{L^2}} with $\bbR^{p+1}$ replacing $\bbR \times \bbC^p$. Let $\{\varphi_{n}\}_{n \in \bbN} \subseteq {\mathcal L}^2_{ds\otimes \mu(d{\mathbf z})}$ be elementary functions converging to $\varphi$ in the norm $\|\cdot\|_{{\mathcal L}^2_{ds\otimes \mu(d{\mathbf z})}}$. The stochastic integral
\begin{equation}\label{e:int_varphi_N_time}
\int_{\bbR^{p+1}} \varphi(s,{\mathbf z}) \widetilde{N}(ds,d{\mathbf z}) \in \bbR^p, \quad \varphi \in {\mathcal L}^2_{ds \otimes \mu(d{\mathbf z})},
\end{equation}
is then naturally defined as in \eqref{e:int_varphi_N}.

We further define the random measure ${\mathcal M}(ds)$ by means of the relation
\begin{equation}\label{e:int_f(s)M(ds)}
\int_{\bbR} g(s) {\mathcal M}(ds) := \int_{\bbR\times \R^p} g(s){\mathbf z}  \widetilde{N}(ds,d{\mathbf z})  \in \bbR^p,
\end{equation}
where
\begin{equation}\label{e:L^2(R,M(p,R))}
g \in L^2\big(\bbR,M(p,\bbR)\big)= \Big\{ \text{measurable } {\mathfrak g}:\bbR \rightarrow M(p,\bbR):  \int_{\bbR} \tr\big( {\mathfrak g}(s) {\mathfrak g}(s)^* \big) ds  < \infty \Big\}
\end{equation}
(in particular, $g(s){\mathbf z} \in {\mathcal L}^2_{ds \otimes \mu(d{\mathbf z})}$). Note that the space of integrands for the random measure ${\mathcal M}(ds)$ (i.e., \eqref{e:L^2(R,M(p,R))}) is different from that for $\widetilde{{\mathcal M}}(d x)$ (i.e., \eqref{e:L^2(R,M(p,C))}).

Analogously to expression \eqref{e:Xtilde(t)=int_ft(omega)M(domega)}, for $\{g_t(s)\}_{t \in \bbR} \subseteq L^2\big(\bbR,M(p,\bbR)\big)$, we can define the stochastic process $X = \{X(t)\}_{t \in \bbR}$ by means of the stochastic integral
\begin{equation*}\label{e:Y=int_f(t,s)L(ds)}
X(t) = \int_{\bbR} g_t(s) {\mathcal M}(ds).
\end{equation*}
Equivalently, based on relation \eqref{e:int_f(s)M(ds)}, we can reexpress $X$ as
\begin{equation}\label{e:Y=int_f(t,s)L(ds)_change}
\{X(t)\}_{t \in \bbR} \stackrel{\textnormal{f.d.d.}}= \Big\{\int_{\bbR\times \bbR^p} g_t(s) {\mathbf z} \widetilde{N}(ds,d{\mathbf z})\Big\}_{t \in \bbR}.
\end{equation}

\subsection{On integral representations of operator fractional Brownian motion}\label{s:ofBm_integ_repres}

Recall that an ofBm is a Gaussian, o.s.s., stationary-increment stochastic process. Harmonizable representations are the natural starting point for the study of ofBm. This is so because, as briefly recalled in the Introduction, almost every instance of ofBm admits the representation \eqref{e:ofLm_harm}, where $\widetilde{{\mathcal M}}(dx)=\widetilde B(dx)$ is a $p$-variate complex Gaussian random measure satisfying $\widetilde B(-dx)= \overline{\widetilde B(dx)}$ a.s. and $\E\widetilde B(dx)\widetilde B(dx)^* = dx \times I$ (see Example \ref{ex:b-tilde}). So, for notational simplicity, define
\begin{equation}\label{e:D=H-(1/2)I}
D = H - (1/2) I.
\end{equation}
Let
\begin{equation}\label{e:fLm_filter_harm}
L^2_{\textnormal{Herm}}(\bbR) \ni \widetilde{g}_t(x) = \frac{e^{\imag t x}-1}{\imag x} \hspace{0.5mm}\big\{x^{-D}_{+}A  + x^{-D}_{-}\overline{A} \big\}, \quad x \neq 0,
\end{equation}
be the integrand of the harmonizable representation of ofBm. Defining Fourier transforms entry-wise, for
\begin{equation}\label{e:ft(s)=F^(-)(gt(s))}
g_t(s) := {\mathcal F}^{-1}(\widetilde{g}_t)(s) \in L^2\big(\bbR,M(p,\bbR) \big),
\end{equation}
ofBm also admits a moving average representation of the form \eqref{e:ofLm_MA}, where the Gaussian random measure ${\mathcal M}(ds) = B(ds)$ satisfies $\bbE B(ds)B(ds)^* = ds \times I$. For most cases of interest, we can explicitly recast the moving average representation of ofBm. In fact, we can set
\begin{equation}\label{e:fLm_filter}
g_t(s) = \left\{\begin{array}{cc}
\big\{(t-s)^{D}_+ - (-s)^{D}_+ \big\}M_+  + \big\{(t-s)^{D}_- - (-s)^{D}_- \big\}M_-, & \textnormal{if }\Re\hspace{0.5mm}\textnormal{eig}(H) \subseteq (0,1)\backslash\{1/2\};\\
 & \\
\{\textnormal{sign}(t-s) - \textnormal{sign}(-s)\big\}M + \log \Big( \frac{|t-s|}{|s|}\Big)N, & \textnormal{if }H = (1/2)I,
\end{array}\right.
\end{equation}
for some matrix constants $M_+, M_- \in M(p,\bbR)$ or $M, N \in M(p,\bbR)$ (Didier and Pipiras \cite{didier:pipiras:2011}, Theorem 3.2). For the instances
\begin{equation}\label{e:Re_eig(H)=(0,1)-(1/2)}
\Re\hspace{0.5mm}\textnormal{eig}(H) \subseteq (0,1)\backslash\{1/2\},
\end{equation}
expression \eqref{e:fLm_filter} can be extracted based on the fact that
$$
{\mathcal F}\big( (t-s)^{D}_{\pm}-(-s)^{D}_{\pm} \big)(x) := \int_{\bbR}e^{\imag s x}\big\{ (t-s)^{D}_{\pm}-(-s)^{D}_{\pm} \big\} ds
$$
\begin{equation}\label{e:Prop3.1_DP2011}
= \frac{e^{\imag t x}-1}{\imag x} |x|^{-D}\Gamma\big(D+I\big) e^{\mp \textnormal{sign}(x) \imag \pi D/2}
\end{equation}
(see Proposition 3.1 and Theorem 3.2 in Didier and Pipiras \cite{didier:pipiras:2011}, in particular, expressions (3.20), (3.24) and (3.25)). In \eqref{e:Prop3.1_DP2011}, $\Gamma\big(D+I\big)$ is interpreted as a primary matrix function (Horn and Johnson \cite{horn:johnson:1991}, Sections 6.1 and 6.2). Further note that, when
$$
\textnormal{eig}(H) \cap \big\{ z \in \bbC: \Re z = 1/2 \big\} \neq \emptyset \quad \textnormal{and} \quad H \neq (1/2)I,
$$
moving average representations can be quite intricate (see Example 3.1 in Didier and Pipiras \cite{didier:pipiras:2011}).

\section{Operator fractional L\'{e}vy motion}\label{s:fLm}

We are now in a position to define the ofLm class. For the sake of simplicity, hereinafter we focus on purely non-Gaussian constructs.
We first define ofLm in the Fourier and time domains, and then establish its fundamental properties.
\begin{definition}\label{def:rhofLm}
Let $H  \in M(p,\bbR)$ be a (Hurst) matrix whose eigenvalues satisfy
\begin{equation}\label{e:0<Re(h)<1}
\Re \hspace{0.5mm}\textnormal{eig}(H) \subseteq (0,1).
\end{equation}
Let $\widetilde{{\mathcal M}}(ds)$ be a $\bbC^p$-valued random measure as in \eqref{e:int_f(xi)M(dxi)} whose L\'{e}vy measure $\mu(d{\mathbf z})\equiv \mu_{\bbC^p}(d{\mathbf z})$ satisfies condition \eqref{e:int_|z|^2_mu(dz)<infty}. A \textit{real harmonizable operator fractional L\'evy motion} (rhofLm) without Gaussian component $\widetilde{X}_H = \{\widetilde{X}_H(t)\}_{t \in \bbR}$ is a stochastic process such that
\begin{itemize}
\item [$(i)$]
\begin{equation}\label{e:EX2-tildeH(t)>0}
\textnormal{the distribution of $\widetilde{X}_H(t)$ is proper}, \quad t \neq 0;
\end{equation}
\item [$(ii)$] it satisfies the relation
\begin{equation}\label{e:X-tildeH(t)}
\{\widetilde{X}_H(t)\}_{t \in \bbR} = \Big\{ \int_{\bbR}\widetilde{g}_t(x) \widetilde{{\mathcal M}}(dx) \Big\}_{t \in \bbR}.
\end{equation}
\end{itemize}
In \eqref{e:X-tildeH(t)}, the integrand is given by $\widetilde{g}_t(x)$ as in \eqref{e:fLm_filter_harm} for some matrix constant $A \in M(p,\bbC)$.
\end{definition}

We can now turn to the time domain.
\begin{definition}\label{def:maofLm}
Let $H  \in M(p,\bbR)$ be a (Hurst) matrix whose eigenvalues satisfy \eqref{e:0<Re(h)<1}. Also let ${\mathcal M}(ds)$ be a $\bbR^p$-valued random measure as in \eqref{e:int_f(s)M(ds)} whose L\'{e}vy measure $\mu(d{\mathbf z})$ satisfies \eqref{e:int_|z|^2_mu(dz)<infty_in_R}. A \textit{moving average operator fractional L\'evy motion} (maofLm) without Gaussian component $X_H = \{X_H(t)\}_{t \in \bbR}$ is a $\bbR^p$-valued stochastic process such that
\begin{itemize}
\item [$(i)$]
\begin{equation}\label{e:EX2H(t)>0_multivar}
\textnormal{the distribution of $X_H(t)$ is proper}, \quad t \neq 0;
\end{equation}
\item [$(ii)$] it satisfies the relation
\begin{equation}\label{e:XH(t)}
\{X_H(t)\}_{t \in \bbR} \stackrel{\textnormal{f.d.d.}}= \Big\{\int_{\bbR} g_t(s) \hspace{0.5mm}{\mathcal M}(ds)\Big\}_{t \in \bbR}.
\end{equation}
\end{itemize}
In \eqref{e:XH(t)}, the integrand is given by $g_t(s)$ as in \eqref{e:ft(s)=F^(-)(gt(s))}.
\end{definition}

\begin{remark} When $p=1$, maofLm reduces to the classical fLm (e.g., Marquardt \cite{marquardt:2006}).
\end{remark}

\begin{example}\label{ex:well-balanced_maofLm}
From \eqref{e:XH(t)} and \eqref{e:fLm_filter}, when $\Re \eig(H)\subseteq (0,1)\setminus\{1/2\}$, and $M_+ = M_- =: M$, maofLm admits the well-balanced representation
\begin{equation}\label{e:mafLm_well-balanced}
\{X_H(t)\}_{t \in \bbR} = \Big\{\int_{\bbR} \{\hspace{0.5mm}|t-s|^{D} - |-s|^{D}\hspace{0.5mm}\} M \hspace{0.5mm}  {\mathcal M}(ds)\Big\}_{t \in \bbR}
\end{equation}
(cf.\ Benassi et al.\ \cite{benassi:cohen:istas:2004}, Definition 2.1).
\end{example}

In the following proposition, we establish fundamental properties of both rhofLm and maofLm. Statement $(vii)$ pertains to sample path
properties, whereas all remaining statements pertain to existence, continuity and distributional properties.
\begin{theorem}\label{t:rhofLm_maofLm_integ_repres_chf}
Let $H \in M(p,\bbR)$ be a (Hurst) matrix satisfying \eqref{e:0<Re(h)<1}. Also let $\widetilde{X}_H = \{\widetilde{X}_H(t)\}_{t \in \bbR}$ be a rhofLm as in \eqref{e:X-tildeH(t)} and let $X_H = \{X_H(t)\}_{t \in \bbR}$ be a maofLm as in \eqref{e:XH(t)}. Then,
\begin{itemize}
\item [$(i)$] for any $t \in \bbR$, $\widetilde{X}_H(t)$ and $X_H(t)$ are well defined;
\item [$(ii)$] $\widetilde{X}_H$ and $X_H$ are stochastically continuous, namely, $\widetilde{X}_H(t) \stackrel{\bbP}\rightarrow \widetilde{X}_H(t_0)$ and $X_H(t) \stackrel{\bbP}\rightarrow X_H(t_0)$ whenever $t \rightarrow t_0 \in \bbR$. In particular, they have measurable modifications;
\item [$(iii)$] for any $m \in \bbN$, any ${\mathbf u}_1,\hdots,{\mathbf u}_m \in \bbR^p$ and any $t_1 < \hdots < t_m$, the characteristic function of the finite-dimensional distributions of $\widetilde{X}_H(t)$ is given by
$$
\bbE \exp \Big\{\imag \sum^{m}_{j=1}\langle {\mathbf u}_j ,\widetilde{X}_H(t_j) \rangle \Big\}
$$
\begin{equation}\label{e:rhofLm_chf}
= \exp \Big\{ \int_{\bbR} \int_{\bbC^p} \Big(e^{\imag \hspace{0.5mm}2 \sum^{m}_{j=1}{\mathbf u}^*_j \Re (\widetilde{g}_{t_j}(x){\mathbf z})}-1 -\imag \hspace{0.5mm}2 \sum^{m}_{j=1}{\mathbf u}^*_j \Re (\widetilde{g}_{t_j}(x){\mathbf z})\Big)\mu_{\bbC^p}(d{\mathbf z}) dx\Big\}
\end{equation}
$$
= \exp\Big\{ \int_{\bbR} \int_{\bbR^{2p}}\Big[e^{\imag \hspace{0.25mm}2 \hspace{0.25mm} \sum^{m}_{j=1}{\mathbf u}^*_j  \big(\Re\widetilde{g}_{t_j}(x) {\mathbf z}_1 - \Im \widetilde{g}_{t_j}(x) {\mathbf z}_2\big)}- 1 \hspace{2cm}
$$
\begin{equation}\label{e:chf_harm_gt}
\hspace{2cm}- \imag \hspace{0.25mm}2\hspace{0.25mm} \sum^{m}_{j=1}{\mathbf u}^*_j \big(\Re\widetilde{g}_{t_j}(x) {\mathbf z}_1 - \Im \widetilde{g}_{t_j}(x) {\mathbf z}_2\big)  \Big]\mu_{\bbR^{2p}}(d{\mathbf z}) dx  \Big\},
\end{equation}
where
\begin{equation}\label{e:mu_Cp=muR2p}
\mu_{\bbC^{p}} \equiv \mu_{\bbR^{2p}} \textnormal{ and } {\mathbf z} \equiv \begin{pmatrix}{\mathbf z}_1\\{\mathbf z}_2\end{pmatrix} \in \bbR^p \times \bbR^p.
\end{equation}
Moreover, the characteristic function of the finite-dimensional distributions of $X_H(t)$ is given by
$$
\bbE \exp \Big\{\imag \sum^{m}_{j=1}\langle {\mathbf u}_j, X_H(t_j)\rangle \Big\}
$$
\begin{equation}\label{e:mafLm_chf}
= \exp \Big\{ \int_{\bbR} \int_{\bbR^p}\Big(e^{\imag \sum^{m}_{j=1}{\mathbf u}^*_j g_{t_j}(s){\mathbf z}}-1 -\imag \sum^{m}_{j=1}{\mathbf u}^*_j g_{t_j}(s){\mathbf z}  \Big)\mu(d{\mathbf z}) ds\Big\}.
\end{equation}
In particular, $\widetilde{X}_H$ and $X_H$ have mean zero and are cov.o.s.s.;
\item [$(iv)$] 
If $\int_{\R} \mathbf z \mathbf z^* \mu(d\mathbf z)$ has full rank, $X_H$ has the same covariance function as an ofBm with Hurst matrix $QHQ^{-1}$ and parameter $QA$ in its harmonizable representation, where $Q=\big(\int_{\R} \mathbf z \mathbf z^* \mu(d\mathbf z)\big)^{1/2}$.  In particular,  if
\begin{equation}\label{e:int_zz*=I}
\int_{\bbR^p} {\mathbf z} {\mathbf z}^* \mu(d{\mathbf z}) = I
\end{equation}
and condition \eqref{e:Re_eig(H)=(0,1)-(1/2)} holds, then $\bbE X_H(s)X_H(t)^*$, $s,t,\in \bbR$, is the covariance function of an ofBm whose time domain representation has parameters $H$, $M_+$ and $M_-$ (see \eqref{e:fLm_filter}). Also, if $\int_{\bbC^{p}} \Re {\mathbf z} \Re {\mathbf z}^* \mu(d{\mathbf z})=\int_{\bbC^{p}} \Im {\mathbf z} \Im {\mathbf z}^* \mu(d{\mathbf z})$ has full rank, then $\widetilde X_H$ has the same covariance function as an ofBm with Hurst matrix $\widetilde QH\widetilde Q^{-1} $ and parameter $\widetilde QA$ in its harmonizable representation, where $\widetilde Q =\big(4 \int_{\bbC^{p}} \Re {\mathbf z} \Re {\mathbf z}^* \mu(d{\mathbf z})\big)^{1/2}$.  In particular, if
\begin{equation}\label{e:int_RezRez*=sigma^2_I=int_ImzImz*}
4\int_{\bbC^p} (\Re {\mathbf z} ) (\Re {\mathbf z})^* \mu(d{\mathbf z}) = I = 4\int_{\bbC^p} (\Im {\mathbf z}) (\Im {\mathbf z})^* \mu(d{\mathbf z}),
\end{equation}
then $\bbE \widetilde{X}_H(s)\widetilde{X}_H(t)^*$, $s,t,\in \bbR$, is the covariance function of an ofBm whose harmonizable representation has parameters $H$ and $A$ (see \eqref{e:fLm_filter_harm}).
\item [$(v)$] $X_H$ has strict-sense stationary increments and $\widetilde X_H$ has wide-sense stationary increments. If
\begin{equation}\label{e:Mtilde(dxi)=e^(i*theta)Mtilde(dxi)}
\mu_{\C^p}(d\mathbf z) = \mu_{\C^p}(e^{\imag \theta}d\mathbf z), \quad \theta \in [-\pi,\pi)
\end{equation}
(i.e., $\widetilde{{\mathcal M}}(d x) \stackrel{d}= e^{ \imag h x} \widetilde{{\mathcal M}}(d x)$, $h \in \bbR$), then $\widetilde{X}_H$ also has strict-sense stationary increments;
\item [$(vi)$] let $X_H$ be a maofLm whose Hurst matrix $H$ satisfies condition \eqref{e:Re_eig(H)=(0,1)-(1/2)}.
Then, $X_H$ is not o.s.s.  Also, let $\widetilde{X}_H$ be a rhofLm. Then, $\widetilde{X}_H$ is not  o.s.s.;
\end{itemize}
\begin{itemize}
\item [$(vii)$] suppose the additional constraint $\Re \hspace{0.5mm}\eig(H) \subseteq (1/2,1)$ is in place. Then, for every $\gamma\in (0, \min \Re\eig(H)-1/2)$, there exists a modification of maofLm/rhofLm that is a.s.\ $\gamma$-H\"older continuous.
\end{itemize}
\end{theorem}
\begin{example}
A simple example of a L\'{e}vy measure satisfying \eqref{e:int_RezRez*=sigma^2_I=int_ImzImz*} is given by
$$
\mu_{\bbC^p}(d{\mathbf z}) =  \sum^{p}_{k=1}\delta_{(1+\imag)e_k}(d{\mathbf z}),
$$
where $e_k \in \bbR^p$, $k = 1,\hdots,p$, are the first $p$ canonical vectors.
\end{example}

\begin{remark}\label{r:YHtilde_proper}
The properness condition notwithstanding, the integrands in the stochastic integral representations of ofLm can be rank-deficient a.e.\ (see Lemma \ref{l:Ytilde_proper}).
\end{remark}

\begin{remark}\label{r:identifiability} Under conditions, the second order structures of ofLm and ofBm are identical. Therefore, the parametrization of the second order structure of ofLm is not identifiable (Didier and Pipiras \cite{didier:pipiras:2012}). Characterizing the (non)identifiability of the parametrization of ofLm -- namely, in regard to its finite-dimensional distributions -- is a topic for future work.
\end{remark}

Recall that, in the Gaussian case (ofBm), harmonizable and moving average stochastic integrals are representations of the same stochastic process (see Section \ref{s:ofBm_integ_repres}). Equivalently, they have the same covariance structure. As established in Theorem \ref{t:rhofLm_maofLm_integ_repres_chf}, under assumptions on the L\'{e}vy measure, rhofLm and maofLm share the covariance structure of ofBm. Nevertheless, they are rather distinct from ofBm. We shed light on such differences in the next three propositions. In Proposition \ref{p:XH_harm_YH_ma}, we provide natural alternative stochastic integral representations of rhofLm and maofLm in the time and Fourier domains, respectively. The representations (i.e., \eqref{e:XH_harm} and \eqref{e:YH_ma}) are formally similar to \eqref{e:X-tildeH(t)} and \eqref{e:XH(t)}, respectively. However, the random measures involved in each expression do \textit{not} satisfy the conditions stated in Definitions \ref{def:rhofLm} and \ref{def:maofLm}. In particular, the random measures generally display \textit{orthogonal} but \textit{dependent} increments. On the other hand, even though ofLm is never o.s.s., in Proposition \ref{p:maofLm->BH(t)_rhofLm->BH(t)} we establish that (rescaled) rhofLm and maofLm converge to an ofBm over different time ranges, i.e., in the large and small scaling limits, respectively. Remarkably, in Proposition \ref{p:lass2} we further show that maofLm and rhofLm may display operator self-similarity in the other limit directions, namely, maofLm can be operator self-similar in the \emph{small} scale limit, whereas rhofLm can be operator self-similar in the \emph{large} scale limit. However, such limits may display heavy-tailed marginal distributions, i.e., they are not ofBms.

We begin by establishing natural alternative stochastic integral representations of rhofLm and maofLm. These representations are based on random measures with uncorrelated increments; for the reader's convenience they are given explicitly in Proposition \ref{p:spec}.
\begin{proposition}\label{p:XH_harm_YH_ma}
Fix a matrix $H \in M(p,\bbR)$ whose eigenvalues satisfy \eqref{e:Re_eig(H)=(0,1)-(1/2)}. Let $\widetilde{X}_H = \{\widetilde{X}_H(t)\}_{t \in \bbR}$ be a rhofLm as in \eqref{e:X-tildeH(t)} under conditions \eqref{e:int_RezRez*=sigma^2_I=int_ImzImz*} and \eqref{e:Mtilde(dxi)=e^(i*theta)Mtilde(dxi)} on the associated random measure. Also let $X_H = \{X_H(t)\}_{t \in \bbR}$ be a maofLm as in \eqref{e:XH(t)} under condition \eqref{e:int_zz*=I} on the associated random measure.

\begin{itemize}
\item[$(i)$] Then, for $\widetilde{g}_t$ as in \eqref{e:fLm_filter_harm}, $X_H$ admits the representation
\begin{equation}\label{e:XH_harm}
{\{X_H(t)\}_{t \in \bbR} \stackrel{\textnormal{f.d.d.}}= \Big\{\int_\R  \widetilde{g}_t(x) \hspace{0.5mm}\Phi_{\mathcal M}(dx) \Big\}_{t \in \bbR},}
\end{equation}
where $\Phi_{\mathcal M}(dx)$ is a $\bbC^p$-valued, zero mean orthogonal-increment random measure such that $\bbE \Phi_{\mathcal M}(dx) \Phi_{\mathcal M}(dx)^* = dx \times I$.
\item[$(ii)$] Furthermore, for $g_t$ as in \eqref{e:ft(s)=F^(-)(gt(s))}, $\widetilde{X}_H$ admits the representation
\begin{equation}\label{e:YH_ma}
\{\widetilde{X}_H(t)\}_{t \in \bbR} \stackrel{\textnormal{f.d.d.}}= \Big\{\int_\R g_{t}(s) \hspace{0.5mm}\Phi_{\widetilde {\mathcal M}}(ds)\Big\}_{t\in\R},
\end{equation}
where $\Phi_{\widetilde {\mathcal M}}(ds)$ is a $\bbR^p$-valued, zero mean orthogonal-increment random measure such that $\bbE \Phi_{\widetilde {\mathcal M}}(ds) \Phi_{\widetilde {\mathcal M}}(ds)^* = ds \times I$.
\end{itemize}
\end{proposition}

\begin{example}
Suppose $M_- = 0$ and
\begin{equation}\label{e:Re_eig(H)_cap_1/2=empty}
\Re \hspace{0.5mm}\eig(H) \cap \{1/2\} = \emptyset.
\end{equation}
Let $A = \Gamma(D+I)e^{-i\pi D/2}$. Then, by expression \eqref{e:Prop3.1_DP2011}, we can recast representation \eqref{e:XH_harm} as
\begin{equation*}
\{X_H(t)\}_{t \in \bbR} \stackrel{\textnormal{f.d.d.}}= \Big\{\int_\R \frac{e^{\imag tx}-1}{\imag x} \big\{ x^{-D}_+ A + x^{-D}_- \overline A \big\}\hspace{0.5mm}\Phi_{\mathcal M}(dx) \Big\}_{t \in \bbR}.
\end{equation*}
On the other hand, for this same choice of the parameter $A$ and still assuming condition \eqref{e:Re_eig(H)_cap_1/2=empty} holds, again by expression \eqref{e:Prop3.1_DP2011} we can rewrite \eqref{e:YH_ma} as
\begin{equation*}
\{\widetilde{X}_H(t)\}_{t \in \bbR} \stackrel{\textnormal{f.d.d.}}= \Big\{\int_\R \big\{(t-s)_+^D - (-s)_+^D \big\} \hspace{0.5mm}\Phi_{\widetilde {\mathcal M}}(ds)\Big\}_{t\in\R}.
\end{equation*}
\end{example}

In Proposition \ref{p:maofLm->BH(t)_rhofLm->BH(t)}, we establish the large and small scale behaviors of maofLm and rhofLm, respectively. In the statement of the proposition, $\stackrel{\text{f.d.d.}}{\to}$ denotes the convergence of finite-dimensional distributions.
\begin{proposition}\label{p:maofLm->BH(t)_rhofLm->BH(t)} Let $H$ be a (Hurst) matrix whose eigenvalues satisfy condition \eqref{e:0<Re(h)<1}.
\begin{itemize}
\item[$(i)$] Let $X=\{X_H(t)\}_{t\in\R}$ be a maofLm with Hurst exponent $H$. Suppose its associated L\'{e}vy measure satisfies condition \eqref{e:int_zz*=I}.  Then,
$$
\left\{c^{-H}X_H(ct)\right\}_{t\in\R} \stackrel{\textnormal{f.d.d.}}{\to} \left\{B_H(t)\right\}_{t\in\R}, \quad c\to\infty,
$$
where $B_H$ is an ofBm with Hurst exponent $H$.
\item[$(ii)$] Let $\widetilde{X}_{H}=\{\widetilde{X}_{H}(t)\}_{t\in\R}$ be a rhofLm with exponent $H$. Suppose its associated L\'{e}vy measure satisfies condition \eqref{e:int_RezRez*=sigma^2_I=int_ImzImz*}.  Then, for every fixed $s\in\R$,
$$
\left\{\varepsilon^{-H}\big(\widetilde{X}_{H}(s+\varepsilon t)-\widetilde{X}_{H}(s)\big)\right\}_{t\in\R} \stackrel{\textnormal{f.d.d.}}{\to} \left\{B_H(t)\right\}_{t\in\R}, \quad \varepsilon \to 0^+,
$$
where $B_H$ is an ofBm with Hurst exponent $H$.
\end{itemize}
\end{proposition}

In Proposition \ref{p:lass2}, we show that some maofLm and rhofLm instances are o.s.s.\ in the small and large scale limits, respectively -- in both cases, with a different matrix scaling exponent. This occurs when the associated random measures $\mathcal M$ and $\widetilde{ \mathcal M}$ are chosen to be ``locally'' operator-stable, in the sense that their L\'evy measures around ${\mathbf 0}$ behave like that of an operator-stable L\'evy process.  These limiting processes, in turn, are instances of operator-stable o.s.s.~processes recently studied in Kremer and Scheffler \cite{kremer:scheffler:2019}. For the reader's convenience, the precise definition and more details about such measures and associated independently scattered random measures are provided in Section \ref{s:tempered_oper-stable_Levy_measures}.

Recall for any $M\in M(p,\bbR)$, $\lambda_i(M)$ denotes the $i^{\textnormal{th}}$ eigenvalue of $M$ in the ordering $\Re \lambda_1(M)\leq \ldots \leq \Re \lambda_p(M)$, where an arbitrary ordering of eigenvalues is adopted in case real parts are equal.
\begin{proposition}\label{p:lass2}
Let $B\in M(p,\R)$ be such that $\Re\hspace{0.5mm} \textnormal {eig} (B) \subseteq(1/2,1)$.

\begin{itemize}
\item[(i)] Let $X_H= \{X_H(t)\}_{t \in \bbR}$ be a maofLm under \eqref{e:Re_eig(H)=(0,1)-(1/2)}, and suppose its associated L\'{e}vy measure is given by $\mu_{B,q}$ as in \eqref{e:def_TalphaS}.   Further suppose that $HB=BH$, and that $\Re\lambda_p(H -(1/2)I)+ \Re \lambda_p(B)<1$.
Then, $X_H$ is locally o.s.s.\ with exponent
\begin{equation}\label{e:H-tilde_1=D+B}
\widetilde H_1= H+ \big( B-(1/2)I \big)
\end{equation}
in the sense that, for every fixed $s\in\R$,
\begin{equation}\label{e:ofLm_local_behavior_TOS}
\left\{\varepsilon^{-\widetilde H_1}\big(X_H(s+\varepsilon t)-X_H(s)\big)\right\}_{t\in\R} {\stackrel{{\textnormal{f.d.d.}}}{\to}} \left\{\Theta_{\widetilde H_1,B}(t)\right\}_{t\in\R}, \quad \varepsilon\to 0^+.
\end{equation}
In \eqref{e:ofLm_local_behavior_TOS}, $\Theta_{\widetilde H,B}(t)$ is an $\widetilde H_1$-o.s.s.\ process with representation
\begin{equation}\label{e:opstab_MA}
\Theta_{\widetilde H_1,B}(t) = \int_\R g_t(s) L_{B}(ds),
\end{equation}
where $L_B$ is an $\R^p$-valued independently scattered ID random measure generated by a full operator-stable random measure with exponent $B$ as in \eqref{e:def_opstab}.

\item[$(ii)$] Let $\widetilde X_H = \{\widetilde X_H(t)\}_{t \in \bbR}$ be a rhofLm, and suppose its associated L\'{e}vy measure $\mu_{\R^{2p}}$ in the identification \eqref{e:mu_Cp=muR2p} is given by  $ \mu_{\widetilde B,q}$ as in \eqref{e:def_TalphaS}, where  $\widetilde B =B\oplus B$.  Further suppose that $H$ and $A$ commute with $B$, and that $\Re\lambda_1(H)+(\frac{1}{2}-\Re \lambda_p(B))>0$ and $\Re\lambda_p(H)+(\frac{1}{2}-\Re \lambda_1(B))<1$. Then, $\widetilde X_H$ is asymptotically o.s.s.\ with exponent
\begin{equation}\label{e:H-tilde_2=D+(1/2I-B)}
\widetilde H_2=H +\big((1/2)I - B \big)
\end{equation}
 in the sense that
\begin{equation}\label{e:rhofLm_local_behavior_TOS}
\left\{c^{-\widetilde H_2}\widetilde X_H(ct)\right\}_{t\in\R} {\stackrel{\textnormal{f.d.d.}}{\to}} \left\{\Theta'_{\widetilde H_2, B}(t)\right\}_{t\in\R}, \quad c\to\infty. \end{equation}
In \eqref{e:rhofLm_local_behavior_TOS}, $ \Theta'_{\widetilde H,B}$ is an $\widetilde H_2$-o.s.s.\ process with representation
\begin{equation}\label{e:opstab_harm}
\Theta'_{\widetilde H,B}(t) =   2\Re\Big(\int_{\R}\widetilde g_t(x) \widetilde L_{B}(dx)\Big),\end{equation}
where $\widetilde L_B(dx)$ is a $\C^p$-valued ID independently scattered random measure generated by a full operator-stable random measure with exponent $\widetilde B$ as in \eqref{e:def_opstab}.
\end{itemize}
\end{proposition}

\section{Time reversibility}\label{s:time_revers}

Recall that a stochastic process $X = \{X(t)\}_{t \in \bbR}$ is said to be time-reversible if $\{X(-t)\}_{t \in \bbR} \stackrel{\textnormal{f.d.d.}}= \{X(t)\}_{t \in \bbR}$. In this section, we provide characterizations of time reversibility for maofLm and rhofLm under mild assumptions. In the characterizations, the true difficulty lies in establishing \textit{necessary }conditions, i.e., what the assumption of time reversibility implies about the parametric representations of maofLm and rhofLm. The proofs require results on the \textit{uniqueness} of multivariate stochastic integral representations, which are developed in Section \ref{s:uniqueness}. To provide these uniqueness results, we adapt the fundamental framework constructed in Kabluchko and Stoev \cite{kabluchko:stoev:2016}, Sections 2.1 and 2.2 (see also Maruyama \cite{maruyama:1970}, Samorodnitsky \cite{samorodnitsky:2016}, chapter 3, and Rosi\'{n}ski \cite{rosinski:2018}).

\begin{example}
If a maofLm $X_H$ is time-reversible and satisfies \eqref{e:int_zz*=I}, then its covariance function is given by the fBm-like formula \eqref{e:time_revers_ofBm} with $\Sigma = \E X_H(1) X_H(1)^* $. If a rhofLm $\widetilde X_H$ is time-reversible and satisfies \eqref{e:int_RezRez*=sigma^2_I=int_ImzImz*}, then its covariance function is also given by the formula \eqref{e:time_revers_ofBm} with $\Sigma = \E \widetilde X_H(1) \widetilde X_H(1)^*$. In general, an explicit formula for the covariance function of ofLm is not available. In fact, in the Gaussian case, expression \eqref{e:int_RezRez*=sigma^2_I=int_ImzImz*} is equivalent to time reversibility (see Didier and Pipiras \cite{didier:pipiras:2011}, Proposition 5.2).
\end{example}

To investigate time reversibility in the framework of ofLm, it is convenient to slightly generalize the notation. Simply put, the new argument $\omega$ stands for either the Fourier or time arguments $x$ or $s$. In turn, the vector ${\boldsymbol \varpi} = (\omega,{\mathbf z})$ includes both $\omega \in \bbR$ and the L\'{e}vy measure argument ${\mathbf z} \in \bbR^q$, where either $q = p$ or $q = 2p$. So, more precisely, let
$\overline{\Omega} =  \R \times \bbR^q$, ${\mathcal B}= {\mathcal B}(\overline{\Omega})$ (cf.\ expression \eqref{e:Omega-bar}). Let
\begin{equation}\label{e:kappa(d_omega)_rigidity}
\kappa(d {\boldsymbol \varpi}) = d\omega\otimes \mu(d{\mathbf z}),
\end{equation}
where $\mu(d{\mathbf z})$ is a L\'{e}vy measure satisfying \eqref{e:int_|z|^2_mu(dz)<infty_in_R}.Whenever convenient, we write $\eta(d\omega) \equiv d \omega$. Also define
$$
{\mathcal L}^2_{\kappa(d {\boldsymbol \varpi})}(\overline{\Omega}) = {\mathcal L}^2_{d\omega\otimes \mu(d{\mathbf z})}(\R \times \bbR^q)
$$
\begin{equation}\label{e:integrability_univar}
= \Big\{\varphi:\bbR \times \bbR^q \rightarrow \bbR^p:  \int_{\bbR}\int_{\bbR^q}  \varphi(\omega,{\mathbf z})^*\varphi(\omega,{\mathbf z}) \hspace{0.5mm}\mu(d{\mathbf z})d\omega< \infty\Big\}
\end{equation}
(cf.\ expression \eqref{e:|varphi|_{L^2}}). Then, we express the compensated Poisson random measure on ${\mathcal B}(\overline{\Omega})$ as
\begin{equation}\label{e:N-tilde(ds)}
\bbR \ni \widetilde{N}(d{\boldsymbol \varpi}) \equiv N(d{\boldsymbol \varpi}) - \kappa(d{\boldsymbol \varpi}) \equiv \widetilde{N}(d\omega,d {\mathbf z}),
\end{equation}
where $N(d{\boldsymbol \varpi}) \equiv N(d\omega,d {\mathbf z})$ is a Poisson random measure (cf.\ \eqref{e:Ntilde(dxi,dz)} and \eqref{e:N(ds,dz)_real_z}).  Let $f_t({\boldsymbol \varpi})$ and $\mathfrak{g}_t(\omega)$, $t \in \R$, be two families of $\bbR^p$- and $M(p,q,\bbR)$-valued functions, respectively, where
\begin{equation}\label{e:ft(omega)=gt(s)*x}
\{f_t({\boldsymbol \varpi})\}_{t \in \R} := \{\mathfrak{g}_t(\omega){\mathbf z} \}_{t \in \R}  \subseteq  {\mathcal L}^2_{\kappa(d{\boldsymbol \varpi})}(\overline{\Omega}).
\end{equation}

\begin{example}
For \eqref{e:mafLm_chf}, we can write \eqref{e:ft(omega)=gt(s)*x} with
$$
\omega = s, \quad {\mathfrak g}_t(s) = g_t(s) \in M(p,\bbR), \quad q = p \quad \textnormal{and} \quad \mu(d{\mathbf z})\textnormal{ as in \eqref{e:mafLm_chf}}.
$$
For \eqref{e:chf_harm_gt}, we can reexpress \eqref{e:ft(omega)=gt(s)*x} with
$$
\omega = x, \quad {\mathfrak g}_t(x) = (\Re \widetilde{g}_t(x),\Im \widetilde{g}_t(x)) \in M(p,2p,\bbR), \quad q = 2p \quad \textnormal{and} \quad  \mu(d{\mathbf z}) = (\mu_{\bbR^{2p}} \circ \varsigma^{-1})(d{\mathbf z}),
$$
where $\varsigma({\mathbf z}) = (2 \Re{\mathbf z},-2 \Im{\mathbf z})$.
\end{example}

The main results in this section require some notion of \textit{minimal} (stochastic integral) representation. In the following definition, we revisit the notion of minimality as put forward in Kabluchko and Stoev \cite{kabluchko:stoev:2016}.  \begin{definition}\label{def:minimality}
Let $T\subseteq \R$, and consider the $\bbR^p$-valued stochastic process $X = \{X(t)\}_{t \in T}$ given by the stochastic integral representation
\begin{equation}\label{e:X(t)=int_ft(omega)M(domega)}
X(t) = \int_{\overline{\Omega}}f_t({\boldsymbol \varpi}) \widetilde{N}(d{\boldsymbol \varpi}) = \int_{\R \times \bbR^q} {\mathfrak g}_t(\omega) {\mathbf z} \widetilde{N}(d\omega,d{\mathbf z}), \quad t \in T
\end{equation}
We say $\{f_t\}_{t \in T}$ is a \textit{minimal} representation of the ID stochastic process $X$ \textit{with respect to ${\mathcal B} \mod \kappa$} if the following two conditions hold.
\begin{itemize}
\item [$(i)$] $\sigma(\{f_t\}_{t \in T}) = {\mathcal B} \mod \kappa$, i.e., for every $B \in {\mathcal B}$, there exists $ A \in \sigma(\{f_t\}_{t \in T})$ such that $\kappa(A \Delta B) = 0$; and
\item [$(ii)$] there is no $B \in {\mathcal B}$ such that $\kappa(B) > 0$ and, for every $t \in \R$, $f_t \equiv 0$ a.e.\ on $B$.
\end{itemize}
\end{definition}

In the following theorem, we characterize time reversibility for maofLm. For comments on the minimality assumption, see Remark \ref{r:minimality}.
\begin{theorem}\label{t:maofLm_time-reversibility}
Let $H$ be a (Hurst) matrix whose eigenvalues satisfy \eqref{e:Re_eig(H)=(0,1)-(1/2)}. Let $X_H = \{X_H(t)\}_{t \in \bbR}$ be a maofLm with Hurst matrix $H$.  Further assume that
\begin{equation}\label{e;M+,M-_in_GL(p,R)}
M_+,M_- \in GL(p,\bbR)
\end{equation}
and $\{f_t({\boldsymbol \varpi}),t \in \bbR\}= \{g_t(s){\mathbf z},t \in \bbR\}$ is a minimal representation of $X_H$ with respect to ${\mathcal B}(\bbR^{p+1}) \hspace{-1mm}\mod \kappa$, where $\kappa(d{\boldsymbol \varpi}) = ds \otimes \mu(d{\mathbf z})$ and $\mu(d{\mathbf z})$ is as in \eqref{e:mafLm_chf}.  Then, the following conditions are equivalent.
\begin{itemize}
\item[$(i)$]  $X_H$ is time-reversible;
\item [$(ii)$] The following two conditions hold:
\begin{itemize}
\item[$(a)$] $(M^{-1}_{-}M_{+})\hspace{0.5mm}|_{\textnormal{supp}(\mu)}$ is an involution, i.e.,
\begin{equation}\label{e:M-^(-1)*M+*z=M+^(-1)*M-*z}
M^{-1}_- M_+ \hspace{0.5mm}{\mathbf z} = M^{-1}_+ M_- \hspace{0.5mm}{\mathbf z} \quad \mu(d{\mathbf z})\textnormal{--a.e.};
\end{equation}
\item[$(b)$] the map $\mathbf z \mapsto M^{-1}_+ M_-\mathbf z$ preserves the measure $\mu$, i.e.,
\begin{equation}\label{e:mu(matrix*dx)=mu(dz)}
\mu\big((M^{-1}_- M_+) \hspace{0.5mm}d{\mathbf z} \big) = \mu(d{\mathbf z}).
\end{equation}
\end{itemize}
\end{itemize}
In $(ii)$, condition (a) can be replaced by
\begin{itemize}
\item []
\begin{itemize}
\item [$(a')$]
$$
g_{-t}(s)\mathbf z=g_t(-s)M_-^{-1}M_+ \mathbf z=g_t(-s)M_+^{-1}M_- \mathbf z \quad \kappa(ds,d\mathbf z)\textnormal{--a.e.}
$$
\end{itemize}
\end{itemize}
\end{theorem}

\begin{example}\label{ex:time_revers_ofBm_vs_maofLm}
It is illustrative to compare the conditions for time reversibility for ofBm and maofLm. For the former, time reversibility (i.e., expression \eqref{e:time_revers_ofBm}) is equivalent, in the time domain, to the parametric condition
\begin{equation}\label{e:ofBm_time_revers_time}
\cos\Big(\frac{\pi D}{2}\Big)(M_+ + M_-)(M^*_+ - M^*_-)\sin\Big(\frac{\pi D^*}{2}\Big)
= \sin\Big(\frac{\pi D}{2}\Big)(M_+ - M_-)(M^*_+ + M^*_-)\cos\Big(\frac{\pi D^*}{2}\Big)
\end{equation}
(Didier and Pipiras \cite{didier:pipiras:2011}, Corollary 5.1), where $D = H-(1/2)I$ and the cosine and sine of matrices are interpreted in the sense of primary matrix functions (Horn and Johnson \cite{horn:johnson:1991}).

So, suppose the conditions used in Theorem \ref{t:maofLm_time-reversibility} hold; namely, suppose \eqref{e:Re_eig(H)=(0,1)-(1/2)}, \eqref{e;M+,M-_in_GL(p,R)}, and that $X_H$ is time-reversible.  In addition, assume the L\'{e}vy measure $\mu$ satisfies the second moment condition \eqref{e:int_zz*=I} -- otherwise, ofBm and maofLm  have incompatible parameterizations (cf.\ Theorem \ref{t:rhofLm_maofLm_integ_repres_chf}, $(iv)$). Based on a change of variable ${\mathbf w} = M^{-1}_+ M_- {\mathbf z}$, using \eqref{e:mu(matrix*dx)=mu(dz)},
$$
(M^{-1}_+ M_-) \Big(\int_{\bbR^p}{\mathbf z}{\mathbf z}^* \mu(d{\mathbf z})\Big)(M^*_- (M^*_+)^{-1}) =\int_{\bbR^p}{\mathbf w}{\mathbf w}^* \mu(d{\mathbf w}) = I.
$$
Hence, $M_+ M^*_+ = M_- M^*_-$. Also, under \eqref{e:int_zz*=I}, \eqref{e:M-^(-1)*M+*z=M+^(-1)*M-*z} implies $M^{-1}_- M_+= M^{-1}_+ M_- $. As a consequence,
$$
(M_+ + M_-)(M^*_+ - M^*_-) = \mathbf 0=  -(M_- + M_+)(M^*_+ - M^*_-),
$$
 which in turn implies condition \eqref{e:ofBm_time_revers_time}.   Conversely, we may pick $M_+,M_-$ satisfying \eqref{e:ofBm_time_revers_time} but for which $(M_+ + M_-)(M^*_+ - M^*_-) \neq {\mathbf 0}$. In other words, among the instances of maofLm satisfying \eqref{e:int_|z|^2_mu(dz)<infty_in_R}, the conditions for time reversibility of maofLm as established in Theorem \ref{t:maofLm_time-reversibility} are more stringent than those for the time reversibility of ofBm.
\end{example}
\begin{example}\label{ex:gaussian_jumps_symmetry}  Let $X_H$ be a time-reversible maofLm. Assume $\Sigma := \int_{\R^p} \mathbf z \mathbf z^* \mu(d\mathbf z)$ has full rank. A simple calculation shows that there exists a symmetric orthogonal matrix $O$ such that
\begin{equation}\label{e:symmetric_orthog_condition}
M_+^{-1} M_-=M_-^{-1} M_+ = \Sigma^{1/2} O\Sigma^{-1/2}
\end{equation}
(Lemma \ref{l:gaussian_jumps_symmetry}). In other words, \eqref{e:symmetric_orthog_condition} is a necessary condition for time reversibility. In light of Theorem \ref{t:maofLm_time-reversibility}, this implies the following.
\begin{itemize}
\item[$(i)$] If $p=1$, $X_H$ is time-reversible if and only if either $(a)$ $M_+=M_-$; or $(b)$ $M_+=-M_-$ and $\mu$ is symmetric (i.e., $\mu(-dz) = \mu(dz)$). (Indeed, relation \eqref{e:symmetric_orthog_condition} implies necessity, whereas sufficiency follows from both conditions \eqref{e:mu(matrix*dx)=mu(dz)} and \eqref{e:M-^(-1)*M+*z=M+^(-1)*M-*z}).
\item[$(ii)$] If $\mu$ is a full, zero-mean Gaussian measure on $\R^p$, then it is characterized by the matrix $\Sigma$. Therefore, $X_H$ is reversible if and only if \eqref{e:symmetric_orthog_condition} holds.
\end{itemize}
\end{example}

Turning to the Fourier domain, let
\begin{equation}\label{r:rhofLm_change_Levy_repres}
\widetilde{X}_H = \{\widetilde{X}_H(t) \}_{t \in \bbR}
\end{equation}
be a rhofLm with kernel $f_t({\boldsymbol \varpi}) = \widetilde{g}_{t}(x){\mathbf z}$ and measure \eqref{e:kappa(d_omega)_rigidity} given by $\kappa(d{\boldsymbol \varpi}) = dx \otimes \mu(d{\mathbf z}) \equiv \mu(d{\mathbf z}) \otimes dx$. Note that $\widetilde{X}_H$ can also be represented based on the measure
\begin{equation}\label{e:kappa-tilde=dx_x_mu(dz)}
\widetilde \kappa(d\boldsymbol \varpi) =  dx \otimes \widetilde \mu (d\mathbf z),
\end{equation}
where
\begin{equation}\label{e:mu-tilde}
\widetilde \mu(d\mathbf z) = \frac{\mu(d\mathbf z) +\mu(\overline{d\mathbf z}) }{2}
\end{equation}
(see Lemma \ref{l:rhofLm_change_Levy_repres}).
This fact is used in the following theorem, where we characterize time reversibility for rhofLm. For comments on the minimality assumption, see Remark \ref{r:minimality}.
\begin{theorem}\label{t:rhofLm_time-revers_general}
Let $H$ be a (Hurst) matrix satisfying \eqref{e:0<Re(h)<1}. Let $\widetilde{X}_H = \{\widetilde{X}_H(t)\}_{t \in \bbR}$ be a rhofLm with Hurst matrix $H$. Assume
$$
\{f_t({\boldsymbol \varpi})\}_{t \in \bbR}= \{\widetilde{g}_t(x){\mathbf z}\}_{t \in \bbR}\textnormal{ is a minimal representation}
$$
\begin{equation}\label{e:gt_is_minimal_for_X-tilde_general}
\textnormal{of $\{\widetilde{X}_H(t)\}_{t \in \bbR}$ with respect to }{\mathcal B}(\bbR\times \bbR^{2p}) \mod {\widetilde \kappa},
\end{equation}
where ${\widetilde \kappa}$ is a measure of the form \eqref{e:kappa-tilde=dx_x_mu(dz)} and we identify $\mu_{\C^p}\equiv\mu_{\R^{2p}}$ (see \eqref{e:mu_Cp=muR2p}). Further assume
\begin{equation}\label{e:A_invertible_mu(dz)_ae}
 A,\overline A \in GL(p,\C).
\end{equation}
Then, $\widetilde X_H$ is time reversible if and only if the map $\mathbf z\mapsto  - A^{-1} \overline A  \mathbf z$ preserves the measure $\widetilde \mu$ as in \eqref{e:kappa-tilde=dx_x_mu(dz)}, i.e.,
\begin{equation}\label{e:mu(A^(-1)bar(Az))=mu(z)}
\widetilde{\mu}\big( -\overline A^{-1} A \hspace{0.5mm}d \mathbf z\big)= \widetilde{\mu}(d\mathbf z).
\end{equation}
\end{theorem}

\begin{example}
Under the assumptions of Theorem \ref{t:rhofLm_time-revers_general}, a sufficient condition for $\widetilde{X}_H$ to be time-reversible is that $\Re(A) = {\mathbf 0}$. In fact, in this case, $\overline A^{-1} A = -I$; hence, condition \eqref{e:mu(A^(-1)bar(Az))=mu(z)} holds.
\end{example}

\begin{example}\label{ex:time_revers_ofBm_vs_rhofLm}As in Example \ref{ex:time_revers_ofBm_vs_maofLm}, we now compare the conditions for time reversibility of ofBm to those for rhofLm. For the former, time reversibility (i.e., expression \eqref{e:time_revers_ofBm}) is equivalent, in the Fourier domain, to the parametric condition
\begin{equation}\label{e:ofBm_time_revers_Fourier}
AA^* = \overline{AA^*}
\end{equation}
(Didier and Pipiras \cite{didier:pipiras:2011}, Theorem 5.1).

So, suppose the conditions used in Theorem \ref{t:rhofLm_time-revers_general} hold; namely, suppose conditions \eqref{e:0<Re(h)<1}, \eqref{e:A_invertible_mu(dz)_ae} and \eqref{e:gt_is_minimal_for_X-tilde_general} are satisfied. In addition, assume the L\'{e}vy measure $\mu$ satisfies the second moment condition \eqref{e:int_RezRez*=sigma^2_I=int_ImzImz*}, so as to ensure ofBm and rhofLm have the same covariance structure and compatible parametrizations (cf.\ Theorem \ref{t:rhofLm_maofLm_integ_repres_chf}, $(iv)$). Then, for ${\mathbf z} = {\mathbf z}_1 + \imag {\mathbf z}_2$,
\begin{equation}
\int_{\bbC^p} {\mathbf z}{\mathbf z}^* \widetilde \mu(d{\mathbf z}) = (1/2)I + \imag \int_{\bbC^p} ({\mathbf z}_2{\mathbf z}^*_1 - {\mathbf z}_1{\mathbf z}^*_2) \widetilde \mu(d{\mathbf z}) = (1/2)I,
\end{equation}
where the last equality is a consequence of the general property $\widetilde \mu(d\mathbf z) = \widetilde\mu(\overline{d\mathbf z})$. Thus, assuming time reversibility, based on a change of variable $\mathbf w=  -\overline A^{-1} A  \mathbf z$, condition \eqref{e:mu(A^(-1)bar(Az))=mu(z)} implies that
\begin{equation}\label{e:int_C^p_zz*_mu(dz)}
\big( \overline A^{-1} A \big) \int_{\bbC^p} {\mathbf z}{\mathbf z}^* \widetilde{\mu}(d{\mathbf z})\big( \overline A^{-1} A \big)^* =\int_{\C^p}{\mathbf w}{\mathbf w}^* \widetilde{\mu}(d{\mathbf w}) = (1/2)I.
\end{equation}
Hence, $\big( \overline A^{-1} A \big) \big( \overline A^{-1} A \big)^*=I$, which implies condition \eqref{e:ofBm_time_revers_Fourier}. In regard to the converse, however, by choosing $A,A^*$ satisfying \eqref{e:ofBm_time_revers_Fourier}, we may easily find a L\'{e}vy measure $\mu$ under condition \eqref{e:int_RezRez*=sigma^2_I=int_ImzImz*} such that \eqref{e:mu(A^(-1)bar(Az))=mu(z)} is not satisfied. In other words, among the instances of rhofLm satisfying \eqref{e:int_RezRez*=sigma^2_I=int_ImzImz*}, the conditions for time reversibility of rhofLm as established in Theorem \ref{t:rhofLm_time-revers_general} are stronger than those for the time reversibility of ofBm.
\end{example}

\begin{example}\label{e:p=1_harm} Write $\Sigma= \int_{\bbC^p} {\mathbf z}{\mathbf z}^* \widetilde \mu(d{\mathbf z})$, and suppose $\Sigma$ has full rank.  Reasoning similarly to Example \ref{ex:gaussian_jumps_symmetry}, a necessary condition for the time reversibility of $\widetilde X_H$ is that
\begin{equation}\label{e:unitary_condition}
A^{-1}\overline A = \Sigma^{1/2}U\Sigma^{-1/2},
\end{equation}
where $U\in M(p,\C)$ is some unitary matrix. This implies the following.
\begin{itemize}
\item[$(i)$] If $p=1$, $\widetilde X_H$ is time-reversible if and only if $\widetilde \mu(-dz)=\widetilde \mu(e^{\imag 2\theta} dz)$, where $\theta = \text{arg}\hspace{0.5mm} A$.
\item[$(ii)$] If $\mu$ is a zero-mean Gaussian measure on $\C^p$ satisfying $\mu(d\mathbf z) = \mu (e^{\imag \theta} d\mathbf z)$, $\theta\in(-\pi,\pi]$, then $\widetilde X_H$ is reversible if and only if \eqref{e:unitary_condition} holds. (Indeed, in this case $\widetilde \mu = \mu$, and $\mu$ is completely determined by $\Sigma$, showing \eqref{e:unitary_condition} holds if and only if \eqref{e:mu(A^(-1)bar(Az))=mu(z)} holds by taking second moments.)
\end{itemize}
\end{example}

\begin{remark}\label{r:minimality}
The minimality of a representation $\{f_t\}_{t \in T}$ can always be enforced by replacing $\overline{\Omega}$ with $\textnormal{supp}\{f_t, t \in T\}$ and by choosing ${\mathcal B} = \sigma\{f_t,t \in T\}$. However, establishing the minimality of a representation based on a given Borel space such as $(\bbR^{q+1},{\mathcal B}(\bbR^{q+1}))$ is not, in general, straightforward (see Definition \ref{def:isomorphism_Borel_space}; cf.\ Kabluchko and Stoev \cite{kabluchko:stoev:2016}, Remark 2.18).  It can be shown by elementary -- though long and tedious -- arguments, that minimality over $(\R^{q+1}, {\mathcal B}(\R^{q+1}))$ naturally holds for some simple cases such as when $\mu$ is a point mass. It remains an open question as to when it holds in general.
\end{remark}

\section{Conclusion}

In this paper, we construct ofLm, a broad class of generally non-Gaussian stochastic processes that are covariance operator self-similar, have wide-sense stationary increments and display infinitely divisible marginal distributions. The ofLm class generalizes the univariate fractional L\'evy motion as well as the multivariate ofBm. The ofLm class can be divided into two types, namely, maofLm (moving average) and rhofLm (real harmonizable), both of which share the covariance structure of ofBm, under assumptions. We show that both maofLm and rhofLm admit stochastic integral representations in the time and Fourier domains. Though never o.s.s., the small- and large-scale limiting behaviors of maofLm and rhofLm are generally distinct. This stands in sharp contrast with the Gaussian case, where moving average and harmonizable stochastic integrals are representations of the same stochastic process. We characterize time reversibility for ofLm in terms of its parameters and L\'evy measure, starting from a framework for the uniqueness of finite second moment, multivariate stochastic integral representations. In particular, we show that, under non-Gaussianity, the parametric conditions for time reversibility are generally more restrictive those in the Gaussian case (ofBm).

This work leaves a number of issues to be explored and open research questions. These include: $(i)$ efficient simulation schemes for ofLm, in particular in regards to the effect of the dimension $p$; $(ii)$ the construction of statistical methodology that accounts for the impact of the tails of L\'{e}vy noise as a measure of non-Gaussianity; $(iii)$ applications in fields such as in Physics or Signal Processing, where the presence of fractal, second order behavior is well established, but where the modeling of non-Gaussian features is still a widely open area of research.

\appendix

\section{Properties of stochastic integrals}\label{s:properties_stoch_integrals}
Before establishing the results in Section \ref{s:fLm}, we lay out a few basic facts about the general stochastic integrals defined in Section \ref{s:stoch_integrals}. We start off with the Fourier domain.
By construction, for $\varphi_1, \varphi_2  \in  {\mathcal L}^2_{ds\otimes \mu(d{\mathbf z})}$, the $\bbC^p$-valued stochastic integrals of the form \eqref{e:int_varphi_N} satisfies the isometry-type property
$$
\bbE \Big(\int_{\bbR \times \bbC^p} \varphi_1(x,{\mathbf z} ) \widetilde{N}(dx,d{\mathbf z} ) \Big) \Big(\int_{\bbR \times \bbC^p} \varphi_2(x',{\mathbf z}' ) \widetilde{N}(dx',d{\mathbf z}' )\Big)^*
$$
\begin{equation*}
= \int_{\bbR} \int_{ \bbC^p} \varphi_1(x,{\mathbf z} )\varphi_2(x,{\mathbf z} )^* \hspace{1mm}\mu(d{\mathbf z} )dx.
\end{equation*}
In particular, consider the functions
\begin{equation}\label{e:varphi(xi,z)=2Re(g(xi)z)}
\varphi_i(x,{\mathbf z} )= 2 \Re(\widetilde{g}_i(x){\mathbf z}), \quad \widetilde{g}_i \in L^2_{\textnormal{Herm}}(\bbR), \quad i = 1,2.
\end{equation}
Then, since $\widetilde{g}_i$, $i=1,2$, is Hermitian,
$$
\bbE \Big(\int_{\bbR} \widetilde{g}_1(x)\widetilde{{\mathcal M}}(dx)\Big)\Big(\int_{\bbR} \widetilde{g}_2(x')\widetilde{{\mathcal M}}(dx')\Big)^*
$$
\begin{equation}\label{e:E_int_isometry_harm_f(xi)_cross}
= 4 \int_{\bbR} \Re \widetilde{g}_1(x) \Big(\int_{\bbC^p} \Re {\mathbf z} \Re {\mathbf z}^* \mu(d{\mathbf z}) \Big)\Re \widetilde{g}_2(x)^* \hspace{0.5mm}d x +  4 \int_{\bbR} \Im \widetilde{g}_1(x) \Big(\int_{\bbC^p} \Im {\mathbf z} \Im {\mathbf z}^* \mu(d{\mathbf z}) \Big)\Im \widetilde{g}_2(x)^* \hspace{0.5mm}d x.
\end{equation}
Moreover, the joint characteristic function of the real and imaginary parts of the $\bbC^p$-valued stochastic integral $\int \varphi d\widetilde{N}$, $\varphi \in {\mathcal L}^2_{dx\otimes \mu(d{\mathbf z})}$, is given by
$$
\bbE \exp \Big\{ \imag \Big({\mathbf u}^*_1 \int \Re(\varphi)d\widetilde{N}+ {\mathbf u}^*_2 \int \Im(\varphi)d\widetilde{N}\Big)\Big\}
$$
\begin{equation}\label{e:int_varphi_N_chf}
= \exp \Big\{ \int_{\bbR}\int_{ \bbC^p} \Big[ e^{\imag({\mathbf u}^*_1 \Re(\varphi)+{\mathbf u}^*_2 \Im(\varphi))}-1 -\imag({\mathbf u}^*_1 \Re(\varphi)+{\mathbf u}^*_2 \Im(\varphi))\Big]\mu(d{\mathbf z})dx\Big\},
\end{equation}
for ${\mathbf u}_1,{\mathbf u}_2 \in \bbR^p$ (Sato \cite{sato:2006}). Note that, under condition \eqref{e:int_|z|^2_mu(dz)<infty}, the integral on the right-hand side of \eqref{e:int_varphi_N_chf} is finite in view of the inequality
\begin{equation}\label{e:|exp(ix)-1-ix|=<C|x|^2}
|e^{\imag y}-1 - \imag y| \leq |y|^2, \quad y \in \bbR.
\end{equation}
In particular, for a function $\varphi(x,{\mathbf z})$ of the form \eqref{e:varphi(xi,z)=2Re(g(xi)z)},
\begin{equation}
\bbE \exp \Big\{ \imag \Big({\mathbf u}^* \int_{\bbR \times \bbC^p} \varphi(x,{\mathbf z}) \widetilde{N}(d x, d {\mathbf z})\Big) \Big\}
= \exp \Big\{ \int_{\bbR } \widetilde{\psi}(\widetilde{g}(x)^*{\mathbf u}) dx\Big\}, \quad {\mathbf u} \in \bbR^p,
\end{equation}
where
\begin{equation}\label{e:int_varphi=2Re(gz)_chf_Levy_symbol}
\widetilde{\psi}({\mathbf v}) = \int_{\bbC^p} \Big( e^{\imag 2 \Re \langle {\mathbf v},{\mathbf z}\rangle  } -1 -\imag 2 \Re \langle {\mathbf v},{\mathbf z}\rangle \Big) \mu(d{\mathbf z}) , \quad {\mathbf v} \in \bbC^p,
\end{equation}
and $\langle {\mathbf v},{\mathbf z}\rangle := {\mathbf v}^*{\mathbf z}$.  Equivalently, if we regard $\mu(d\mathbf z)=\mu_{\R^{2p}}(d\mathbf z)$ as a measure on ${\mathcal B}(\R^{2p})$, identifying each $\mathbf z_1 + \imag \mathbf z_2 \in \bbC^p$ with $(\mathbf z_1,\mathbf z_2)\in \bbR^{2p}$, then for $\mathbf z, \mathbf v \in \bbR^{2p}$, we may write
\begin{equation}
\widetilde{\psi}({\mathbf v}) = \int_{\bbR^{2p}}  \Big( e^{\imag 2 \langle \mathbf v, \mathbf z\rangle } -1 -\imag 2  \langle \mathbf v, \mathbf z\rangle  )\Big) \mu_{\R^{2p}}(d{\mathbf z}), \ {\mathbf v}\in \bbR^{2p}.
\end{equation}

Given the stochastic integral \eqref{e:X-tilde(t)}, for every $n$, any ${\mathbf t} = (t_1,\hdots,t_n) \in \bbR^n$ and any ${\mathbf u}_1,\hdots,{\mathbf u}_n \in \bbR^p$, the characteristic function of the finite-dimensional distributions of the $\bbR^p$-valued stochastic process $\widetilde{X}$ is given by
$$
\bbE \exp \Big\{\imag \sum^{n}_{k=1}\langle {\mathbf u}_k, \widetilde{X}(t_k) \rangle\Big\}
$$
\begin{equation}\label{e:chf_harm}
= \exp\Big\{ \int_{\bbR} \int_{\bbC^p}\Big[e^{\imag \hspace{0.25mm}2 \hspace{0.25mm}\Re \big( \sum^{n}_{k=1}{\mathbf u}^*_k \hspace{0.5mm}\widetilde{g}_{t_k}(x){\mathbf z} \big)}- 1 -
\imag \hspace{0.25mm}2\hspace{0.25mm} \Re \Big( \sum^{n}_{k=1}{\mathbf u}^*_k \hspace{0.5mm}\widetilde{g}_{t_k}(x){\mathbf z} \Big)\Big]\mu(d{\mathbf z}) dx \Big\}.
\end{equation}
In particular, the random vectors $\widetilde{X}(t)$, $t \in \bbR$, are ID (cf.\ Samorodnitsky \cite{samorodnitsky:2016}, Theorem 3.3.2, $(ii)$).

Turning to the time domain, let $\varphi_i(s,{\mathbf z})= g_i(s){\mathbf z}$, $i = 1,2$. By construction, the stochastic integral \eqref{e:int_f(s)M(ds)} satisfies the isometry property
\begin{equation}\label{e:E_int_isometry}
\bbE \Big(\int_{\bbR} g_1(s){\mathcal M}(ds)\Big)\Big(\int_{\bbR} g_2(s'){\mathcal M}(ds')\Big)^* = \int_{\bbR} g_1(s)\Big(\int_{\bbR^{p}} {\mathbf z}{\mathbf z}^* \hspace{0.5mm}\mu(d {\mathbf z}) \Big)g_2(s)^* \hspace{0.5mm}ds,
\end{equation}
where $\tr(\int_{\bbR^{p}} {\mathbf z}{\mathbf z}^* \hspace{0.5mm}\mu(d {\mathbf z})) < \infty$. Moreover, for $\varphi(s,{\mathbf z}) \in {\mathcal L}^2_{ds \otimes \mu(d{\mathbf z})}$, the characteristic function of the $\bbR^p$-valued stochastic integral $\int \varphi(s,{\mathbf z}) \widetilde{N}(ds,d{\mathbf z})$ is given by
$$
\bbE e^{\imag {\mathbf u}^* \int_{\bbR^{p+1}} \varphi(s,{\mathbf z}) \widetilde{N}(ds,d{\mathbf z})  }
$$
\begin{equation}\label{e:exp(int_psi(f)ds)_time}
= \exp \Big\{\int_{\bbR} \int_{\bbR^p} \Big(e^{\imag  \langle {\mathbf u}, \varphi(s,{\mathbf z})\rangle }-1-\imag \langle {\mathbf u}, \varphi(s,{\mathbf z})\rangle\Big)\hspace{0.5mm}\mu(d{\mathbf z})ds \Big\}, \quad {\mathbf u} \in \bbR^p.
\end{equation}
Under condition \eqref{e:int_|z|^2_mu(dz)<infty} (restricted to $\bbR^p$), the integral on the right-hand side of \eqref{e:exp(int_psi(f)ds)_time} is convergent in view of the inequality \eqref{e:|exp(ix)-1-ix|=<C|x|^2}.

Given expression \eqref{e:Y=int_f(t,s)L(ds)_change}, for any $n$, the joint characteristic function of the stochastic process $X = \{X(t)\}_{t \in \bbR}$ at the time points ${\mathbf t} = (t_1,\hdots,t_n)$ is given by
$$
\bbE e^{\imag \sum^{n}_{k=1}\langle{\mathbf u}_k, X(t_k)\rangle} = \exp \Big\{ \int_{\bbR} \int_{\bbR^p} \Big(e^{\imag \sum^{n}_{k=1} \langle{\mathbf u}_k, g_{t_k}(s){\mathbf z}\rangle }-1-\imag \sum^{n}_{k=1}\langle{\mathbf u}_k, g_{t_k}(s){\mathbf z}\rangle\Big)\hspace{0.5mm}\mu(d{\mathbf z})ds \Big\}
$$
\begin{equation}\label{e:exp(int_psi(f)ds)}
= \exp \Big\{\int_{\bbR} \psi\Big(\sum^{n}_{k=1} g_{t_k}(s)^*{\mathbf u}_k \Big) ds \Big\}, \quad {\mathbf u}_1,\hdots,{\mathbf u}_n \in \bbR^p.
\end{equation}
In \eqref{e:exp(int_psi(f)ds)}, the L\'{e}vy symbol $\psi$  can be expressed as
\begin{equation}\label{e:Levy_symbol_finite_2nd_moment}
\psi({\mathbf u}) = \int_{\bbR^p} (e^{\imag \langle{\mathbf u}, {\mathbf z}\rangle}-1-\imag \langle{\mathbf u}, {\mathbf z}\rangle )\hspace{0.5mm}\mu(d {\mathbf z}),\quad {\mathbf u} \in \bbR^p.
\end{equation}
In particular, the random vectors $X(t)$, $t \in \bbR$, are ID (cf.\ Samorodnitsky \cite{samorodnitsky:2016}, Theorem 3.3.2, $(ii)$).
\begin{example}\label{ex:b-tilde}
For the reader's convenience, we explicitly show how $\widetilde{\mathcal M}(dx)$ and ${\mathcal M}(ds)$ can be constructed to encompass both Gaussian and purely non-Gaussian instances of ofLm. For rhofLm, consider a $\C^{p}$-valued independently scattered ID random measure on $\bbR$ given by
$$
{\mathfrak M(dx) := \mathfrak B(dx) + \int_{\C^{p}} \mathbf z \widetilde N(dx, d\mathbf z)} =  \Re(\mathfrak M(dx)) + \imag \Im(\mathfrak M(dx))=:{\mathcal M}^{(1)}(dx) + \imag {\mathcal M}^{(2)}(dx),
$$
where $\mathfrak B(dx):=\mathfrak B^{(1)}(dx)+ \imag \mathfrak B^{(2)}(dx)$ is independent of $\widetilde N( dx, d\mathbf z)$ and $\mathfrak B^{(\ell)}(dx)$, $\ell=1,2$, are $\R^p$-valued independently-scattered Gaussian random measures on $\bbR$. We may then define
$$
\widetilde {\mathcal M}(dx) = \big(\widetilde{\mathcal M}^{(1)}(dx) + \widetilde{\mathcal M}^{(1)}(-dx)\big) + \imag \big(\widetilde{\mathcal M}^{(2)}(dx) - \widetilde{\mathcal M}^{(2)}(-dx)\big)
$$
(n.b.: $\widetilde{\mathcal M}(dx)$ is not independently scattered).  When $\widetilde N(ds,d\mathbf z)$ is absent ($\mathfrak M(dx)\equiv \mathfrak B(dx)$), the resulting measure $\widetilde{\mathcal M}(dx) =:\widetilde B(dx)$ is Gaussian. Assuming properness, the corresponding process \eqref{e:X-tildeH(t)} is an ofBm. Without loss of generality, we may further assume $\mathfrak B^{(1)}(dx),\mathfrak B^{(2)}(dx)$ are independent and take $\bbE \widetilde B(dx) \widetilde B(dx)^* = dx\times I$ (note that a slightly different construction of $\widetilde{ B}(dx)$ -- but one that is equivalent for representing ofBm -- is given in Didier and Pipiras \cite{didier:pipiras:2011}; see also Samorodnitsky and Taqqu \cite{samorodnitsky:taqqu:1994}, Section 7.2.2). When $\mathfrak M(dx)$ is purely non-Gaussian $(\mathfrak B(dx)\equiv \mathbf 0)$, one recovers rhofLm as in Definition \ref{def:rhofLm}.

To encompass both Gaussian and non-Gaussian instances of maofLm, one can take $\mathcal M(ds) = B(ds) + \int_{\bbR^p} \mathbf z \widetilde N(ds,d\mathbf z)$, where $B(ds)$ is an $\bbR^p$-valued independently scattered Gaussian random measure on $\bbR$ independent of $N(ds,d\mathbf z)$, without loss of generality satisfying $\E B(ds) B(ds)^* = ds\times I$. In this case, when $\widetilde N$ absent from $\mathcal M$, assuming properness, the corresponding process \eqref{e:XH(t)} is an ofBm.
\end{example}

\section{Proofs: Section \ref{s:fLm}}

\noindent {\sc  Proof of Theorem \ref{t:rhofLm_maofLm_integ_repres_chf}}: Statement $(i)$ is a consequence of the facts that $\widetilde{g}_t \in L^2_{\textnormal{Herm}}(\bbR)$, $g_t \in L^2(\bbR,M(p,\bbR))$, $t \in \bbR$ (cf.\ Didier and Pipiras \cite{didier:pipiras:2011}).

In regard to $(ii)$, it results from the dominated convergence theorem that the covariance functions of both $\widetilde{X}_H$ and $X_H$ are continuous. Therefore, both processes are stochastically continuous.

As for $(iii)$, relation \eqref{e:rhofLm_chf} is a consequence of expression \eqref{e:chf_harm} for the characteristic function of stochastic integrals of the form \eqref{e:Xtilde(t)=int_ft(omega)M(domega)}. Expression \eqref{e:chf_harm_gt} now follows from the fact that, for ${\mathbf z} = {\mathbf z}_1 + \imag {\mathbf z}_2 \in \bbC^p$, and $j = 1,\hdots,m$,
$$
\Re\Big\{\big(\Re g_{t_j}(x)+ \imag \Im g_{t_j}(x)\big){\mathbf z} \Big\}
= \Re\Big\{\big[\Re g_{t_j}(x) {\mathbf z}_1 - \Im g_{t_j}(x) {\mathbf z}_2 + \imag \hspace{0.5mm}(\Re g_{t_j}(x){\mathbf z}_2 + \Im g_{t_j}(x){\mathbf z}_1)\big]\Big\}
$$
$$
=  \Re g_{t_j}(x){\mathbf z}_1 - \Im g_{t_j}(x){\mathbf z}_2 .
$$
In turn, relation \eqref{e:mafLm_chf} is a consequence of \eqref{e:exp(int_psi(f)ds)}. Moreover, $\widetilde{X}_H$ and $X_H$ are cov.o.s.s.\ as a consequence of the isometry relations \eqref{e:E_int_isometry_harm_f(xi)_cross} and \eqref{e:E_int_isometry}, as well as of the scaling properties
\begin{equation}\label{e:gtilde_ct(x)=c^(D+I)gtilde_t(x)}
\widetilde{g}_{ct}(x) = c^{D+I} \widetilde{g}_t(cx) \quad \textnormal{a.e.}, %
\end{equation}
\begin{equation}\label{e:g_ct(cs)=c^D*g_t(s)}
g_{ct}(cs) =  {\mathcal F}^{-1} \big(c^{-1}\widetilde{g}_{ct}(c^{-1} \cdot )\big)(s) = c^{D}{\mathcal F}^{-1} \big(\widetilde{g}_{t}\big)(s)= c^{D} g_{t}(s) \quad \textnormal{a.e.},
\end{equation}
for any $c > 0$. This establishes $(iii)$.

We now turn to $(iv)$. First consider $X_H$, and let $H= P J_H P^{-1}$ be the Jordan decomposition of $H$.  Define $H' = Q H Q^{-1}$, and observe that $\Re \eig(H')\in(0,1)$. Expression \eqref{e:E_int_isometry} and Parseval-Plancherel imply $\E X_H(s) X_H(t)^* = \int_{\bbR} g_s(u) Q Q^* g_t(u)^* \hspace{0.5mm}du= \int_{\bbR} \widetilde  g_s(x)Q Q^* \widetilde g_t(x)^* \hspace{0.5mm}dx$, which is the covariance function of an ofBm with parameters $H'$ and $QA$ in its harmonizable representation. This establishes the claim.  The claim under \eqref{e:Re_eig(H)=(0,1)-(1/2)} follows, since in this case $Q=I$.  The statements for $\widetilde{X}_H$ follow similarly based on expression \eqref{e:E_int_isometry_harm_f(xi)_cross}. %

In regard to $(v)$, for any $t,h \in \bbR$,
$$
X_{H}(t+h)- X_{H}(h) =
\int_{\bbR} \big( g_{t+h}(s) - g_{h}(s)\big) {\mathcal M}(ds)
$$
$$
= \int_{\bbR} \big(  \{(t+h-s)^{D}_+ -(h-s)^{D}_+ \}M_+ + \{(t+h-s)^{D}_- -(h-s)^{D}_-\} M_-\big) {\mathcal M}(ds)
$$
\begin{equation}\label{e:XH_strictly_station_increm}
\stackrel{\textnormal{f.d.d.}}= \int_{\bbR} \big(  \{(t-s')^{D}_+ -(-s')^{D}_+ \}M_+ + \{(t-s')^{D}_- -(-s')^{D}_-\} M_-\big) {\mathcal M}(ds).
\end{equation}
In \eqref{e:XH_strictly_station_increm}, the equality of f.d.d.\ follows from a change of variable $h-s = s'$ in the characteristic function for
$$
\big( X_{H}(t+h_1)- X_{H}(h_1),\hdots,X_{H}(t+h_m)- X_{H}(h_m) \big), \quad h_1,\hdots,h_m  \in \bbR,
$$
which in turn stems from expression \eqref{e:exp(int_psi(f)ds)}. This shows that the maofLm $X_{H}$ has strict-sense stationary increments. That $\widetilde X_H$ has wide-sense stationary increments is a direct consequence of \eqref{e:E_int_isometry_harm_f(xi)_cross} and \eqref{e:E_int_isometry}.  Moreover, under condition \eqref{e:Mtilde(dxi)=e^(i*theta)Mtilde(dxi)}, for any $t,h \in \bbR$,
$$
\widetilde{X}_{H}(t+h)- \widetilde{X}_{H}(h) =
\int_{\bbR} \big( \widetilde{g}_{t+h}(x) - \widetilde{g}_{h}(x)\big) \widetilde{{\mathcal M}}(dx)
$$
$$
= \int_{\bbR} \Big( \frac{e^{\imag t  x}-1}{\imag x}\Big)\{x^{-D}_+ A+x^{-D}_- \overline{A}\} e^{\imag h x}\widetilde{{\mathcal M}}(dx)
$$
$$
\stackrel{\textnormal{f.d.d.}}= \int_{\bbR} \Big( \frac{e^{\imag t  x}-1}{\imag x}\Big)\{x^{-D}_+ A+x^{-D}_- \overline{A}\} \widetilde{{\mathcal M}}(dx) = \widetilde{X}_{H}(t).
$$
This shows that the rhofLm $\widetilde{X}_{H}$ has strict-sense stationary increments. This establishes $(v)$.

We now show $(vi)$. By means of contradiction, suppose $X_H$ is o.s.s. Therefore, relation \eqref{e:def_ss} holds for $X_H$ and some matrix $\mathcal H$. So, fix $c > 0$ and $t \neq 0$. Then,
$$
\exp\Big\{ \int_{\bbR}\int_{\bbR^p} \Big( e^{\imag {\mathbf u}^* g_{ct}(s){\mathbf z} }-1 -\imag {\mathbf u}^* g_{ct}(s){\mathbf z} \Big) \mu(d{\mathbf z}) ds \Big\}
$$
\begin{equation}\label{e:chf_s.s._by_contradiction}
= \exp\Big\{ \int_{\bbR}\int_{\bbR^p} \Big( e^{\imag {\mathbf u}^* c^{\mathcal H} g_{t}(s){\mathbf z} }-1 -\imag {\mathbf u}^* c^{\mathcal H} g_{t}(s){\mathbf z} \Big) \mu(d{\mathbf z}) ds \Big\}, \quad  {\mathbf u} \in \bbR^p.
\end{equation}
Define the measurable functions $H_1(s,{\mathbf z}) = g_{ct}(s){\mathbf z}$ and $H_2(s,{\mathbf z}) = c^{\mathcal H} g_{t}(s){\mathbf z}$, and consider the product measure $\mu \otimes \eta$. Now define the induced measures on $\R^p$ via
$$
\nu_i(d{\mathbf y})= [ H_i \hspace{1mm}(\mu \otimes \eta)](d{\mathbf y})
= (\mu \otimes \eta)[H^{-1}_i (d{\mathbf y})], \quad i=1,2.
$$
By a change of measures, we can rewrite \eqref{e:chf_s.s._by_contradiction} as
$$
\exp\Big\{ \int_{\bbR^{p}} \Big( e^{\imag \langle{\mathbf u},{\mathbf y}\rangle }-1 -\imag \langle{\mathbf u}, {\mathbf y}\rangle \Big) \nu_1(d{\mathbf y}) \Big\}
= \exp\Big\{ \int_{\bbR^{p}}\Big( e^{\imag \langle{\mathbf u},{\mathbf y}\rangle }-1 -\imag \langle{\mathbf u}, {\mathbf y} \rangle\Big) \nu_2(d{\mathbf y}) \Big\}, \quad {\mathbf u} \in \bbR^p.
$$
By the measure-theoretic convention $0 \times \infty = 0$,  we arrive at
$$
\exp\Big\{ \int_{\bbR^{p}\backslash\{{\mathbf 0}\}} \Big( e^{\imag \langle{\mathbf u},{\mathbf y}\rangle }-1 -\imag \langle{\mathbf u}, {\mathbf y}\rangle \Big) \nu_1(d{\mathbf y}) \Big\}
= \exp\Big\{ \int_{\bbR^{p}\backslash\{{\mathbf 0}\}}\Big( e^{\imag \langle{\mathbf u},{\mathbf y}\rangle }-1 -\imag \langle{\mathbf u}, {\mathbf y} \rangle\Big) \nu_2(d{\mathbf y}) \Big\}, \quad {\mathbf u} \in \bbR^p.
$$
By the uniqueness of the L\'{e}vy measure, $\nu_1(B)= \nu_2(B)$ for all $B \in \mathcal B(\R^p \setminus\{\mathbf 0\})$. Equivalently,
\begin{equation}\label{e:equality_Levy_s.s.}
\int_{\bbR}\int_{\bbR^p} 1_{B}\big(g_{ct}(s){\mathbf z}\big) \mu(d{\mathbf z})ds = \int_{\bbR}\int_{\bbR^p} 1_{B}\big(c^{\mathcal H} g_{t}(s){\mathbf z}\big) \mu(d{\mathbf z})ds, \quad B \in {\mathcal B}(\bbR^p \backslash\{{\mathbf 0}\}).
\end{equation}
Note that the kernel $g_{t}$ satisfies the scaling relation \eqref{e:g_ct(cs)=c^D*g_t(s)}.
By a change of variable $s = c w$ and \eqref{e:g_ct(cs)=c^D*g_t(s)} applied to the integral $\int_{\bbR}\int_{\bbR^p} 1_{B}\big(g_{ct}(s){\mathbf z}\big) \mu(d{\mathbf z})ds$, we can rewrite relation \eqref{e:equality_Levy_s.s.} as
\begin{equation}\label{e:nu*_scaling}
c \int_{\bbR}\int_{\bbR^p} 1_{c^{(1/2)I-H}B}\Big(g_{t}(w){\mathbf z}\Big) \mu(d{\mathbf z})dw = \int_{\bbR}\int_{\bbR^p} 1_{c^{-\mathcal H}B}\Big(g_{t}(s){\mathbf z} \Big) \mu(d{\mathbf z})ds.
\end{equation}
Let
\begin{equation}\label{e:nu*}
\nu_*(B) =\int_{\bbR}\int_{\bbR^p} 1_{B}\Big(g_{t}(s){\mathbf z} \Big) \mu(d{\mathbf z})ds, \quad B \in {\mathcal B}(\bbR^{p}\setminus\{\mathbf 0\}),
\end{equation}
and set $\nu_*(\{\mathbf 0\}):=0$. Note that
$$
\int_{\R^p} \|\mathbf y\|^2 \nu_*(d\mathbf y) =  \int_{\R}\int_{\R^p}  \|g_t(s) \mathbf z\|^2 \mu(d\mathbf z)ds\leq \int_{\R}\|g_t(s)\|^2 ds \cdot \int_{\R^p}  \| \mathbf z\|^2 \mu(d\mathbf z)<\infty.
$$
Based on $\nu_*$, we can rewrite \eqref{e:nu*_scaling} as
\begin{equation}\label{e:c_nu*(c^((1/2)I-H)B)=nu*(c^{-H}B)}
c \hspace{0.5mm}\nu_*(c^{(1/2)I-H}B) = \nu_*(c^{-\mathcal H} B).
\end{equation}
Fix $B_0 \in {\mathcal B}(\bbR^p)$. Then, for $B(c):=c^{H-(1/2)I}B_0 \in {\mathcal B}(\bbR^{p})$, \eqref{e:c_nu*(c^((1/2)I-H)B)=nu*(c^{-H}B)} implies that $c \hspace{0.5mm}\nu_*(B_0) = \nu_*(c^{-\mathcal H} c^{ H -(1/2)I}B_0)$, i.e.,
\begin{equation}\label{e:cnu*(dy)=nu*(c^(-1/2I)dy)}
c \hspace{0.5mm}\nu_*(d{\mathbf y}) = \nu_*(c^{-\mathcal H} c^{ H -(1/2)I}d{\mathbf y}).
\end{equation}
Starting from the L\'{e}vy measure $\nu_*$, we can define a L\'{e}vy process $L_*$ by means of the characteristic function
\begin{equation}\label{e:L*_chf}
\bbE e^{\imag \langle\mathbf{u}, L_*(t)\rangle} = \exp\Big\{t\hspace{1mm} \int_{\bbR^{p}}
\Big( e^{\imag \langle{\mathbf u},{\mathbf y}\rangle }-1 -\imag \langle{\mathbf u}, {\mathbf y}\rangle \Big) \nu_*(d{\mathbf y}) \Big\}.
\end{equation}
In particular, $L_*(0) = 0$, $L_*$ has finite second moment and $L_*(t)$ is proper for $t \neq 0$. By \eqref{e:cnu*(dy)=nu*(c^(-1/2I)dy)},
\begin{equation}\label{e:chf_Levy}
\bbE e^{\imag \langle\mathbf{u}, L_*(ct)\rangle} = \bbE e^{\imag \langle\mathbf{u}, c^{\mathcal H} c^{(1/2)I-H} L_*(t)\rangle}, \quad c > 0.
\end{equation}
 Since $L_*$ is a L\'{e}vy process, relation \eqref{e:chf_Levy} implies that, for every $c > 0$, there exists a linear operator $B(c)$ and a vector $b(c)$ -- namely, $B(c) = c^{\mathcal H} c^{(1/2)I-H}$ and $b(c) = {\mathbf 0}$ -- such that $\{ L_*(ct)\}_{t\in\bbR} \stackrel {\text{f.d.d.}} = \{ B(c) L_*(t)+ b(c)\}_{t\in\bbR}$. Therefore, since $L_*$ is proper, Theorem 1 in Hudson and Mason \cite{hudson:mason:1982} shows that there exist a matrix $H'$ and a nonrandom function $d:[0,\infty)\to \bbR^p$ such that $\{ L_*(ct)\}_{t\in\bbR} \stackrel {\text{f.d.d.}} = \{ c^{H'}L_*(t)+d(c)\}_{t\in\bbR}$. Furthermore, since $L_*$ is stochastically continuous, Theorem 7 in Hudson and Mason \cite{hudson:mason:1982} implies $d\equiv 0$, and that $L_*(1)$ is operator-stable (with exponent $H'$). However, $L_*$ has finite second moment, which in turn implies $L_*$ must be Gaussian. This contradicts the fact that, by \eqref{e:L*_chf}, $L_*$ has no Gaussian component. Hence, maofLm $X_H$ is not o.s.s., as claimed. The case for rhofLm $\widetilde X_H$ is analogous.

We now turn to $(vii)$. Without loss of generality, we can rewrite $\Re h_1 \leq \ldots \leq \Re h_p$. Let $d_1 = \Re h_1-\frac{1}{2}$. Again without loss of generality, let $0\leq s<t\leq 1$, and write $r=t-s$.  Recall that the Frobenius inner product $\langle \cdot,\cdot\rangle_F$ of two matrices $A,B\in M(p,\C)$ is given by $\langle A ,B\rangle_F=\tr(A^*B)$, and write $\|\cdot \|_F$ for the corresponding norm. Also, write $H=P J_H P^{-1}$ for the Jordan decomposition of $H$, and recall that for any Jordan block $J_{h}$ corresponding to $h\in \text{eig}(H)$ of size $k$,
$$
r^{J_h} = \begin{pmatrix}
r^h & 0 & 0 & \dots & 0\\
(\log r) r^h & r^h & 0 & \dots & 0\\
\frac{(\log r)^2}{2!}r^h & (\log r) r^h & r^h & &\\
\vdots  & \vdots &  &\ddots &\\
\frac{(\log r)^{k-1}}{(k-1)!}r^h& \frac{(\log r)^{k-2}}{(k-2)!}r^h & \ldots & (\log r)r^h & r^h
\end{pmatrix}, \quad r > 0.
$$
By stationarity of the increments of maofLm,
$$
\E \| X_H(t)-X_H(s)\|^2 = \E \| X_H(r)\|^2 = \tr \big( \E X_H(r) X_H(r)^*\big)
$$
$$
=\tr \big( r^H\E X_H(1) X_H(1)^* r^{H^*}\big)= \tr \big( r^{H^*}r^H \E X_H(1) X_H(1)^*\big)
$$
$$
\leq \| r^{H^*}r^H\|_F \|  \E X_H(1) X_H(1)^*\|_F \leq C \| r^{H}\|_F^2 \leq C'  \| r^{J_H}\|_F^2
$$
$$
\leq
  C''\max_{h\in \text{eig}(H)}\| r^{J_h}\|_F^2 \leq  C''' ({|\log r|^{2(p-1)}} \vee 1) r^{2\min \{\Re h_1,\ldots,\Re h_p\}} {\leq C'''' r^{2 d_1 +1-\varepsilon}},
$$
for every $\varepsilon>0$. Hence, by the Kolmogorov-$\breve{\textnormal{C}}$entsov theorem (e.g., Kallenberg \cite{kallenberg:2006}, Theorem 2.23), $X$ has a modification that is a.s.\ locally $\gamma$-H\"older continuous for each $\gamma\in (0,  d_1)$, as claimed. The statement also holds for rhofLm in view of its wide-sense stationary increments (see statement $(v)$). $\Box$\\

\noindent {\sc Proof of Proposition \ref{p:XH_harm_YH_ma}}: Both statements are a consequence of the Parseval-type relations \eqref{e:parseval1}, \eqref{e:parseval2} and Proposition \ref{p:spec}. $\Box$\\

\noindent {\sc Proof of Proposition \ref{p:maofLm->BH(t)_rhofLm->BH(t)}}: We begin with $(i)$. For any $t \neq 0$, let $g_{t}$ be as in \eqref{e:ft(s)=F^(-)(gt(s))}, and recall the scaling relation \eqref{e:g_ct(cs)=c^D*g_t(s)} the kernel $g_{t}$ satisfies. For $n \in \bbN$, let $t_1,\ldots,t_n\in \R$. By Theorem \ref{t:rhofLm_maofLm_integ_repres_chf}, $(i)$, the joint characteristic function of $X_H(t_1),\ldots,X_H(t_n)$ is given by
$$
\E \exp \Big \{\imag \sum_{j=1}^n\langle {\mathbf u}_j, X_H(t_j)\rangle \Big \}=  \exp   \Big\{\int_\R \psi\Big( \sum_{j=1}^n g_{t_j}(s)^*{\mathbf u}_j\Big)ds\Big\},
$$
where the L\'{e}vy symbol $\psi$ is as in \eqref{e:Levy_symbol_finite_2nd_moment}. Now consider the collection of rescaled vectors
$$
c^{-H}X_H(ct_1),c^{-H}X_H(ct_2),\ldots,c^{-H}X_H(ct_n).
$$
Then, their joint characteristic function is given by
$$
\E e^{ \imag\sum_{j=1}^n\langle {\mathbf u}_j,{c^{-H}X_H(ct_j)}\rangle}=  \E e^{ \imag\sum_{j=1}^n\langle c^{-H^*}{\mathbf u}_j,X_H(ct_j)\rangle}=  \exp \Big\{ \int_\R \psi\Big( \sum_{j=1}^n g_{ct_j}(s)^*c^{-H^*}{\mathbf u}_j\Big)ds \Big\}
$$
$$
\stackrel{cv=s}{=}  \exp \Big\{ \int_\R \psi\Big( \sum_{j=1}^n g_{ct_j}(cv)^*c^{-H^*}{\mathbf u}_j\Big)cdv \Big\} =   \exp \Big\{\int_\R \psi\Big(  \sum_{j=1}^n g_{t_j}(v)^*c^{-\frac{1}{2}I}{\mathbf u}_j\Big)c \hspace{1mm} dv \Big\} $$
\begin{equation*}\begin{split}
=\exp \Big\{\int_\R\int_{\R^p}\Big(\exp \Big\{\imag\Big\langle \sum_{j=1}^ng_{t_j}&(v)^*c^{-\frac{1}{2}I}{\mathbf u}_j,{\mathbf z}\Big\rangle \Big\}\\
& - 1 - \imag\Big\langle{ \sum_{j=1}^n g_{t_j}(v)^*c^{-\frac{1}{2}I}{\mathbf u}_j,{\mathbf z}}\Big\rangle \Big)c\hspace{1mm}\mu(d{\mathbf z})dv \Big\}
\end{split}\end{equation*}
\begin{equation}\begin{split}\label{e:c1/2_prelim}
=\exp \Big\{ \int_\R\int_{\R^p}\Big(\exp \Big\{ c^{-1/2}\imag\Big\langle \sum_{j=1}^n g_{t_j}&(v)^*{\mathbf u}_j,{\mathbf z}\Big\rangle \Big\}\\
&- 1 - \imag\Big\langle{ c^{-1/2}\sum_{j=1}^n g_{t_j}(v)^*{\mathbf u}_j,{\mathbf z}}\Big\rangle \Big)c \hspace{1mm}\mu(d{\mathbf z})dv \Big\}.
\end{split}\end{equation}
Note that $h_c(y):= c\hspace{0.5mm}(e^{\imag  c^{-1/2}y}-1- \imag c^{-1/2}y) \sim -\frac{1}{2}y^2$ as $c\to\infty$ and that $|h_c(y)|\leq y^2$ for all $y\in\R$. Writing $\xi(v,{\mathbf z})=\big\langle{ \sum_{j=1}^ng_{t_j}(v)^*{\mathbf u}_j,{\mathbf z}}\big\rangle$, the integrand $h_c(\xi(v,{\mathbf z}))$ in expression \eqref{e:c1/2_prelim} satsifies
$$
|h_c(\xi(v,{\mathbf z}))|\leq  \xi(v,{\mathbf z})^2 =\sum_{j,k=1}^n  {{\mathbf u}^*_jg_{t_j}(v){\mathbf z}{\mathbf z}^*g_{t_k}^*(v)u_k},
$$
which is integrable with respect to $\mu(d{\mathbf z})dv$ since each $g_{t_j}(\cdot)\in L^2(\R,M(p,\R))$ and $\|\int_{\R^p} {\mathbf z}{\mathbf z}^*\mu(d{\mathbf z})\|=1$.
Thus, by the dominated convergence theorem,
as $c\to\infty$, \eqref{e:c1/2_prelim} converges to
$$
\exp \Big \{-\frac{1}{2} \int_\R\int_{\R^p}\Big(\sum_{j=1}^n\langle{g_{t_j}(v)^*{\mathbf u}_j,{\mathbf z}}\rangle\Big)^2 \hspace{0.5mm} \mu(d{\mathbf z})dv\Big\}
$$
$$
=\exp \Big \{-\frac{1}{2} \int_\R\int_{\R^p}\sum_{j,k=1}^n{{\mathbf u}^*_jg_{t_j}(v) {\mathbf z}{\mathbf z}^* g_{t_k}^*(v){\mathbf u}_k} \hspace{0.5mm} \mu(d{\mathbf z})dv\Big\}
$$
\begin{equation}\label{e:chf_ofbm}
=\exp \Big \{-\frac{1}{2} \int_\R\sum_{j,k=1}^n{{\mathbf u}^*_jg_{t_j}(v)g_{t_k}^*(v){\mathbf u}_k} \hspace{0.5mm}dv\Big\}
\end{equation}
where we use condition \eqref{e:int_zz*=I}.  Note that \eqref{e:chf_ofbm} is equal to $\exp \Big\{ -\frac{1}{2} {\mathbf u}^*\Sigma_{B_H}{\mathbf u}\Big\}$, where $ \Sigma_{ B_H}$ is the $pn \times pn$ block matrix
$$
 \Sigma_{B_H} = \Big( \hspace{1mm}\int_{\bbR} g_{t_i}(s) g_{t_j}(s)^*ds \hspace{1mm}\Big)_{i,j=1,\hdots,n}.
 $$
Hence, \eqref{e:chf_ofbm} is the characteristic function of an ofBm at times $t_1,\ldots,t_n$.

For $(ii)$, for any $t \neq 0$, let $\widetilde{g}_{t}$ be as in \eqref{e:fLm_filter_harm}. For $s \in \bbR$ and $\varepsilon > 0$, Theorem \ref{t:rhofLm_maofLm_integ_repres_chf}, $(iv)$, implies that
$$
\widetilde{X}_{H}(s + \varepsilon t)- \widetilde{X}_{H}(s) \stackrel{\textnormal{f.d.d.}}= \widetilde{X}_{H}(\varepsilon t).
$$
So, fix $n \in \bbN$ and let ${\mathbf u}_1,\hdots,{\mathbf u}_n \in\bbR^p$, $t_1,\hdots,t_n \in\bbR$. Note that $\widetilde{g}_{\varepsilon t}(\varepsilon^{-1}x) = x^{D+I}\widetilde{g}_{t}(x)$. By Proposition \eqref{t:rhofLm_maofLm_integ_repres_chf}, $(i)$,
$$
\bbE e^{\imag \sum^{N}_{j=1}{\mathbf u}^*_j \varepsilon^{-H}\widetilde{X}_H(t_j)}
= \exp   \Big\{\int_\R \widetilde\psi\Big( \sum_{j=1}^n \widetilde{g}_{\varepsilon t_j}(x)^*\varepsilon^{-H^*}{\mathbf u}_j\Big)dx\Big\}
$$
\begin{equation}\label{e:phitilde_changedvar}
\stackrel{y = \varepsilon x}{=}  \exp \Big\{ \int_\R \widetilde \psi\Big( \sum_{j=1}^n \widetilde{g}_{ t_j}(y \varepsilon^{-1})^*\varepsilon^{-H^*}{\mathbf u}_j\Big)\frac{dy}{\varepsilon} \Big\}
=\exp \Big\{ \int_\R \widetilde \psi\Big(\sum_{j=1}^n \widetilde{g}_{ t_j}(y)^*\varepsilon^{\frac{1}{2}I}{\mathbf u}_j\Big)\frac{dy}{\varepsilon} \Big\},
\end{equation}
where $\widetilde \psi$ is given by \eqref{e:int_varphi=2Re(gz)_chf_Levy_symbol}. Recast ${\mathbf z} = {\mathbf z}_1 + \imag {\mathbf z}_2$. By a similar dominated convergence argument as in part $(i)$, and using the fact that $\widetilde{g}_{t_j} \in L^{2}_{\textnormal{Herm}}(\bbR)$, $j=1,\hdots,n$, as $\varepsilon \to 0$ expression \eqref{e:phitilde_changedvar} converges to
$$
\exp\Big\{-2 \Big(\int_{\bbR \times \bbC^p} \sum_{j,k=1}^n {\mathbf u}^*_j\widetilde{g}_{t_j}(y){\mathbf z}_1 {\mathbf z}^*_1 \widetilde{g}_{t_k}^*(y){\mathbf u}_k  +\sum_{j,k=1}^n {\mathbf u}^*_j\widetilde{g}_{t_j}(y){\mathbf z}_2 {\mathbf z}^*_2 \widetilde{g}_{t_k}^*(y){\mathbf u}_k  \Big)\mu(d{\mathbf z})dy \Big\}.
$$
By using condition \eqref{e:int_RezRez*=sigma^2_I=int_ImzImz*}, we arrive at
$$
\exp \Big\{-\frac{1}{2} \int_\R \sum_{j,k=1}^n{{\mathbf u}^*_j\widetilde{g}_{t_j}(y)  \widetilde{g}_{t_k}^*(y){\mathbf u}_k}\hspace{1mm} dy\Big\}.
$$
This establishes $(ii)$. $\Box$\\

\noindent {\sc Proof of Proposition \ref{p:lass2}}: Whenever convenient, we write $D=H-(1/2)I$ (see \eqref{e:D=H-(1/2)I}). First note that, since $\Re \lambda_p(B)<1$, the expression \eqref{e:Levy_symbol_finite_first_moment} is well-defined and corresponds to the L\'evy symbol of a full operator-stable distribution in $\R^p$ (see Section \ref{s:tempered_oper-stable_Levy_measures}).  Furthermore,
that the process \eqref{e:opstab_MA} is well defined by Theorem 5.4 in Maejima and Mason \cite{maejima:mason:1994} (see also Theorem 4.2 in Kremer and Scheffler \cite{kremer:scheffler:2019}), since $\Re\lambda_1(\widetilde H - B) + \Re\lambda_1(B) = \Re\lambda_1(D)+\Re\lambda_1(B) >0$ and $\Re\lambda_p(\widetilde H - B -I) + \Re\lambda_p(B) = \Re \lambda_p(D)-1+\Re\lambda_p(B)<0$ {(n.b.: there is a typo in the original statement of Theorem 5.4 in Maejima and Mason \cite{maejima:mason:1994}; in the notation of that paper, their assumption should read $\Lambda_{D-B-I}+\Lambda_B<0$)}. We now proceed as in the proof of Proposition \ref{p:maofLm->BH(t)_rhofLm->BH(t)}, $(i)$. In fact, fix $n \in \bbN$ and $t_1,\hdots,t_n \in \bbR$, as well as the vectors ${\mathbf u}_1,\hdots,{\mathbf u}_n \in \bbR^p$. Let $\varepsilon > 0$. Consider $g_{t_j}$ as in \eqref{e:ft(s)=F^(-)(gt(s))}, $j = 1,\hdots,n$, and recall the scaling relation \eqref{e:g_ct(cs)=c^D*g_t(s)}. Then, by stationary increments and \eqref{e:H-tilde_1=D+B},
$$
 \E e^{ \imag\sum_{j=1}^n\langle {\mathbf u}_j,{\varepsilon^{-\widetilde H_1} \big( X_H(s+\varepsilon t_j)-X_H(s)\big)}\rangle}= \E e^{ \imag\sum_{j=1}^n\langle {\mathbf u}_j,{\varepsilon^{-\widetilde H_1} X_H(\varepsilon t_j)}\rangle}
$$
$$
 = \exp\Big\{\int_\R \psi\Big( \sum_{j=1}^n g_{\varepsilon t_j}(s)^*\varepsilon^{-\widetilde H_1^*}{\mathbf u}_j\Big)ds\Big\}
\stackrel{\varepsilon v=s}{=} \exp \Big\{ \int_\R \psi\Big( \sum_{j=1}^n g_{\varepsilon t_j}(\varepsilon v)^*\varepsilon ^{-\widetilde H_1^*}{\mathbf u}_j\Big)\varepsilon dv\Big\}
$$
$$
=\exp \Big\{\int_\R \psi\Big( \sum_{j=1}^n g_{t_j}( v)^*\varepsilon ^{(D-\widetilde H_1)^*}{\mathbf u}_j\Big)\varepsilon dv \Big\}
=\exp \Big\{\int_\R \psi\Big( \sum_{j=1}^n \varepsilon ^{-B^*}g_{t_j}( v)^*{\mathbf u}_j\Big)\varepsilon dv\Big\}
$$
\begin{equation*}\begin{split}
=\exp \Big\{\int_\R\int_{S_0}\int_{\R_+}\Big(\exp\Big\{\imag\Big\langle \sum_{j=1}^n g_{t_j}&(v)^*{\mathbf u}_j,(r/\varepsilon)^B\boldsymbol\theta\Big\rangle\Big\}\\
 &- 1 - \imag\Big\langle{ \sum_{j=1}^n g_{t_j}(v)^*{\mathbf u}_j,(r/\varepsilon)^B\boldsymbol\theta}\Big\rangle \Big)\varepsilon q(r,\boldsymbol\theta)\frac{dr}{r^2}\lambda(d\boldsymbol\theta)dv \Big\}
\end{split}\end{equation*}
\begin{equation*}\begin{split}
\stackrel{\xi=r/\varepsilon}{=}\exp \Big\{\int_\R\int_{S_0}\int_{\R_+}\Big(\exp\Big\{\imag\Big\langle \sum_{j=1}^n g_{t_j}&(v)^*{\mathbf u}_j,\xi^B\boldsymbol\theta\Big\rangle\Big\}\\
& - 1 - \imag\Big\langle{ \sum_{j=1}^n g_{t_j}(v)^*{\mathbf u}_j,\xi^B\boldsymbol\theta}\Big\rangle \Big)q(\varepsilon \xi,\boldsymbol\theta)\frac{d\xi}{\xi^2}\lambda(d\boldsymbol\theta)dv \Big\}
\end{split}\end{equation*}
\begin{equation}\begin{split}\label{e:chf_opstab_MA}
\to\exp \Big\{\int_\R\int_{S_0}\int_{\R_+}\Big(\exp\Big\{\imag\Big\langle \sum_{j=1}^n g_{t_j}&(v)^*{\mathbf u}_j,\xi^B\boldsymbol\theta\Big\rangle\Big\}\\
& - 1 - \imag\Big\langle{ \sum_{j=1}^n g_{t_j}(v)^*{\mathbf u}_j,\xi^B\boldsymbol\theta}\Big\rangle \Big)\frac{d\xi}{\xi^2}\lambda(d\boldsymbol\theta)dv \Big\},
\end{split}\end{equation}
as $\varepsilon\to 0^+$. The limit in \eqref{e:chf_opstab_MA} is a consequence of the dominated convergence theorem and of relation \eqref{e:q(0+,theta)=1}, since $|q(r,\boldsymbol\theta)|\leq 1$. By Proposition 2.17 in Sato \cite{sato:2006}, \eqref{e:chf_opstab_MA} is the characteristic function of \eqref{e:opstab_MA}. Therefore, the limit \eqref{e:ofLm_local_behavior_TOS} of the rescaled finite-dimensional distributions of $X_H$ holds.

For $(ii)$, we first verify the existence of the limiting process  \eqref{e:opstab_harm}  by applying Theorem 2.5 in  Kremer and Scheffler \cite{kremer:scheffler:2019}.  To use the proposition, we view \eqref{e:opstab_harm} as a process in $\R^{2p}$, where we identify
$$
\R^{2p\times 2p}\ni \widetilde g_t(x) \equiv \begin{pmatrix}
{\Re \widetilde{g}_t(x)} & {-\Im \widetilde{g}_t(x)}\\
0 & 0
\end{pmatrix}
$$
(see Proposition 5.10 in Kremer and Scheffler \cite{kremer:scheffler:2017}.)
Fix $t\neq 0$. To apply Theorem 2.5 in  Kremer and Scheffler \cite{kremer:scheffler:2019}, we need to verify
\begin{equation}\label{e:kremer_scheffler_2.5}
\int_{\{x:\|\widetilde g_t(x)\|<R\}} \|\widetilde g_t(x)\|^{{\frac{1}{\Re\lambda_p(B)} - \delta_1}}dx + \int_{\{x:\|\widetilde g_t(x)\|>R\}} \|\widetilde g_t(x)\|^{{\frac{1}{\Re\lambda_1(B)} + \delta_2}} dx <\infty
\end{equation}
for some $R>0$ and appropriate $\delta_1\in(0,\frac{1}{\Re \lambda_p(B))}) ,\delta_2>0$. Recall that $D = H - (1/2)I$. Since $\|\widetilde g_t(x)\|\to0$ as $|x|\to\infty$, and $\|\widetilde g_t(\cdot)\|$ is continuous on $\R\setminus\{0\}$, for $R$ large enough there exists an $\varepsilon>0$ so that the  set $\{x:\|\widetilde g_t(x)\|>R\}\subseteq\{x:|x|<\varepsilon\}$. Further note that for each $\delta>0$ there exists $C>0$ such that
\begin{equation}\label{e:max2}
 \max\{\|\Re\widetilde{g}_{t}(x)\|,\|\Im\widetilde{g}_{t}(x)\|\}\mathbf 1_{\{|x|>\varepsilon\}} \leq C |x|^{\delta-\Re{\lambda_1 (D)}-1}
 $$$$
 \text{max}\{\|\Re\widetilde{g}_{t}(x)\|,\|\Im\widetilde{g}_{t}(x)\|\}\mathbf 1_{\{|x|\leq \varepsilon\}} \leq C |x|^{-\delta-\Re{\lambda_p (D)}},
\end{equation}
(see Theorem 2.2.4 in Meerschaert and Scheffler \cite{meerschaert:scheffler:2001}), where in the second inequality we used the fact that $\frac{e^{\imag t x}-1}{\imag x}$ is bounded for all small $|x|$.

 So, take $\delta>0$ small enough so that
\begin{equation}\label{e:Re_lambdap(B)+Re_lambdap(D)>1+delta}
\Re\lambda_1(D)+ 1 -\Re \lambda_p(B) = \Re\lambda_1(H)+\Big( \frac{1}{2} - \Re\lambda_p(B)\Big) > \delta
\end{equation}
and
\begin{equation}\label{e:Re_lambda1(D)+Re_lambda1(B)<1-delta}
\Re\lambda_p(D) +1 -\Re \lambda_1(B)=\Re\lambda_p(H)+\Big(\frac{1}{2}-\Re \lambda_1(B)\Big)<1-\delta.
\end{equation}
Let $C>0$ be the constant satisfying both inequalities \eqref{e:max2}.
For notational simplicity, write $\rho_1=-(\delta - \Re{\lambda_1 (D)}-1)$ and $\rho_2 = \delta+\Re\lambda_p (D)$.  Now, for any $\delta_1>0$,
\begin{equation}\label{e:max_int_near_0}
\int_{\{|x|>\varepsilon \}}  \max\{\|\Re\widetilde{g}_{t}(x)\|,\|\Im\widetilde{g}_{t}(x)\|\}^{\frac{1}{\Re\lambda_p(B)} - \delta_1}dx\leq C \int_{\{|x|> \varepsilon \}} |x|^{\frac{-\rho_1}{\Re\lambda_p(B)} + \delta_1\rho_1}dx.
\end{equation}
By \eqref{e:Re_lambdap(B)+Re_lambdap(D)>1+delta}, $\frac{-\rho_1}{\Re\lambda_p(B)}= \frac{\delta-\Re\lambda_1 (D)-1}{\Re\lambda_p(B)}<-1$.  By choosing $\delta_1$ so $\rho_1\delta_1$ is small enough, we obtain ${\frac{-\rho_1}{\Re\lambda_p(B)} + \delta_1\rho_1}<-1$, implying the integral \eqref{e:max_int_near_0} is finite.  This shows the first summand in \eqref{e:kremer_scheffler_2.5} is finite, since
$$
\int_{\{x:\|\widetilde g_t(x)\|<R\}} \|\widetilde g_t(x)\|^{{\frac{1}{\Re\lambda_p(B)} + \delta_1}}(\mathbf 1_{\{|x|\leq \varepsilon\}}+\mathbf 1_{\{|x|>\varepsilon\}})dx \leq C'+  \int_{\{|x|> \varepsilon \}} |x|^{\frac{-\rho_1}{\Re\lambda_p(B)} + \delta_1\rho_1}dx
$$
by \eqref{e:max_int_near_0}.
 Now, for any $\delta_2>0$,
\begin{equation}\label{e:max_int_near_inf}
\int_{\{|x|\leq \varepsilon\}}  \max\{\|\Re\widetilde{g}_{t}(x)\|,\|\Im\widetilde{g}_{t}(x)\|\}^{\frac{1}{\Re\lambda_1(B)} + \delta_2}dx\leq C \int_{\{|x|\leq \varepsilon\}}|x|^{\frac{-\rho_2}{\Re\lambda_1(B)} - \delta_2\rho_2}dx.
\end{equation}
By \eqref{e:Re_lambda1(D)+Re_lambda1(B)<1-delta}, we see that $\frac{-\rho_2}{\Re\lambda_1(B)}=\frac{-\delta-\Re \lambda_p(D)}{\Re\lambda_1(B)} >-1 $. Hence, by choosing $\delta_2$ small enough we have $\frac{\rho_2}{\Re\lambda_1(B)} + \delta_2\rho_2>-1$ and  the integral \eqref{e:max_int_near_inf} is also finite.  If we write $\|\cdot\|_{q\times q}$ for the operator norm in $\R^{q\times q}$ then, clearly, $\| \widetilde g_t(x)\|_{2p\times 2p} \leq2\max\{\|\Re\widetilde{g}_{t}(x)\|_{p\times p},\|\Im\widetilde{g}_{t}(x)\|_{p\times p}\}$.  Hence, the conditions of Theorem 2.5 in Kremer and Scheffler \cite{kremer:scheffler:2019} are satisfied, which implies the process \eqref{e:opstab_harm} exists due to Proposition 5.10 in Kremer and Scheffler \cite{kremer:scheffler:2017}. Moreover, by Corollary 5.11(b) in Kremer and Scheffler \cite{kremer:scheffler:2017}, the characteristic function of the candidate limiting process \eqref{e:opstab_harm} at times $t_1,\ldots,t_n$ is given by
\begin{equation}\label{e:chf_opstab_harm}
 \exp \int_\R\int_{\R^{2p}}W\Big(2\hspace{0.25mm} \sum^{n}_{k=1}{\mathbf u}^*_k(\Re\widetilde{g}_{t_k}(y) \mathbf z_1 - \Im \widetilde{g}_{t_k}(y) \mathbf z_2 )  \Big)\Big)\mu_{\widetilde B}(d\mathbf z) dy
 $$
 $$
 =\exp \int_\R\int_{ S_0}\int_{\R_+}W\Big(2\hspace{0.25mm} \sum^{n}_{k=1}{\mathbf u}^*_k(\Re\widetilde{g}_{t_k}(y) \xi^B\boldsymbol \theta - \Im \widetilde{g}_{t_k}(y) \xi^B\boldsymbol \theta )  \Big)\Big)\frac{d\xi}{\xi^2}\lambda(d\boldsymbol \theta)dy
\end{equation}
(see \eqref{e:def_opstab}), where for notational simplicity we used the expression $W(y)=e^{\imag y}-1-\imag y$, $y\in\R$.
Now, to establish the convergence \eqref{e:rhofLm_local_behavior_TOS}, observe that the scaling relation
$$
\widetilde{g}_{ct}(x c^{-1})c^{-\widetilde H_2} = \widetilde g_t(x) c^{H+\frac{1}{2}I-\widetilde H_2} = \widetilde g_t(x) c^{B}
$$
holds. So, the characteristic function of the rescaled vector $(c^{-\widetilde H_2}\widetilde X_H(c t_1),\ldots,c^{-\widetilde H_2}\widetilde X_H(c t_1))^*$ is given by
$$
\exp\Big\{ \int_{\bbR} \int_{\bbC^p}\Big[e^{\imag \hspace{0.25mm}2 \hspace{0.25mm}\Re \big(  \sum^{n}_{k=1}{\mathbf u}^*_k \hspace{0.5mm}c^{-\widetilde H_2}\widetilde{g}_{ct_k}(x){\mathbf z}\big)}- 1 -
\imag \hspace{0.25mm}2\hspace{0.25mm} \Re \Big( \sum^{n}_{k=1}{\mathbf u}^*_k \hspace{0.5mm}c^{-\widetilde H_2}\widetilde{g}_{ct_k}(x){\mathbf z}\Big)\Big]\mu(d{\mathbf z}) dx \Big\}
$$
$$
= \exp\Big\{ \int_{\bbR} \int_{\bbR^{2p}}W\Big(2\hspace{0.25mm} \sum^{n}_{k=1}{\mathbf u}^*_kc^{-\widetilde H_2} (\Re\widetilde{g}_{ct_k}(x) {\mathbf z}_1 - \Im \widetilde{g}_{ct_k}(x) {\mathbf z}_2)  \Big)\mu_{\bbR^{2p}}(d{\mathbf z}) dx  \Big\},
$$
$$
\stackrel{y=cx}= \exp\Big\{ \int_{\bbR} \int_{\bbR^{2p}}W\Big(2\hspace{0.25mm} \sum^{n}_{k=1}{\mathbf u}^*_kc^{-\widetilde H_2} (\Re\widetilde{g}_{ct_k}(yc^{-1}) {\mathbf z}_1 - \Im \widetilde{g}_{ct_k}(yc^{-1}) {\mathbf z}_2)  \Big)\mu_{\bbR^{2p}}(d{\mathbf z}) \frac{dy}c  \Big\},
$$
$$
= \exp\Big\{ \int_{\bbR} \int_{\bbR^{2p}}W\Big(2\hspace{0.25mm} \sum^{n}_{k=1}{\mathbf u}^*_k(\Re\widetilde{g}_{ct_k}(y)c^B{\mathbf z}_1- \Im \widetilde{g}_{ct_k}(y) c^B{\mathbf z}_2)  \Big)\mu_{\bbR^{2p}}(d{\mathbf z}) \frac{dy}c  \Big\},
$$
$$
\stackrel{\eqref{e:def_TalphaS}}=\exp \int_\R\int_{S_0}\int_{\R_+}W\Big(2\hspace{0.25mm} \sum^{n}_{k=1}{\mathbf u}^*_k(\Re\widetilde{g}_{t_k}(y) (cr)^B\boldsymbol \theta - \Im \widetilde{g}_{t_k}(y) (cr)^B\boldsymbol \theta )  \Big)q(r,\boldsymbol \theta)\frac{dr}{r^2}\lambda(d\boldsymbol \theta)\frac{dy}{c}
$$
$$
\stackrel{\xi = cr}=\exp \int_\R\int_{S_0}\int_{\R_+}W\Big(2\hspace{0.25mm} \sum^{n}_{k=1}{\mathbf u}^*_k(\Re\widetilde{g}_{t_k}(y) \xi^B\boldsymbol \theta - \Im \widetilde{g}_{t_k}(y) \xi^B\boldsymbol \theta )  \Big)q(c^{-1}\xi,\boldsymbol \theta)\frac{d\xi}{\xi^2}\lambda(d\boldsymbol \theta)dy
$$
$$
\stackrel{c \to \infty}\to \exp \int_\R\int_{S_0}\int_{\R_+}W\Big(2\hspace{0.25mm} \sum^{n}_{k=1}{\mathbf u}^*_k(\Re\widetilde{g}_{t_k}(y) \xi^B\boldsymbol \theta - \Im \widetilde{g}_{t_k}(y) \xi^B\boldsymbol \theta )  \Big)\Big)\frac{d\xi}{\xi^2}\lambda(d\boldsymbol \theta)dy,
$$
where the limit is again a consequence of the dominated convergence theorem and relation \eqref{e:q(.,theta)->0}. The conclusion follows. $\Box$\\

\section{Proofs: Section \ref{s:time_revers}}

\noindent {\sc  Proof of Theorem \ref{t:maofLm_time-reversibility}}: Let $\{Y_H(t)\}_{t \in \bbR} = \{X_H(-t)\}_{t \in \bbR}$ be the time-reversed process. First note that, if $\{f_t({\boldsymbol \varpi}),t \in \bbR\}= \{g_t(s){\mathbf z},t \in \bbR\}$ is a minimal representation of $X_H$ with respect to ${\mathcal B} \hspace{-1mm}\mod \kappa$, then
\begin{equation}\label{e:f_minus_t_is_minimal}
\{f_{-t}({\boldsymbol \varpi}),t \in \bbR\} \textnormal{ is a minimal representation of $Y_H$ with respect to } {\mathcal B} \hspace{-1mm}\mod \kappa.
\end{equation}
We first show $(ii)\Rightarrow (i)$. Note that, by \eqref{e:M-^(-1)*M+*z=M+^(-1)*M-*z},
$$
g_{-t}(s){\mathbf z} = \Big\{[(-t - s)^{D}_+ - (-s)^{D}_+]M_+ +  [(-t -s)^{D}_- - (-s)^{D}_-]M_- \Big\}{\mathbf z}
$$
\begin{equation}\label{e:g-t(s)_mu(dz)-a.e.}
= [(t + s)^{D}_- - (s)^{D}_-]M_- \hspace{0.5mm}(M^{-1}_+ M_-) {\mathbf z}
+  [(t + s)^{D}_+ - s^{D}_+] M_+ \hspace{0.5mm}(M^{-1}_+ M_-){\mathbf z} \quad \mu(d{\mathbf z})\textnormal{--a.e.}
\end{equation}
So, for any $m \in \bbN$, fix $t_1 < \hdots < t_m$, pick any vectors ${\mathbf u}_1,\hdots,{\mathbf u}_m \in \bbR^p$. By expressions \eqref{e:exp(int_psi(f)ds)} and \eqref{e:g-t(s)_mu(dz)-a.e.}, the finite-dimensional distributions of $\{X_H(-t)\}_{t \in \bbR}$ are given by
$$
\bbE \exp\Big\{ \sum^{n}_{j=1} {\mathbf u}^*_j X_H(-t_j) \Big\} = \exp \Big\{ \int_{\bbR} \int_{\bbR^p} \Big(e^{\imag \sum^{n}_{j=1} {\mathbf u}^*_j g_{-t_j}(s) \hspace{0.25mm}{\mathbf z}}- 1 - \imag \sum^{n}_{j=1} {\mathbf u}^*_j g_{-t_j}(s)\hspace{0.25mm}{\mathbf z} \Big) \hspace{0.5mm}\mu(d{\mathbf z}) \hspace{0.5mm} ds \Big\}
$$
$$
= \exp \Big\{ \int_{\bbR} \int_{\bbR^p} \Big(e^{\imag \sum^{n}_{j=1} {\mathbf u}^*_j g_{t_j}(s') \hspace{0.25mm}{\mathbf z}'}- 1 - \imag \sum^{n}_{j=1} {\mathbf u}^*_j g_{t_j}(s')\hspace{0.25mm}{\mathbf z}' \Big) \hspace{0.5mm}\mu(d{\mathbf z}') \hspace{0.5mm} ds \Big\}
= \bbE \exp\Big\{ \sum^{n}_{j=1} {\mathbf u}^*_j X_H(t_j) \Big\},
$$
where we make the change of variable $(s',{\mathbf z}') = (-s,(M^{-1}_+ M_-){\mathbf z})$ and apply conditions \eqref{e:M-^(-1)*M+*z=M+^(-1)*M-*z} and \eqref{e:mu(matrix*dx)=mu(dz)}. Therefore, $X_H$ is time-reversible. This establishes $(i)$.

Now, we establish $(i)\Rightarrow (ii)$. So, suppose $X_H$ is time-reversible. We first show that \eqref{e:M-^(-1)*M+*z=M+^(-1)*M-*z} holds. In terms of spectral representations, time reversibility means that, for $f_t({\boldsymbol \varpi})$ as in \eqref{e:ft(omega)=gt(s)*x},
$$
\Big\{\int_{\bbR^{p+1}} f_t({\boldsymbol \varpi}) \widetilde{N}(d{\boldsymbol \varpi})\Big\}_{t \in \bbR} \stackrel{{\mathcal L}}= \Big\{\int_{\bbR^{p+1}} f_{-t}({\boldsymbol \varpi}) \widetilde{N}(d{\boldsymbol \varpi})\Big\}_{t \in \bbR}.
$$
By assumption and by \eqref{e:f_minus_t_is_minimal}, both $\{f_t({\boldsymbol \varpi})\}_{t \in \bbR}$ and $\{f_{-t}({\boldsymbol \varpi})\}_{t \in \bbR}$ are minimal representations of $X_H$ on the space $(\bbR^{p+1},{\mathcal B}(\bbR^{p+1}),\kappa)$. Then, Proposition \ref{p:theorem_2.17_KabluchkoStoev} implies that there is a (unique modulo $\kappa$--null sets) mapping
$$
\Phi:\bbR^{p+1} \rightarrow \bbR^{p+1}, \quad {\boldsymbol \varpi} \mapsto (\Phi_1({\boldsymbol \varpi}),\Phi_2({\boldsymbol \varpi})), \quad \Phi_1({\boldsymbol \varpi}) \in \bbR, \hspace{1mm} \Phi_2({\boldsymbol \varpi})\in \bbR^p,
$$
such that, for all $t \in \bbR$,
\begin{equation}\label{e:f(-t)(s)=f_t(Phi(s))_a.e._general}
f_{-t}({\boldsymbol \varpi}) = f_t(\Phi({\boldsymbol \varpi})) \quad  \kappa(d{\boldsymbol \varpi}) \textnormal{--a.e.}
\end{equation}
So, for each $t\in\R$, let
$$
V_t = \big\{{\boldsymbol \varpi}: \eqref{e:f(-t)(s)=f_t(Phi(s))_a.e._general} \textnormal{ holds at } {\boldsymbol \varpi}\big\} \in \mathcal B(\R\times\R^p).
$$
Define $V = \bigcap_{t\in \mathbb Q} V_{t}$.  Observe that
\begin{equation}\label{e:indic(f-t_n(s)=f_{t_n}(Phi(s)))=1_general}
f_{-t}({\boldsymbol \varpi}) = f_{t}(\Phi({\boldsymbol \varpi})), \quad {\boldsymbol \varpi} = (s,{\mathbf z}) \in V,\quad t \in \mathbb Q.
\end{equation}
Now, for notational simplicity, consider the integrand $g_t$ with $s$ in place of $-s$. Fix any
\begin{equation}\label{e:w0=(s0,z0)_s0<0}
{\boldsymbol \varpi}_0 = (s_0,{\mathbf z}_0) \in V, \quad s_0 < 0.
\end{equation}
Let $\{t_n\}_{n\in\N} \subseteq \mathbb Q$ be a sequence such that $t_n \uparrow \infty$. For large enough $n$,
$$
[(-t_n + s_0)^{D}_- - (s_0)^{D}_-]M_- \hspace{1mm}{\mathbf z}_0
$$
$$
= \Big([(t_n+\Phi_1({\boldsymbol \varpi}_0))^{D}_+ - (\Phi_1({\boldsymbol \varpi}_0))^{D}_+]M_+   - (\Phi_1({\boldsymbol \varpi}_0))^{D}_- M_- \Big)\Phi_2({\boldsymbol \varpi}_0),
$$
i.e.,
$$
(-t_n + s_0)^{D}_-  M_- \hspace{1mm}{\mathbf z}_0   - (t_n+\Phi_1({\boldsymbol \varpi}_0))^{D}_+ M_+ \hspace{1mm}\Phi_2({\boldsymbol \varpi}_0)
$$
\begin{equation}\label{e:equiv_filters_at_t_and_-t_compPoi_general}
=  (s_0)^{D}_- M_- \hspace{1mm}{\mathbf z}_0 - (\Phi_1({\boldsymbol \varpi}_0))^{D}_+ M_+ \hspace{1mm}\Phi_2({\boldsymbol \varpi}_0) - (\Phi_1({\boldsymbol \varpi}_0))^{D}_- M_- \Phi_2({\boldsymbol \varpi}_0).
\end{equation}
Consider the Jordan decomposition $D = PJ_DP^{-1}$. The right-hand side of \eqref{e:equiv_filters_at_t_and_-t_compPoi_general} is a constant with respect to $n$. Therefore, after pre-multiplying both sides by $P^{-1}$, we can recast \eqref{e:equiv_filters_at_t_and_-t_compPoi_general} as
\begin{equation}\label{e:equiv_filters_at_t_and_-t_compPoi_general_const}
(-t_n + s_0)^{J_D}_-  P^{-1}M_- \hspace{1mm}{\mathbf z}_0  - (t_n+\Phi_1({\boldsymbol \varpi}_0))^{J_D}_+ P^{-1}M_+ \hspace{1mm}\Phi_2({\boldsymbol \varpi}_0) = C \in M(p,\bbR).
\end{equation}
We want to show that
\begin{equation}\label{e:M-*z0=M+*Phi2}
M_- \hspace{0.5mm}{\mathbf z}_0 = M_+ \Phi_2({\boldsymbol \varpi}_0) \equiv M_+ \Phi_2(s_0,{\mathbf z}_0).
\end{equation}
Without loss of generality, we can assume $J_D$ is a single Jordan block. In view of condition \eqref{e:Re_eig(H)=(0,1)-(1/2)}, it suffices to consider two cases, namely, when $J_D$ is a Jordan block associated with an eigenvalue $d$ with positive real part or with negative real part. So, first assume $\Re(d) > 0$ and rewrite \eqref{e:equiv_filters_at_t_and_-t_compPoi_general_const} as
$$
\Big(\frac{t_n+\Phi_1({\boldsymbol \varpi}_0)}{t_n-s_0}\Big)^{-J_D}P^{-1}M_- \hspace{1mm}{\mathbf z}_0  - P^{-1}M_+ \hspace{1mm}\Phi_2({\boldsymbol \varpi}_0) = (t_n + \Phi_1({\boldsymbol \varpi}_0))^{-J_D} C .
$$
If $C \neq {\boldsymbol 0}$, by taking $n \rightarrow \infty$, we arrive at a contradiction, since $\lim_{n \rightarrow \infty}\frac{t_n+\Phi_1({\boldsymbol \varpi}_0)}{t_n - s_0}=1$ and $\lim_{n \rightarrow \infty}\|(t_n - s_0)^{-J_D} C\|=\infty$. Therefore,
\begin{equation}\label{e:equiv_filters_at_t_and_-t_compPoi_general_const=0}
(-t_n + s_0)^{J_D}_-  P^{-1}M_- \hspace{1mm}{\mathbf z}_0  = (t_n+\Phi_1({\boldsymbol \varpi}_0))^{J_D}_+ P^{-1}M_+ \hspace{1mm}\Phi_2({\boldsymbol \varpi}_0).
\end{equation}
Alternatively, assume $\Re(d) < 0$. Rewrite \eqref{e:equiv_filters_at_t_and_-t_compPoi_general_const} as
$$
P^{-1}M_- \hspace{1mm}{\mathbf z}_0  - \Big(\frac{t_n+\Phi_1({\boldsymbol \varpi}_0)}{t_n - s_0}\Big)^{J_D} P^{-1}M_+ \hspace{1mm}\Phi_2({\boldsymbol \varpi}_0) = (t_n - s_0)^{-J_D} C .
$$
Again by taking $n \rightarrow \infty$, we arrive at a contradiction unless $C = {\boldsymbol 0}$. So, \eqref{e:equiv_filters_at_t_and_-t_compPoi_general_const=0} also holds. So, in any case, by taking $n \rightarrow \infty$ we conclude that \eqref{e:M-*z0=M+*Phi2} holds, as we wanted to show.

Still for $s_0 < 0$, now let $\{t_n\}_{n \in \bbN}\subseteq \mathbb Q$ be a sequence such that $t_n \downarrow -\infty$. Then, for large enough $n$,
$$
\Big\{[(-t_n + s_0)^{D}_+ ]M_+ + [ - (s_0)^{D}_-]M_- \Big\}\hspace{1mm}{\mathbf z}_0
$$
$$
= \Big\{[ - (\Phi_1({\boldsymbol \varpi}_0))^{D}_+]M_+   + [ (t_n+\Phi_1({\boldsymbol \varpi}_0))^{D}_- - (\Phi_1({\boldsymbol \varpi}_0))^{D}_- ] M_- \Big\}\Phi_2({\boldsymbol \varpi}_0).
$$
By an analogous argument to the one leading to \eqref{e:M-*z0=M+*Phi2}, we conclude that
\begin{equation}\label{e:M+*z0=M-*Phi2}
M_+ \hspace{0.5mm}{\mathbf z}_0 = M_- \Phi_2({\boldsymbol \varpi}_0) \equiv M_- \Phi_2(s_0,{\mathbf z}_0).
\end{equation}
As a consequence of \eqref{e:M-*z0=M+*Phi2} and \eqref{e:M+*z0=M-*Phi2}, for arbitrary  $(s,{\mathbf z}) \in V $ with $s<0$,
\begin{equation}\label{e:phi2_s<0}
\Phi_2(s,\mathbf z) = M_-^{-1}M_+ \mathbf z = M_+^{-1}M_- {\mathbf z}.
\end{equation}
This establishes \eqref{e:M-^(-1)*M+*z=M+^(-1)*M-*z} for all $(s,{\mathbf z}) \in V$ such that $s < 0$. Now let
\begin{equation}\label{e:w0=(s0,z0)_s0>0}
{\boldsymbol \varpi}_0 = (s_0,{\mathbf z}_0) \in V, \quad s_0 > 0.
\end{equation}
Let $\{t_n\}_{n \in \bbN} \subseteq \bbQ$ be a sequence such that $t_n \uparrow \infty$ as $n \rightarrow \infty$. By adapting the arguments for showing \eqref{e:equiv_filters_at_t_and_-t_compPoi_general_const=0}, we conclude that
\begin{equation}\label{e:(tn-s0)^DM-=(tn+Phi1)^DM+Phi2}
(t_n - s_0)^{D}_{+} M_- {\mathbf z}_0 = (t_n + \Phi_1({\boldsymbol \varpi}_0))^{D}_{+} M_+ \Phi_2({\boldsymbol \varpi}_0).
\end{equation}
Thus, by adapting the argument we conclude that relation \eqref{e:M-*z0=M+*Phi2} holds also for $s > 0$. Similarly, relation \eqref{e:M+*z0=M-*Phi2} holds for $s > 0$.

In summary, we conclude that
\begin{equation}\label{e:phi2_s_neq_0}
\Phi_2(s,\mathbf z) = M_-^{-1}M_+ \mathbf z = M_+^{-1}M_- {\mathbf z}, \quad (s,{\mathbf z}) \in V, \quad s\neq0.
\end{equation}
In other words, \eqref{e:M-^(-1)*M+*z=M+^(-1)*M-*z} holds, as we wanted to show.

Next, we show that \eqref{e:mu(matrix*dx)=mu(dz)} holds. So, fix again ${\boldsymbol \varpi}_0 = (s_0,\mathbf z_0)\in V$, $s_0 < 0$, as in \eqref{e:w0=(s0,z0)_s0<0}. Consider a sequence $\{t_n\}_{n\in\N} \subseteq \bbQ$ such that $t_n \uparrow \infty$. Then, for some fixed large $n$ (depending on $(s_0,\mathbf z_0)$), expressions \eqref{e:equiv_filters_at_t_and_-t_compPoi_general_const} (with $C = {\boldsymbol 0}$) and \eqref{e:phi2_s<0} imply that
$$
\Big(\frac{(-t_n+s_0)_-}{t_n+\Phi_1(s_0,\mathbf z_0)}\Big)^D M_- \mathbf z_0 =\Big(\frac{t_n-s_0}{t_n+\Phi_1(s_0,\mathbf z_0)}\Big)^D M_- \mathbf z_0 = M_+\Phi_2(s_0,\mathbf z_0) = M_-\mathbf z_0.
$$
In particular, $1$ is an eigenvalue of $\Big(\frac{t_n-s_0}{t_n+\Phi_1(s,\mathbf z_0)}\Big)^D$ with corresponding eigenvector $M_-\mathbf z_0$.  However, in view of condition \eqref{e:Re_eig(H)=(0,1)-(1/2)}, $\text{eig}(D)\cap \{0\} = \emptyset$. Hence,
$$
\{1\}\in \textnormal{eig}\Big( \Big(\frac{t_n-s_0}{t_n+\Phi_1(s,\mathbf z)}\Big)^D\Big) = \Big\{ w\in \C : w= \Big(\frac{t_n-s_0}{t_n+\Phi_1(s,\mathbf z)}\Big)^d, \hspace{1mm}d\in \text{eig}(D) \Big\}
$$
$$
\Leftrightarrow  \frac{t_n-s_0}{t_n+\Phi_1(s_0,\mathbf z_0)} =1.
$$
Thus,
$$
\Phi_1(s_0,\mathbf z_0) = -s_0.
$$
Now fix again ${\boldsymbol \varpi}_0 = (s_0,\mathbf z_0)\in V$, $s_0 > 0$, as in \eqref{e:w0=(s0,z0)_s0>0}, and let $\{t_n\}_{n \in \bbN}$ be a sequence such that $t_n \uparrow \infty$ as $n \rightarrow \infty$. Hence, \eqref{e:(tn-s0)^DM-=(tn+Phi1)^DM+Phi2} holds. Then, by the explicit expression \eqref{e:phi2_s_neq_0} for $\Phi_2$, we conclude that 1 is an eigenvalue of $(\frac{t_n-s_0}{t_n + \Phi_1({\boldsymbol \varpi}_0)})^D$. Thus, once again we arrive at the relation $t_n + s_0 = t_n - \Phi_1({\boldsymbol \varpi}_0)$, i.e.,
$$
\Phi_1({\boldsymbol \varpi}_0)= - s_0.
$$
Since $(s_0,\mathbf z_0)$ was arbitrary, we conclude that
$$
\Phi_1(s,\mathbf z)=-s\quad \text{for all }(s,\mathbf z) \in V.
$$
In particular,
$$
\Phi(s,\mathbf z) = (-s, M_-^{-1}M_+ \mathbf z)  \quad \kappa(ds,d{\mathbf z})\text{--a.e. }
$$
However, by Proposition \ref{p:theorem_2.17_KabluchkoStoev}, $(ii)$, the mapping $\Phi$ is a measure space isomorphism from the space $(\R\times \R^p, \mathcal{B}(\R\times \R^p), \kappa )$ to itself. In particular, expression \eqref{e:mubar1(A)=mubar2(Phi(A))} holds with $\kappa_1 = \kappa = \kappa_2$). Since, in addition, $\eta[-1,0]=1=\eta[0,1]$, then
$$
\mu(B) = \kappa([0,1]\times B) =\kappa \circ \Phi^{-1}\big([0,1]\times B)\big)= \kappa\big([-1,0]\times \Phi^{-1}_2(B)\big) =
\mu(\Phi^{-1}_2(B))
$$
for any Borel set $B\in \mathcal B(\R^p)$.
Hence, expression \eqref{e:mu(matrix*dx)=mu(dz)} holds. This shows that $(i)\Rightarrow (ii)$. Therefore, $(i)\Leftrightarrow (ii)$, as claimed.

We now show that $(a) \Leftrightarrow (a')$.  So, suppose $(a)$ in $(ii)$ holds.  Then, for each $s,t\in\R$, $M_+\mathbf z = M_-M_+^{-1}M_-\mathbf z$ $\mu(d{\mathbf z})$--a.e. Thus,
$$
g_{-t}(-s) \mathbf z =\Big\{[(-t +s)^{D}_+ - s^{D}_+]M_+ +  [(-t +s)^{D}_- - s^{D}_-]M_- \Big\}{\mathbf z}
$$
$$
= \Big\{[(-t +s)^{D}_+ - s^{D}_+]M_- +  [(-t +s)^{D}_- - s^{D}_-]M_+ \Big\}M_+^{-1}M_-{\mathbf z}
$$
$$
= \Big\{[(t -s)^{D}_- - (-s)^{D}_-]M_- +  [(t -s)^{D}_+ - (-s)^{D}_+]M_+ \Big\}M_+^{-1}M_-{\mathbf z} = g_{t}(s)M_+^{-1}M_-{\mathbf z}.
$$
Analogously, $g_{-t}(-s){\mathbf z}=g_{t}(s) M_-^{-1}M_+{\mathbf z}$ $\mu(d\mathbf z)$--a.e. Hence, $(a')$ holds. In turn, assuming $(a')$, for fixed $s$, by taking large enough $t$, by reasoning similarly to the argument leading to \eqref{e:equiv_filters_at_t_and_-t_compPoi_general_const=0} we obtain the relation
$$
\Big(\frac{t+s}{t-s}\Big)^D M_+ M_-^{-1}M_+ \mathbf z = M_-\mathbf z \quad \mu(d\mathbf z) \text{--a.e.}
$$
By taking the limit $t\to\infty$, we obtain \eqref{e:M-^(-1)*M+*z=M+^(-1)*M-*z}. Thus, $(a)$ holds. In other words, $(a) \Leftrightarrow (a')$, as claimed. $\Box$\\

\noindent {\sc  Proof of Theorem \ref{t:rhofLm_time-revers_general}}: As in the proof of Theorem \ref{t:maofLm_time-reversibility}, let $\{\widetilde{Y}_H(t)\}_{t \in \bbR} = \{\widetilde{X}_H(-t)\}_{t \in \bbR}$ be the time-reversed process. First note that, if $\{f_t({\boldsymbol \varpi}),t \in \bbR\}= \{\widetilde{g}_t(x){\mathbf z},t \in \bbR\}$ is a minimal representation of $\widetilde{X}_H$ with respect to ${\mathcal B} \hspace{-1mm}\mod \kappa$, then
\begin{equation}\label{e:f_minus_t_is_minimal_harm}
\{f_{-t}({\boldsymbol \varpi}),t \in \bbR\} \textnormal{ is a minimal representation of $\widetilde{Y}_H$ with respect to } {\mathcal B} \hspace{-1mm}\mod \kappa.
\end{equation}
First, we show the condition \eqref{e:mu(A^(-1)bar(Az))=mu(z)} implies time reversibility.  Observe that
$$
\Re (\widetilde g_t(x)\mathbf z) = \Re \Big( \frac{e^{\imag t x}-1}{\imag x} \hspace{0.5mm}\big[x^{-D}_{+}A  + x^{-D}_{-}\overline{A} \big]\mathbf z\Big)
$$
$$
=\Re \Big( \Big(\frac{e^{\imag t x}-1}{-\imag x}\Big) \hspace{0.5mm}\big[x^{-D}_{+}\overline A (-\overline A ^{-1} A ) + x^{-D}_{-}A(- A^{-1} \overline A)  \big]\mathbf z\Big)
$$
$$
=\Re \Big( \Big(\frac{e^{-\imag t x'}-1}{\imag x'}\Big) \hspace{0.5mm}\big[(x')^{-D}_{-}\overline A  + (x')^{-D}_{+}A \big]\mathbf z'\Big) =\Re (\widetilde g_{-t}(x')\mathbf z'),
$$
where $(x',\mathbf z') = (-x,   - A^{-1} \overline A \mathbf z \mathbf 1_{\{x<0\}}- \overline A^{-1} A \mathbf z\mathbf 1_{\{x>0\}})=: \Psi(x,\mathbf z)$. Then,
$$
\Re (\widetilde g_t(x)\mathbf z) = f_t(\boldsymbol\varpi) = f_{-t}(\Psi^{-1}(\boldsymbol\varpi)).
$$
For $\widetilde{\kappa}(d{\boldsymbol \varpi}) = dx \otimes \widetilde \mu(d\mathbf z)$, recall that we define $\widetilde \mu(d\mathbf z) = \frac{\mu(d\mathbf z) + \mu(\overline{d\mathbf z})}{2}$. Also observe that, by condition \eqref{e:mu(A^(-1)bar(Az))=mu(z)}, $\widetilde \mu(d\mathbf z) = \widetilde \mu (-A^{-1}\overline A d\mathbf z)$. We now show that
\begin{equation}\label{e:kappa-tilde=kappa-tilde_o_Psi^(-1)}
\widetilde \kappa = \widetilde \kappa \circ \Psi^{-1}.
\end{equation}
Indeed, if $I\in \mathcal B (\R)$, $B \in \mathcal B(\C^p)$, $I\pm = I\cap \R_{\pm}$, then
$$
\widetilde\kappa(I_+\times B) = \eta(I_+)\widetilde \mu(B) = \eta(-I_+)\widetilde \mu(- A^{-1} \overline A  B) = \widetilde\kappa \circ \Psi^{-1}(I_+\times B)
$$
and
$$
\widetilde\kappa(I_-\times B) =  \eta(I_-)\widetilde \mu(B) = \eta(-I_-)\widetilde \mu(- \overline A^{-1} A  B) = \widetilde\kappa \circ \Psi^{-1}(I_-\times B).
$$
This shows that $\widetilde\kappa(I\times B) = \widetilde\kappa \circ \Psi^{-1} (I\times B)$, which in turn implies the measures coincide on $\mathcal B(\R\times \C^p)$. This establishes \eqref{e:kappa-tilde=kappa-tilde_o_Psi^(-1)}.

Therefore, starting from \eqref{e:chf_harm}, by a change of variables,
$$
\bbE \exp \Big\{\imag \sum^{n}_{k=1}{\mathbf u}^*_k \widetilde{X}_H(t_k) \Big\}
$$
$$
= \exp\Big\{ \int_{\bbR\times \bbC^p}\Big[e^{\imag \hspace{0.25mm}2 \hspace{0.25mm}\Re \big(  \sum^{n}_{k=1}{\mathbf u}^*_k \hspace{0.5mm}f_{t_k}(\boldsymbol \varpi)\big)}- 1 -
\imag \hspace{0.25mm}2\hspace{0.25mm} \Re \Big(  \sum^{n}_{k=1}{\mathbf u}^*_k \hspace{0.5mm}f_{t_k}(\boldsymbol \varpi)\Big)\Big]\widetilde \kappa(d\boldsymbol \varpi) \Big\}
$$
$$
= \exp\Big\{ \int_{\bbR\times \bbC^p}\Big[e^{\imag \hspace{0.25mm}2 \hspace{0.25mm}\Re \big(  \sum^{n}_{k=1}{\mathbf u}^*_k \hspace{0.5mm}f_{-t_k}(\Psi^{-1}(\boldsymbol \varpi))\big)}- 1 -
\imag \hspace{0.25mm}2\hspace{0.25mm} \Re \Big(  \sum^{n}_{k=1}{\mathbf u}^*_k \hspace{0.5mm}f_{-t_k}(\Psi^{-1}(\boldsymbol \varpi))\Big)\Big]\widetilde\kappa (d\boldsymbol \varpi) \Big\}
$$
$$
= \exp\Big\{ \int_{\bbR\times \bbC^p}\Big[e^{\imag \hspace{0.25mm}2 \hspace{0.25mm}\Re \big(  \sum^{n}_{k=1}{\mathbf u}^*_k \hspace{0.5mm}f_{-t_k}(\boldsymbol \varpi)\big)}- 1 -
\imag \hspace{0.25mm}2\hspace{0.25mm} \Re \Big(  \sum^{n}_{k=1}{\mathbf u}^*_k \hspace{0.5mm}f_{-t_k}(\boldsymbol \varpi)\Big)\Big](\widetilde\kappa \circ \Psi^{-1}) (d\boldsymbol \varpi) \Big\}
$$
$$
= \exp\Big\{ \int_{\bbR\times \bbC^p}\Big[e^{\imag \hspace{0.25mm}2 \hspace{0.25mm}\Re \big(  \sum^{n}_{k=1}{\mathbf u}^*_k \hspace{0.5mm}f_{-t_k}(\boldsymbol \varpi)\big)}- 1 -
\imag \hspace{0.25mm}2\hspace{0.25mm} \Re \Big(  \sum^{n}_{k=1}{\mathbf u}^*_k \hspace{0.5mm}f_{-t_k}(\boldsymbol \varpi)\Big)\Big]\widetilde\kappa(d\boldsymbol \varpi) \Big\}
$$
$$
= \bbE \exp \Big\{\imag \sum^{n}_{k=1}{\mathbf u}^*_k \widetilde{X}_H(-t_k) \Big\}.
$$
This shows $\widetilde X_H$ is time-reversible.\\

Now suppose $\widetilde X_H$ is time-reversible, i.e., that
$$
\Big\{\int_{\bbR^{p+1}} f_{-t}({\boldsymbol \varpi})\widetilde{N}(d{\boldsymbol \varpi})\Big\}_{t \in \bbR} \stackrel{{\mathcal L}}= \Big\{\int_{\bbR^{p+1}} f_{t}({\boldsymbol \varpi})\widetilde{N}(d{\boldsymbol \varpi})\Big\}_{t \in \bbR}.
$$
So, recall that $\widetilde{\kappa}(d {\boldsymbol \varpi})\equiv  dx \otimes \widetilde{\mu}(d{\mathbf z})$. Under condition \eqref{e:gt_is_minimal_for_X-tilde_general}, Proposition \ref{p:theorem_2.17_KabluchkoStoev} implies that there exists a measurable bijection $$
\Phi: \bbR^{p+1} \rightarrow \bbR^{p+1}, \quad {\boldsymbol \varpi} \mapsto \big( \Phi_1({\boldsymbol \varpi}),\Phi_2({\boldsymbol \varpi})\big)
\in \bbR \times \bbR^p,
$$
such that, for all $t \in \bbR$,
\begin{equation}\label{e:gt(x)z=g_t(phi)phi}
\Re (\widetilde g_t(x)\mathbf z) = \Re \Big(\widetilde g_{-t}\big(\Phi_1(x,\mathbf z)\big)\Phi_2(x,\mathbf z)\Big) \quad dx\otimes \widetilde{\mu}(d\mathbf z)\text{--a.e.}
\end{equation}
Moreover, by Proposition \ref{p:theorem_2.17_KabluchkoStoev}, the mapping $\Phi$ is a measure space isomorphism from the space $(\R\times \C^p, \mathcal{B}(\R\times \C^p), \widetilde \kappa )$ to itself. In particular, it also preserves the measure $\widetilde \kappa$, the same being true of $\Phi^{-1}$ (cf.\ \eqref{e:mubar1(A)=mubar2(Phi(A))}). Noting that the set $[0,1]$ has Lebesgue measure 1,
$$
\widetilde \mu(B) = \widetilde{\kappa}([0,1] \times B) = \widetilde \kappa \circ \Phi^{-1}([0,1] \times B), \quad \widetilde{\kappa}\big([0,1] \times (- A^{-1} \overline{A  B})\big)= \widetilde \mu\big(- A^{-1} \overline{A  B}\big).
$$
Our goal is to show that
\begin{equation}\label{e:kappa-tilde_circ_Phi^(-1)}
\widetilde \kappa \circ \Phi^{-1}([0,1] \times B)=\widetilde{\kappa}\big([0,1] \times (- A^{-1} \overline{A  B})\big),
\end{equation}
whence \eqref{e:mu(A^(-1)bar(Az))=mu(z)} is established. For this purpose, we need to conveniently reexpress $\Phi$.

So, for each $t$, let
$$
\widetilde{V}_t = \{ {\boldsymbol \varpi}: \eqref{e:gt(x)z=g_t(phi)phi} \textnormal{ holds at }{\boldsymbol \varpi} \} \in {\mathcal B}(\R \times \R^p),
$$
and set $\widetilde{V} = \bigcap_{t\in \mathbb Q} \widetilde{V}_t$. In particular, $\widetilde{\kappa}(\widetilde{V}_t^c)=0$. Fix a vector $(x_0,\mathbf z_0)\in \widetilde{V}$ with
\begin{equation}\label{e:x0>0}
x_0>0,
\end{equation}
and note that $A \mathbf z_0 \neq 0$ as a consequence of condition \eqref{e:A_invertible_mu(dz)_ae}. For notational simplicity, write
$$
\Phi_1=\Phi_1(x_0,\mathbf z_0), \quad \Phi_2=\Phi_2(x_0,\mathbf z_0).
$$
By \eqref{e:gt(x)z=g_t(phi)phi}, for all $t \in \mathbb Q$,
\begin{equation}\label{e:Re(gt(x))=Re(gt(phi))_x0z0}
 \Re \Big( \frac{e^{\imag t x_0}-1}{\imag x_0} \hspace{0.5mm}(x_0)^{-D}_+ A  \mathbf z_0\Big)
 =  \Re \Big( \frac{e^{-\imag t \Phi_1}-1}{\imag \Phi_1} \hspace{0.5mm}\big[(\Phi_1)^{-D}_{+}A  + (\Phi_1)^{-D}_{-}\overline{A} \big] \Phi_2\Big).
\end{equation}
By continuity in $t$, relation \eqref{e:Re(gt(x))=Re(gt(phi))_x0z0} holds for all $t\in \R$.  Taking derivatives in $t$, we find that for all integers $k\geq 0$,
\begin{equation}\label{e:Re(gt(x))=Re(gt(phi))_x0z0_derivs}
\Re \big((\imag x)^k e^{\imag t x_0}(x_0)^{-D}_{+}A  \mathbf z_0\big)=\Re \Big( -(-\imag \Phi_1)^k e^{-\imag t \Phi_1}\hspace{0.5mm}\big[(\Phi_1)^{-D}_{+}A  + (\Phi_1)^{-D}_{-}\overline{A} \big] \Phi_2\Big), \quad t \in \R.
\end{equation}
By taking $k=4$ and $k=0$, respectively, we arrive at the equalities
$$
x_0^4 \hspace{1mm}\Re \big( e^{\imag t x_0}(x_0)^{-D}_{+}A  \mathbf z_0\big) = \Phi_1^4 \hspace{1mm}\Re \Big( - e^{-\imag t \Phi_1}\hspace{0.5mm}\big[(\Phi_1)^{-D}_{+}A  + (\Phi_1)^{-D}_{-}\overline{A} \big] \Phi_2\Big)
$$
$$
=  \Phi_1^4 \hspace{1mm}\Re \big( e^{\imag t x}(x_0)^{-D}_{+}A  {\mathbf z}_{0}\big).
$$
In particular,
\begin{equation}\label{e:Phi_1=x0_or_Phi_1=-x0}
\textnormal{ either }\Phi_1= x_0 \textnormal{ or }\Phi_1=-x_0.
\end{equation}
So, in view of \eqref{e:Phi_1=x0_or_Phi_1=-x0}, first suppose $\Phi_1=x_0 > 0$. Then, from \eqref{e:Re(gt(x))=Re(gt(phi))_x0z0_derivs} with $k=0$, and by the invertibility of the matrix $(x_0)_+^{-D}$,
\begin{equation}\label{e:Re(e^(itx0)Az0)}
\Re \big( e^{\imag t x_0}A  \mathbf z_0\big)=\Re \big( - e^{-\imag t x}\hspace{0.5mm}A \Phi_2 \big),\quad t\in \R.
\end{equation}
Hence, by choosing $t =0$ in \eqref{e:Re(e^(itx0)Az0)}, we see that $\Re (A  \mathbf z_0)= \Re (-A \Phi_2)$. On the other hand, by choosing $t$ such that $e^{\imag t x_0}=\imag$ in \eqref{e:Re(e^(itx0)Az0)}, we see that $\Im (A  \mathbf z_0)= -\Im (-A \Phi_2)$. This implies that $A  \mathbf z_0 = - \overline{A \Phi_2}$. Hence,
\begin{equation}\label{e:Az_if_phi1_x}
\Phi_2(x_0,\mathbf z_0)=- A^{-1}\overline{A  \mathbf z_0},\quad \text{if }\Phi_1(x_0,\mathbf z_0)=x_0.
\end{equation}
Alternatively, suppose $\Phi_1=-x_0$ in \eqref{e:Phi_1=x0_or_Phi_1=-x0}. From \eqref{e:Re(gt(x))=Re(gt(phi))_x0z0_derivs} with $k=0$, we obtain
$$
\Re \big( e^{\imag t x_0}A  \mathbf z_0\big)=\Re \Big( - e^{\imag t x_0}\hspace{0.5mm}\overline A \Phi_2\Big),\quad t\in \R.
$$
As with relation \eqref{e:Re(e^(itx0)Az0)}, this implies that $\Re (A  \mathbf z_0) = \Re(-\overline A \Phi_2)$, $\Im (A  \mathbf z_0) = \Im(-\overline A \Phi_2)$ by choosing $t$ appropriately. In other words, $A  \mathbf z_0 =  -\overline A \Phi_2$.  Hence,
\begin{equation}\label{e:Az_if_phi1_-x}
\Phi_2(x_0,\mathbf z_0)= -\overline A^{-1}A  \mathbf z_0,\quad \text{if }\Phi_1(x_0,\mathbf z_0)=-x_0.
\end{equation}

Consequently, for each $(x,\mathbf z) \in \widetilde{V}$ with $x>0$ (namely, under condition \eqref{e:x0>0}), either \eqref{e:Az_if_phi1_x} or \eqref{e:Az_if_phi1_-x} holds.  An analogous argument shows that $|\Phi_1(x,\mathbf z)|=|x|$ for $x<0$. It also shows that, for each fixed $(x_0,\mathbf z_0)$ with $x_0<0$,
\begin{equation}\label{e:Phi2(x0,z0)_x0<0}
\Phi_2(x_0,\mathbf z_0)=\begin{cases}
- \overline A^{-1} A \overline{\mathbf z},& \text{if } \Phi_1(x_0,\mathbf z_0)=x_0;\\
-  A^{-1}{\overline A  \mathbf z},& \text{if }\Phi_1(x_0,\mathbf z_0)=-x_0,
\end{cases}
\qquad\text{when }x_0<0.
\end{equation}

So, define the sets $\widetilde{V}_{\pm}=\{(x,\mathbf z) \in \widetilde{V} :\Phi_1(x,\mathbf z) = \pm x\}$. In view of \eqref{e:Az_if_phi1_x}, \eqref{e:Az_if_phi1_-x} and \eqref{e:Phi2(x0,z0)_x0<0}, $\widetilde{V}_-$ and $\widetilde{V}_+$ partition $\widetilde{V}$. Now define the functions
$$
\Psi(s,\mathbf z) = \big(x,- \overline A^{-1}A  \overline{\mathbf z}\mathbf 1_{\{x<0\}} - A^{-1}\overline{A  \mathbf z}\mathbf 1_{\{x>0\}}\big),\quad \gamma(x,\mathbf z) = (-x,\overline{\mathbf z}).%
$$
Then, we can write
\begin{equation}\label{e:Phi(x,z)_decomp}
\Phi(x,\mathbf z) = \Psi(x,\mathbf z)\mathbf 1_{\widetilde{V}_+}(s,\mathbf z)+(\Psi\circ \gamma) (x,\mathbf z)\mathbf 1_{\widetilde{V}_-}(s,\mathbf z).
\end{equation}
Observe that $\Psi,\gamma$ are bijective, and also that
\begin{equation}\label{e:kappa-tilde_circ_gamma^(-1)}
\widetilde \kappa \circ \gamma^{-1}=\widetilde \kappa.
\end{equation}
So, consider any set $B \in {\mathcal B}(\R^p)$. For notational simplicity, write $[0,1]\times B= S_+ \cup S_-$, where $S_\pm = ([0,1]\times B)\cap \widetilde{V}_{\pm}$. Since $S_-$ and $S_+$ are disjoint, then
\begin{equation}\label{e:kappa-tilde_circ_Phi^(-1)(S+_U_S-)}
\widetilde \kappa \circ \Phi^{-1}(S_+ \cup S_-) = \widetilde \kappa \circ \Phi^{-1}(S_+) + \widetilde \kappa \circ \Phi^{-1}(S_-).
\end{equation}
Therefore, by relations \eqref{e:Phi(x,z)_decomp}, \eqref{e:kappa-tilde_circ_gamma^(-1)} and \eqref{e:kappa-tilde_circ_Phi^(-1)(S+_U_S-)},
$$
\widetilde \kappa([0,1]\times B) = \widetilde \kappa \circ \Phi^{-1}([0,1]\times B)=\widetilde \kappa \circ \Phi^{-1}(S_+ \cup S_-)
=\widetilde \kappa \circ \Psi^{-1}(S_+) + \widetilde \kappa\circ \gamma^{-1}\circ \Psi^{-1}(S_-)
$$
$$
= \widetilde \kappa \circ  \Psi^{-1}(S_+)+ \widetilde \kappa \circ \Psi^{-1}(S_-)
= \widetilde \kappa\circ \Psi^{-1}(B) =  \widetilde \kappa\big([0,1]\times(- A^{-1} \overline{A  B})\big).
$$
This shows \eqref{e:kappa-tilde_circ_Phi^(-1)}. Therefore, \eqref{e:mu(A^(-1)bar(Az))=mu(z)} holds. $\Box$\\

\section{On the uniqueness of stochastic integral representations}\label{s:uniqueness}

To characterize the uniqueness of stochastic integral representations, we first recap the concept of isomorphism between measurable spaces, as well as related notions.
\begin{definition} \label{def:isomorphism_Borel_space}Consider the following definitions.
\begin{itemize}
\item [$(i)$] An \textit{isomorphism between two measurable spaces} $(\overline{\Omega}_i,{\mathcal B}_i)$, $i = 1,2$, is a bijection $\Phi: \overline{\Omega}_1 \rightarrow \overline{\Omega}_2$ such that both $\Phi$ and $\Phi^{-1}$ are measurable.
\item [$(ii)$] A measurable space $(\overline{\Omega},{\mathcal B})$ is called a \textit{Borel space }if it is isomorphic (in the sense of $(i)$) to a complete separable metric space endowed with its Borel $\sigma$-algebra.
\item [$(iii)$] A Borel space endowed with a $\sigma$-finite measure is called a \textit{$\sigma$-finite Borel space}.
\item [$(iv)$] An \textit{isomorphism modulo null sets between two measure spaces} $(\overline{\Omega}_i,{\mathcal B}_i,\kappa_i)$, $i = 1,2$, is a bijection $\Phi:\overline{\Omega}_1 \backslash A_1 \rightarrow \overline{\Omega}_2 \backslash A_2$, where $A_1 \in {\mathcal B}_1$ and $A_2 \in {\mathcal B}_2$ are null sets, such that both $\Phi$ and $\Phi^{-1}$ are measurable and
    \begin{equation}\label{e:mubar1(A)=mubar2(Phi(A))}
    \kappa_1(A) = \kappa_2(\Phi(A))
    \end{equation}
    for all measurable $A \subseteq \overline{\Omega}_1 \backslash A_1$. Two isomorphisms $\Phi$, $\Psi$ are considered \textit{equal modulo null sets} if $\Phi({\boldsymbol \varpi})=\Psi({\boldsymbol \varpi})$ for $\kappa_1$-a.a.\ ${\boldsymbol \varpi} \in \overline{\Omega}_1$.
\end{itemize}
\end{definition}

In the following proposition, we establish the uniqueness of finite second moment stochastic integral representations of processes based on compensated Poisson random measures. Its proof is omitted since it is similar to that of Theorem 2.17 in Kabluchko and Stoev \cite{kabluchko:stoev:2016}, originally involving univariate integrals.
\begin{proposition}\label{p:theorem_2.17_KabluchkoStoev}
 For $\emptyset \neq T \subseteq \bbR$, let $X = \{X(t)\}_{t \in T}$ be an ID stochastic process with stochastic representation of the form \eqref{e:X(t)=int_ft(omega)M(domega)}. For $i=1,2$, let
\begin{equation}\label{e:f^(i)(t)_in_L2}
\{f^{(i)}_t\}_{t \in T} \subseteq L^2(\overline{\Omega}_i,{\mathcal B}_i,\kappa_i)
\end{equation}
be two minimal representations of $X$, where
\begin{equation}\label{e:Omega-bar}
\overline{\Omega}_i = \R \times \bbR^{q} \in {\mathcal B}(\bbR \times \bbR^{q}), \quad {\mathcal B}_i = {\mathcal B}(\overline{\Omega}_i).
\end{equation}
\begin{itemize}
\item [$(i)$] If $(\overline{\Omega}_1,{\mathcal B}_1, \kappa_1)$ is a ($\sigma$-finite) Borel space, then there is a measurable map $\Phi:\overline{\Omega}_2 \rightarrow \overline{\Omega}_1$ such that $\kappa_1 = \kappa_2 \circ \Phi^{-1}$ and, for all $t \in T$,
    $$
    f^{(2)}_t({\boldsymbol \varpi}) = f^{(1)}_t [\Phi({\boldsymbol \varpi})] \quad \textnormal{for $\kappa_2$-a.a.\ ${\boldsymbol \varpi} \in \overline{\Omega}_2$}.
    $$
\item [$(ii)$] If both $(\overline{\Omega}_i,{\mathcal B}_i, \kappa_i)$, $i = 1,2$, are ($\sigma$-finite) Borel spaces, then the mapping $\Phi$ in $(i)$ is a measure space isomorphism and it is unique modulo null sets.
\end{itemize}
\end{proposition}

\section{Auxiliary results}

As pointed out in Remark \ref{r:YHtilde_proper}, there are instances of ofLm -- hence, stochastic processes whose distributions are proper except at the origin -- whose stochastic integral representations have deficient rank over a positive measure set. In fact, for $x \neq 0$, define the integrand
$$%
L^2(\bbR,M(p,\bbR)) \ni \widetilde{h}_t(x) = \Re \widetilde{g}_t(x) - \Im \widetilde{g}_t(x)
$$%
$$
=\Big\{\frac{\sin(tx)}{x}-\frac{(1-\cos(tx))}{x}\Big\} |x|^{-D} \Re(A)
$$
\begin{equation}\label{e:ht=Regt-Imgt_def}
-  \Big\{\frac{\sin(tx)}{x}+ \frac{(1-\cos(tx))}{x}\Big\}\big\{x^{-D}_{+}-x^{-D}_{-}\big\}\Im(A), \quad t \in \bbR.
\end{equation}
Also, in \eqref{e:chf_harm_gt}, suppose $\mu_{\bbR^{2p}}(d{\mathbf z}) = \delta_{{\mathbf 1}}(d{\mathbf z})$, where ${\mathbf 1} = (1,\hdots,1)^* \in \bbR^{2p}$. Then, for ${\mathbf 1} \in \bbR^{p}$ and by relation \eqref{e:E_int_isometry}, we can express the rhofLm $\widetilde{X}_H$ as
\begin{equation}\label{e:EYH(t)YH(T)*_pos_def_delta1}
\bbE \widetilde{X}_H(t)\widetilde{X}_H(t)^* =  \int_{\bbR} \widetilde{h}_t(x){\mathbf 1}{\mathbf 1}^* \widetilde{h}_t(x)^* d x, \quad t \neq 0.
\end{equation}
Note that the integrand on the right-hand side of \eqref{e:EYH(t)YH(T)*_pos_def_delta1} has rank 1 a.e. However, the properness condition \eqref{e:EX2-tildeH(t)>0} may still be satisfied, as we show next.
\begin{lemma}\label{l:Ytilde_proper}
For $p = 2$, let $A = \Re(A) + \imag \Im(A)\in M(2,\bbC)$ be such that $\Re(A){\mathbf 1} = {\mathbf 0}$ and $\Im(A) = I$. Let $\widetilde{X}_H = \{\widetilde{X}_H(t)\}_{t \in \bbR}$ be as in \eqref{e:EYH(t)YH(T)*_pos_def_delta1}. Then, there exist $- \frac{1}{2} < d_1,d_2 < \frac{1}{2}$ such that, for $D = \textnormal{diag}(d_1,d_2) = H - (1/2)I$, the matrix $\bbE \widetilde{X}_H(t)\widetilde{X}_H(t)^*$ has full rank matrix for any $t \neq 0$.
\end{lemma}
\begin{proof}Fix $t \neq 0$. Let
\begin{equation}
M:= \Im(A){\mathbf 1}{\mathbf 1}^* \Im(A)^*=
\left(\begin{array}{cc}
1 & 1\\
1 & 1
\end{array}\right).
\end{equation}
For $\widetilde{h}_t$ as in \eqref{e:ht=Regt-Imgt_def}, expression \eqref{e:EYH(t)YH(T)*_pos_def_delta1} implies that
$$
\bbE \widetilde{X}_H(t)\widetilde{X}_H(t)^* %
= \int_{\bbR} \Big\{  \Big(\frac{\sin(t x)}{x}\Big)^2 + \Big( \frac{1- \cos(t x)}{x}\Big)^2 \Big\} |x|^{-D} M |x|^{-D^*} d x
$$
$$
= \int_{\bbR} 2\frac{(1- \cos(t x))}{x^2}  |x|^{-D} M |x|^{-D^*} d x,
$$
where we make use of the fact that $\frac{\sin(t x)(1- \cos(t x))}{x^2}$ is an odd function. For $- 1 < \delta < 1$, we can write
$$
\bbR \ni \int_{\bbR}  2\frac{(1- \cos(t x))}{x^2}  |x|^{-\delta} d x
= |t|^{1+\delta}\int_{\bbR}  2\frac{(1- \cos(y))}{y^2}  |y|^{-\delta} d y =: |t|^{1+\delta}\beta(\delta),
$$
where we make the change of variable $y = tx$. Then,
$$
\det\big(\bbE \widetilde{X}_H(t)\widetilde{X}_H(t)^*\big)  = |t|^{2+2(d_1+d_2)}\big(\beta(2d_1)\beta(2d_2) - \beta^2(d_1 + d_2)\big).
$$
In particular, for all $t \neq 0$, $\det\big(\bbE \widetilde{X}_H(t)\widetilde{X}_H(t)^*\big) = 0$  if and only if
$$
\beta(2d_1)\beta(2d_2) - \beta^2(d_1 + d_2) = 0.
$$
However, the latter condition cannot hold for every $- \frac{1}{2} < d_1, d_2 < \frac{1}{2}$ (cf.\ Remark 4.1 in Didier and Pipiras \cite{didier:pipiras:2011}).  This establishes the claim. $\Box$\\
\end{proof}

The following proposition is used in the proof of Proposition \ref{p:XH_harm_YH_ma}. It establishes orthogonal-increment random measures that can be used to yield integral representations for maofLm and rhofLm (cf.\ Rozanov \cite{rozanov:1967},  \textsection1.3).
\begin{proposition}\label{p:spec}
\begin{itemize}
\item [(i)] Let ${\mathcal M}(ds)$ be the random measure \eqref{e:int_f(s)M(ds)} under the assumption \eqref{e:int_zz*=I}. Then, the random measure
$$
\Phi_{{\mathcal M}} (a,b]:=  \frac{1}{2\pi}\int_\R \frac{e^{\imag sb}-e^{\imag sa}}{ \imag s}{\mathcal M}(ds), \quad -\infty < a \leq b < \infty,
$$
defined on intervals, naturally extends to a $\C^p$-valued orthogonal-increment random measure on $\mathcal{B}(\R)$. In particular, \begin{equation}\label{e:EPhi_M(B)=0}
\E \Phi_{\mathcal M} (B) =0, \quad \textnormal{for all }B\in\mathcal{B}(\R) \textnormal{ with }\eta(B)<\infty,
\end{equation}
and
\begin{equation}\label{e:EPhi_M(0,1]^2}
\bbE \Phi_{\mathcal M}(dy) \Phi_{\mathcal M}(dy)^* = dy \times I.
\end{equation}
Furthermore, for $f\in L^2(\R;M(p,\R))$,
\begin{equation}\label{e:parseval1}
\int_\R f(s) {\mathcal M}(ds) = \int_\R \widehat{f}(x) \Phi_{{\mathcal M}}(dx) \quad \text{a.s.},
\end{equation}
where, in the Parseval-type relation \eqref{e:parseval1} we define, entry-wise,
\begin{equation}\label{e:f-hat=F(f)}
\widehat{f} = {\mathcal F}(f).
\end{equation}
\item [(ii)] Let $\widetilde{{\mathcal M}}(dx)$ be the random measure \eqref{e:int_f(xi)M(dxi)} under the assumption \eqref{e:int_RezRez*=sigma^2_I=int_ImzImz*}. Then, the expression
\begin{equation}\label{e:Phi_M-tilde(a,b]}
\Phi_{\widetilde{{\mathcal M}}}(a,b]:= \frac{1}{2\pi } \int_\R \frac{e^{-\imag xa}-e^{-\imag xb}}{\imag x} \widetilde{{\mathcal M}}(dx),
\end{equation}
defined on intervals, naturally extends to an $\R^p$-valued orthogonal-increment random measure on $\mathcal{B}(\R)$.  In particular, relations \eqref{e:EPhi_M(B)=0} and \eqref{e:EPhi_M(0,1]^2} also hold for $\Phi_{\widetilde{{\mathcal M}}}$. Furthermore, for $\widehat{f}\in L^2_{\textnormal{Herm}}(\R)$,
\begin{equation}\label{e:parseval2}
\int_\R \widehat f(x) \widetilde{{\mathcal M}}(dx) = \int_\R f(s) \Phi_{\widetilde{{\mathcal M}}}(ds) \quad \text{a.s.},
\end{equation}
where $f$ satisfies \eqref{e:f-hat=F(f)}.
\end{itemize}
\end{proposition}
\begin{proof}
Statement $(i)$ can be shown by means of a direct adaptation of the statement of Theorem 3.5 in Marquardt and Stelzer \cite{marquardt:stelzer:2007}, which in turn is based on a multivariate generalization of Rozanov \cite{rozanov:1967},  Theorem 2.1. So, we only show \eqref{e:EPhi_M(B)=0} and \eqref{e:EPhi_M(0,1]^2}. It suffices to establish the statement over intervals $(a,b]$, $- \infty < a \leq b < \infty$. Note that the characteristic function of $\Phi_{\widetilde{{\mathcal M}}}(a,b]$ at ${\mathbf u } \in \bbR^p$ is given by
$$
\exp\Big\{ \int_{\bbR}\int_{\bbR^p} \Big( e^{\imag {\mathbf u}^* \frac{1}{2\pi}\big(\frac{e^{\imag s b}-e^{\imag s a}}{\imag s}\big){\mathbf z}}
- 1 - \imag {\mathbf u}^* \frac{1}{2\pi}\Big(\frac{e^{\imag s b}-e^{\imag s a}}{\imag s}\Big){\mathbf z}\Big) \mu(d{\mathbf z})ds\Big\}.
$$
By taking the first derivative with respect to ${\mathbf u}$ and setting ${\mathbf u} = {\mathbf 0}$, we conclude that $\bbE \Phi_{\widetilde{{\mathcal M}}}(a,b] = {\mathbf 0}$, which proves \eqref{e:EPhi_M(B)=0}. Moreover, by taking the second derivative with respect to ${\mathbf u}$ and setting ${\mathbf u} = {\mathbf 0}$, we obtain
$$
\bbE \Phi_{\widetilde{{\mathcal M}}}(a,b]\Phi_{\widetilde{{\mathcal M}}}(a,b]^* = \frac{1}{4\pi^2}\Big(\int_{\bbR} \Big|\frac{e^{\imag s b}-e^{\imag s a}}{\imag s}\Big|^2 ds \Big)
\Big(\int_{\bbR^p} {\mathbf z}{\mathbf z}^* \mu(d{\mathbf z})\Big).
$$
Thus, by setting $a = 0$ and $b = y$ under condition \eqref{e:int_zz*=I}, we further conclude that \eqref{e:EPhi_M(0,1]^2} holds, where we make use of the fact that
$\frac{1}{4 \pi^2}\int_{\bbR} \big| \frac{e^{\imag x y}-1}{\imag x}\big|^2 dx = y$ (see Taqqu \cite{taqqu:2003}, expression (9.7)). Likewise, the orthogonality of the increments can be verified by Parseval's theorem.

Statement $(ii)$ can be established by a similar argument. In particular, the random measure \eqref{e:Phi_M-tilde(a,b]} is $\bbR^p$-valued because the integrand $\frac{e^{-\imag xa}-e^{-\imag xb}}{\imag x}$ is a Hermitian function. $\Box$\\
\end{proof}

\begin{remark}
The integrals on the right-hand side of \eqref{e:parseval1} and \eqref{e:parseval2} with respect to the orthogonal-increment random measures $\Phi_{\widetilde{{\mathcal M}}}$, $\Phi_{{{\mathcal M}}}$ are interpreted in the traditional $L^2$ sense. Note that despite that $\Phi_{\widetilde{{\mathcal M}}}$, $\Phi_{{{\mathcal M}}}$ have uncorrelated increments,  $\widetilde {\mathcal M} \neq \Phi_{{{\mathcal M}}}$, and $ {\mathcal M} \neq \Phi_{{\widetilde{\mathcal M}}}$.
\end{remark}

The following lemma is used in Example \ref{ex:gaussian_jumps_symmetry}.
\begin{lemma}\label{l:gaussian_jumps_symmetry}
Under the assumptions laid out in Example \ref{ex:gaussian_jumps_symmetry}, expression \eqref{e:symmetric_orthog_condition} holds.
\end{lemma}
\begin{proof}
Observe that $\Sigma=\int_{\R^p} \mathbf z \mathbf z^* \mu\big((M_-^{-1}M_+)d\mathbf z\big) = (M_+^{-1}M_-)\Sigma (M_+^{-1}M_-)^*$, by \eqref{e:mu(matrix*dx)=mu(dz)}. If $\Sigma^{1/2}$ is the unique symmetric positive definite square root of $\Sigma$, it follows that
$I = \{\Sigma^{-1/2}(M_+^{-1}M_-)\Sigma^{1/2}\}\{ \Sigma^{1/2} (M_+^{-1}M_-)^*\Sigma^{-1/2}\}$. Therefore, $\Sigma^{-1/2}(M^{-1}_- M_+)\Sigma^{1/2} =O$ for some orthogonal matrix $O\in M(p,\R)$, showing $M_-^{-1} M_+ = \Sigma^{1/2} O\Sigma^{-1/2}$.  Moreover, under condition \eqref{e:M-^(-1)*M+*z=M+^(-1)*M-*z}, $I=(M_-^{-1} M_+)^2=\Sigma^{1/2} O^2\Sigma^{-1/2}$. This shows that $O^2=I$, i.e., $O=O^*$.  Hence, \eqref{e:symmetric_orthog_condition} holds, as claimed. $\Box$\\
\end{proof}

The following lemma is used in Section \ref{s:time_revers}.
\begin{lemma}\label{l:rhofLm_change_Levy_repres}
The rhofLm $\widetilde{X}_H$ as in \eqref{r:rhofLm_change_Levy_repres} can also be represented based on \eqref{e:kappa-tilde=dx_x_mu(dz)} and \eqref{e:mu-tilde}.
\end{lemma}
\begin{proof}
This is a consequence of the fact that, for any $m$ and any $t_1,\hdots,t_m$,
$$
\int_{\bbR} \int_{\bbC^p} \Big(e^{\imag \hspace{0.5mm}2 \sum^{m}_{j=1}{\mathbf u}^*_j \Re (\widetilde{g}_{t_j}(x){\mathbf z})}-1 -\imag \hspace{0.5mm}2 \sum^{m}_{j=1}{\mathbf u}^*_j \Re (\widetilde{g}_{t_j}(x){\mathbf z})\Big)\mu(d{\mathbf z}) dx
$$
$$
=\int_{\bbR} \int_{\bbC^p} \Big(e^{\imag \hspace{0.5mm}2 \sum^{m}_{j=1}{\mathbf u}^*_j \Re (\widetilde{g}_{t_j}(x'){\mathbf z}')}-1 -\imag \hspace{0.5mm}2 \sum^{m}_{j=1}{\mathbf u}^*_j \Re (\widetilde{g}_{t_j}(x'){\mathbf z}')\Big)\mu(\overline{d{\mathbf z}'}) dx,
$$
where we make the change of variables $(x',\mathbf z')=(-x,\overline{\mathbf z})$. $\Box$\\
\end{proof}

\section{Tempered operator-stable L\'{e}vy measures}\label{s:tempered_oper-stable_Levy_measures}
 More specifically, the framework of tempered operator-stable L\'{e}vy measures is used in Proposition \ref{p:lass2}.  To recap it, we first recall the notion of operator-stable distributions (see, for instance, Meerschaert and Scheffler \cite{meerschaert:scheffler:2001}).  Let $B\in M(p,\R)$, $\text{eig}\hspace{0.5mm}B\subseteq \{z\in \C: \Re z \in(1/2,\infty)\}$.  Consider a norm $\| \cdot \|_B$ on $\R^q$ with unit sphere $S_0:=\{x:\|x\|_B=1\}$ satisfying
\begin{itemize}
\item[$(i)$] For each $x\in \R^q\setminus\{0\}$, $r\mapsto \| r^B x\|$ is monotonically increasing for $r>0$;
\item[$(ii)$] $(r,\boldsymbol\theta)\mapsto r^B\boldsymbol\theta$ from $\R_+\times S_B$ to $\R^q\setminus\{0\}$ is a homeomorphism
\end{itemize}
(see Lemma 6.1.5 in Meerschaert and Scheffler \cite{meerschaert:scheffler:2001}). A \textit{full operator-stable distribution} (recall that \emph{full} means a distribution is not supported on any proper subspace of $\R^p$) has L\'evy measure $\mu$ that can be written as
\begin{equation}\label{e:def_opstab}
\mu_B(A) = \int_{S_0}\int_{\R_+} 1_A(r^B\boldsymbol\theta)\frac{dr}{r^2}\lambda(d\boldsymbol\theta),
\end{equation}
where $\lambda$ is a finite Borel measure on $S_0$ (see Theorem 7.2.5 in Meerschaert and Scheffler \cite{meerschaert:scheffler:2001}).  Provided $\Re \lambda_q(B)<1$, it can be shown that
$$
\int_{\|\mathbf z\|\geq 1} \|\mathbf z\| \mu_B(d\mathbf z)<\infty
$$
(see Corollary 8.2.6 in Meerschaert and Scheffler \cite{meerschaert:scheffler:2001} and Sato \cite{sato:1999}, Theorem 25.3).  In particular, under the assumption $\Re\lambda_q(B)<1$, the integral
\begin{equation}\label{e:Levy_symbol_finite_first_moment}
\psi_B({\mathbf u}) = \int_{\bbR^q} (e^{\imag \langle{\mathbf u}, {\mathbf z}\rangle}-1-\imag \langle{\mathbf u}, {\mathbf z}\rangle )\hspace{0.5mm}\mu_B(d {\mathbf z}),\quad {\mathbf u} \in \bbR^q,
\end{equation}
(c.f.\ \eqref{e:Levy_symbol_finite_2nd_moment}) exists, and the L\'{e}vy symbol $\psi_B$ satisfies
$
\psi(c^B\mathbf u)  = c\psi(\mathbf u),
$
i.e., the function $e^{\psi(\mathbf u)}$ is the characteristic function of a \emph{strictly} operator-stable distribution $\nu_B$ (see Kremer and Scheffler \cite{kremer:scheffler:2019}, p.\ 4085).  Based on $\nu_B$, we define the random measures $L_B$ used in Proposition  \ref{p:lass2} as $\R^p$-valued (in \eqref{e:opstab_MA})  or $\C^p$-valued (in \eqref{e:opstab_harm}) infinitely divisible independently scattered random measures on  $(\R,\mathcal B(\R))$ generated by $\nu_B$ and the Lebesgue measure on $\R$ in the sense of Example 3.7(a) and Remark 3.8 in Kremer and Scheffler \cite{kremer:scheffler:2017}.

For the purposes of retaining second moments and constructing asymptotically operator self-similar instances of ofLm, we consider \textit{tempered }counterparts of \eqref{e:def_opstab}, where the associated L\'evy measure is given by
\begin{equation}\label{e:def_TalphaS}
\mu_{B,q}(A) =  \int_{S_0}\int_{\R_+} 1_A(r^B\boldsymbol\theta)q(r,\boldsymbol\theta)\frac{dr}{r^2}\lambda(d\boldsymbol\theta), \qquad r\in(0,\infty),\quad \boldsymbol\theta\in S_0.
\end{equation}
In \eqref{e:def_TalphaS},  $q:(0,\infty)\times S_0\to [0,1]$ is any Borel measurable function such that, for $\lambda(d\boldsymbol \theta)$-a.e. $\boldsymbol \theta \in S_0$,
\begin{equation}\label{e:q(.,theta)->0}
q(\cdot,\boldsymbol\theta) \textnormal{ decays to 0}
\end{equation}
sufficiently fast as $r\to \infty$ to guarantee that $\mu_{B,q}$ has second moments and
\begin{equation}\label{e:q(0+,theta)=1}
q(0^+,\boldsymbol \theta)=1
\end{equation}
for each $\boldsymbol \theta\in S_0$ (for instance, $q(r,\boldsymbol\theta)=\mathbf 1_{\{|r|\leq 1\}}, \boldsymbol \theta \in S_0$). When $q$ is a completely monotone function (namely, $(-1)^k \frac{d^k}{dt^k}q(t,\boldsymbol\theta)>0$ for all $t>0$ and each $k\geq 0$), L\'evy measures defined by \eqref{e:def_TalphaS} are called \emph{tempered operator-stable L\'evy measures}, studied in  Ali \cite{ali:2014} (see also the seminal work of Rosi\'{n}ski \cite{rosinski:2007} for the tempered stable case). For conditions for the existence of moments in the tempered stable case, see Rosi\'{n}ski \cite{rosinski:2007}, Proposition 2.7.  For the tempered operator-stable case, see Ali \cite{ali:2014}, \textit{Korollar} 3.2.5.

\bibliographystyle{agsm}

\bibliography{ofLm_main}

\end{document}